\documentclass[a4paper, 12pt, oneside]{article}
\pdfoutput=1 
\usepackage[T1]{fontenc}
\usepackage[british]{babel}
\usepackage{etoolbox}
\usepackage{amsmath, amsthm, amssymb, bm, mleftright}
\usepackage{booktabs, longtable, caption, subcaption, multirow}
\usepackage[space]{grffile}
\usepackage[margin=2.5cm, footskip=1cm]{geometry}
\usepackage{enumitem}
\setenumerate[1]{label=\textup{(\arabic*)}, ref=\arabic*}
\setenumerate[2]{label=\textup{(\alph*)}, ref=\alph*}
\setenumerate[3]{label=\textup{(\roman*)}, ref=\roman*}
\usepackage[protrusion=allmath]{microtype}
\usepackage{tikz}
\usetikzlibrary{cd, calc}
\usepackage{listings}
\usepackage{mathtools}
\usepackage{mathrsfs}
\usepackage[colorlinks, citecolor=blue, linkcolor=blue, linktoc=section]{hyperref}

\allowdisplaybreaks
\usepackage[capitalise, noabbrev]{cleveref}
\creflabelformat{enumi}{(#2#1#3)}
\creflabelformat{enumii}{(#2#1#3)}
\creflabelformat{enumiii}{(#2#1#3)}
\Crefname{condition}{Condition}{Conditions}
\Crefname{question}{Question}{Questions}

\usepackage{titlesec}
\titleformat*{\section}{\large\bfseries}
\titleformat*{\subsection}{\normalsize\bfseries}
\newlength{\VerticalSpaceAfterParagraph}
\titlespacing*{\paragraph}{0pt}{\VerticalSpaceAfterParagraph}{1em}

\usepackage{titling}
\pretitle{\vspace{-\baselineskip}\begin{center}\Large\bfseries}
\posttitle{\end{center}\vspace{-0.25\baselineskip}}
\preauthor{\begin{center}}
\postauthor{\end{center}}
\predate{\begin{center}}
\postdate{\end{center}}

\setlist
  {
    topsep = 5.0pt plus 2.0pt minus 3.0pt,
    partopsep = 1.5pt plus 1.0pt minus 1.0pt,
    parsep = 2.5pt plus 1.25pt minus 0.5pt,
    itemsep = 0pt plus 1.25pt minus 0.5pt
  }

\theoremstyle{plain}
\newtheorem{theorem}{Theorem}
\newtheorem{proposition}[theorem]{Proposition}
\newtheorem{lemma}[theorem]{Lemma}
\newtheorem{corollary}[theorem]{Corollary}
\newtheorem{question}[theorem]{Question}
\newtheorem{claim}[theorem]{Claim}

\theoremstyle{definition}
\newtheorem{definition}[theorem]{Definition}
\newtheorem{example}[theorem]{Example}
\newtheorem{notation}[theorem]{Notation}
\newtheorem{condition}[theorem]{Condition}

\theoremstyle{remark}
\newtheorem{remark}[theorem]{Remark}

\usepgfmodule{oo}

\DeclareRobustCommand\ShowAuthors[2]{%
  \ShowAuthorsSignal.emit({#1},{#2})%
}

\pgfoonew \ShowAuthorsSignal=new signal()

\DeclareRobustCommand\ShowAffiliations[1]{%
  \ShowAffiliationsSignal.emit(#1)%
}

\pgfoonew \ShowAffiliationsSignal=new signal()

\newcommand\Author[1]{
  \pgfoonew \CurrentPerson=new person()
  \CurrentPerson.set author(#1)
}

\newcommand\Email[1]{
  \CurrentPerson.set email(#1)
}

\newcommand\Address[1]{
  \CurrentPerson.set address(#1)
}

\newcommand\FirstPerson{0}
\newcommand\LastPerson{}

\pgfooclass{person}{
  \attribute author;
  \attribute email;
  \attribute address;

  \method person() {
    \pgfoothis.get id(\theid)

    \ifnum \FirstPerson=0
      \edef\FirstPerson{\theid}
    \fi

    \edef\LastPerson{\theid}

    \pgfoothis.get handle(\me)
    \ShowAuthorsSignal.connect(\me,show author)
    \ShowAffiliationsSignal.connect(\me,show affiliation)
  }

  \method get author(#1) {
    \pgfooget{author}{#1}
  }

  \method set author(#1) {
    \pgfooset{author}{#1}
  }

  \method get email(#1) {
    \pgfooget{email}{#1}
  }

  \method set email(#1) {
    \pgfooset{email}{#1}
  }

  \method get address(#1) {
    \pgfooget{address}{#1}
  }

  \method set address(#1) {
    \pgfooset{address}{#1}
  }

  \method show author(#1,#2) {%
    \pgfoothis.get id(\theid)%
    \pgfooget{author}{\theauthor}%
    \ifnum\theid=\FirstPerson%
    \else%
      \ifnum\theid=\LastPerson%
        #2%
      \else%
        #1%
      \fi%
    \fi%
    \mbox{\theauthor}%
  }

  \method show affiliation(#1) {%
    \pgfoothis.get id(\theid)%
    \pgfooget{author}{\theauthor}%
    \pgfooget{email}{\theemail}%
    \pgfooget{address}{\theaddress}%
    \ifnum\theid=\FirstPerson%
    \else%
      #1%
    \fi%
    \noindent\begin{minipage}{\linewidth}
      \noindent\begin{tabular}[t]{@{}l}
        \theauthor \quad \textsf{\theemail}\\
      \end{tabular}\\
      \theaddress%
    \end{minipage}%
  }
}

\newcommand\blankfootnote[1]
  {%
    \begin{NoHyper}%
      \renewcommand\thefootnote{}%
      \footnote{#1}%
      \addtocounter{footnote}{-1}%
    \end{NoHyper}%
  }

\newcommand*\abs[1]{\left\lvert #1 \right\rvert}

\numberwithin{theorem}{section}
\numberwithin{equation}{section}

\crefname{sec}{section}{sections}
\Crefname{sec}{Section}{Sections}

\crefname{app}{appendix}{appendices}
\Crefname{app}{Appendix}{Appendices}

\crefname{eqn}{equation}{equations}
\Crefname{eqn}{Equation}{Equations}
\creflabelformat{eqn}{(#2#1#3)}

\crefname{ineq}{inequality}{inequalities}
\Crefname{ineq}{Inequality}{Inequalities}
\creflabelformat{ineq}{(#2#1#3)}

\Crefname{algocf}{Algorithm}{Algorithms}
\usepackage[ruled, vlined, linesnumbered, commentsnumbered]{algorithm2e}
\usepackage{setspace}
\SetKwInput{Input}{Input}
\SetKwInput{Output}{Output}
\crefalias{AlgoLine}{line}

\def\replace#1#2#3{%
 \def\tmp##1#2{##1#3\tmp}%
   \tmp#1\stopreplace#2\stopreplace}
\def\stopreplace#1\stopreplace{}
\newcommand\smallerspaces[1]{\replace{#1}{,}{,\!}}

\newcommand\bubble[1]{#1\text{-}bubble}
\newcommand\bubbles[1]{#1\text{-}bubbles}
\newcommand\variable[1]{\textnormal{\texttt{#1}}}
\newcommand\upperbound{42}
\newcommand\upperboundminusone{41}

\Author{Stevell Muller}
\Address{Mathematik und Informatik, Universität des Saarlandes, 66123 Saarbrücken, Germany.\\
\textit{Current address:} Institut für Algebraische Geometrie, Leibniz Universität Hannover, 30167 Hannover, Germany.}
\Email{muller@math.uni-hannover.de}

\Author{Erik Paemurru}
\Address{Mathematik und Informatik, Universität des Saarlandes, 66123 Saarbrücken, Germany.\\
\textit{Current address:} Institute of Mathematics and Informatics, Bulgarian Academy of Sciences, Acad.\ G.~Bonchev Str.\ bl.~8, 1113, Sofia, Bulgaria.}
\Email{algebraic.geometry@runbox.com}

\newcommand\Thanks{This work was supported by SFB 195 No.~286237555 of DFG. The second author is also supported by the Simons Foundation, grant SFI-MPS-T-Institutes-00007697, and the Ministry of Education and Science of the Republic of Bulgaria, grant DO1-239/10.12.2024.}

\title{Very ample line bundles on weighted projective spaces and weighted blowups}
\author{\ShowAuthors{, }{, }}
\date{6th~August 2026}
\newcommand\keywords{Divisors, invertible sheaves, graded algebras, Veronese subalgebras, projectively normal}
\newcommand\subjclass{05E14, 13A02, 13A30, 14A25, 14A15}

\begin{document}

\maketitle

\begin{abstract}
We consider line bundles $\mathcal{O}(kd)$ on weighted projective spaces, where $k$ is an integer and $d$ is the least common multiple of the weights. Such line bundles are ample if and only if $k$ is positive. On the other hand, determining which line bundles are very ample is a delicate problem. We give various sharp criteria for very ampleness. As an example, if the weights are pairwise coprime, then $\mathcal{O}(d)$ is always very ample, which implies that general smooth well-formed weighted hypersurfaces of dimension at least two are simply connected. We also treat weighted blowups, relative very ampleness, projective normality, Rees rings and generation in degree 1 of Veronese subrings.
\blankfootnote{\textup{2020} \textit{Mathematics Subject Classification}. \subjclass{}.}
\blankfootnote{\textit{Keywords}. \keywords{}.}
\blankfootnote{\Thanks{}}
\end{abstract}

\tableofcontents

\section{Introduction} \label[sec]{sec:introduction}

All rings are commutative and unital. Except where otherwise stated, all graded rings are $\mathbb Z_{\geq0}$-graded.
For every positive integer~$r$, the \emph{$r$th Veronese subring} of a graded ring $S$ is the graded ring $S^{(r)} := \bigoplus_{m \in \mathbb Z_{\geq0}} S_{rm}$, where $S_{rm}$ is the $m$th graded component of~$S^{(r)}$.
We say that a graded ring $S$ is generated in degree $1$ if $S = S_0 \oplus \bigoplus_{m \in \mathbb Z_{\geq1}} S_1^m$.

For the rest of the paper, we fix a nonzero ring $Q$.

\subsection{Graded polynomial rings} \label[sec]{sec:intro polynomial}

In this paper, we study graded polynomial rings
\begin{equation} \label{eqn:intro polynomial}
P_{\bm w} := Q[x_1, \ldots, x_n],
\end{equation}
where $n$ is a positive integer and the variables $x_1, \ldots, x_n$ respectively have positive integer degrees $\bm w := (w_1, \ldots, w_n)$.
In this case, the scheme $\operatorname{Proj} P_{\bm w}$ is called a \emph{weighted projective space} over $Q$ and is denoted by $\mathbb P_Q(\bm w)$.
We give precise answers to the following questions:
\begin{itemize}
\item When is $P_{\bm w}^{(r)}$ generated in degree~1?
\item When is $\mathcal O_{\mathbb P_Q(\bm w)}(r)$ very ample over~$\operatorname{Spec}Q$?
\end{itemize}
While the most interesting case is when $Q$ is the field of complex numbers~$\mathbb C$, it is no more difficult to work over an arbitrary ring.

\begin{remark}
Given a homogeneous ideal $I$ of~$P_{\bm w}$, if $P_{\bm w}$ is generated in degree~1, then $P_{\bm w}/I$ is generated in degree~1, and if $\mathcal O_{\mathbb P_Q(\bm w)}(r)$ is very ample, then $\iota^* \mathcal O_{\mathbb P_Q(\bm w)}(r)$ is very ample for every closed immersion $\iota$ into $\operatorname{Proj} P_{\bm w}$. So considering graded polynomial rings instead of arbitrary finitely generated graded rings does not restrict the setting.
\end{remark}

Denote
\[
d_{\bm w} := \operatorname{lcm}(w_1, w_2, \ldots, w_n).
\]
We explain why we can restrict to $r$th Veronese subrings with $r$ a positive multiple of~$d_{\bm w}$.

\begin{definition} \label{def:well-formed}
A vector $\bm w = (w_1, \ldots, w_n)$ of positive integers is \textbf{well-formed} if $n = w_1 = 1$ or $\gcd(\{w_i \mid i \in I\}) = 1$ for every $(n-1)$-element subset $I$ of $\{1, \ldots, n\}$.
\end{definition}

If $\bm w$ is \emph{well-formed}, then $\mathcal O_{\mathbb P_Q(\bm w)}(r)$ is invertible if and only if $d_{\bm w}$ divides~$r$.
Every weighted projective space is isomorphic to a well-formed weighted projective space (\cite[Corollary~1.1.16]{PS26}).
In any case, \cref{rem:may assume divisible by d} shows that assuming divisibility by $d_{\bm w}$ does not restrict the setting.

\begin{remark} \label{rem:may assume divisible by d}
Let $r$ be an integer. If $n \geq 2$, then all of the following hold:
\begin{enumerate}[label=\textup{(\alph*)}, ref=\alph*]
\item the sheaf $\mathcal O_{\mathbb P_Q(\bm w)}(kd_{\bm w})$ is invertible for every integer~$k$,
\item the sheaf $\mathcal O_{\mathbb P_Q(\bm w)}(kd_{\bm w})$ is generated by its global sections if and only if $k \geq 0$,
\item the sheaf $\mathcal O_{\mathbb P_Q(\bm w)}(kd_{\bm w})$ is ample if and only if $k > 0$ (\cite[Proposition~2.3(i)]{Del75}),
\item the smallest integer $m$ such that $\mathcal O_{\mathbb P_Q(\bm w)}(m)$ is an ample invertible sheaf is $m = d_{\bm w}$ (\cite[Proposition~3C.7]{BR86}),
\item if $P_{\bm w}^{(r)}$ is generated in degree~1, then $r$ is positive and $d_{\bm w}$ divides~$r$.
\end{enumerate}
\end{remark}

\begin{remark} \label{rem:generation implies}
Generation in degree 1 has two immediate consequences:
\begin{enumerate}[label=\textup{(\alph*)}, ref=\alph*]
\item By \cref{thm:very ample}, if $P_{\bm w}^{(r)}$ is generated in degree~1, then $\mathcal O_{\mathbb P_Q(\bm w)}(r)$ is a line bundle that is very ample over~$\operatorname{Spec}Q$.
\item \label{itm:normal domain} By \cite[Corollary~2.24 and Theorem~4.40]{BG09}, if $Q$ is a normal domain and $\bm w$ is well-formed, then $P_{\bm w}^{(r)}$ is generated in degree $1$ if and only if $\mathcal O_{\mathbb P_Q(\bm w)}(r)$ is invertible, very ample over $\operatorname{Spec}Q$ and defines a projectively normal embedding.
\end{enumerate}
\end{remark}

In all the well-formed examples found in this paper with $n \geq 2$, if $\mathcal O_{\mathbb P_Q(\bm w)}(r)$ is very ample, then $P_{\bm w}^{(r)}$ is generated in degree~$1$.

\begin{question} \label{que:very ample and generation equivalence - graded}
Let $n \geq 2$ and $k$ be positive integers. If $\mathcal O_{\mathbb P_Q(\bm w)}(kd_{\bm w})$ is very ample, then is $P_{\bm w}^{(kd_{\bm w})}$ generated in degree~1?
\end{question}

\begin{remark} \label{rem:erraneous}
In \cite[Lemma~4B.6]{BR86}, it is erroneously claimed that for every projective algebraic $\Bbbk$-scheme $X$ and ample Cartier divisor $Y$, the conditions of $Y$ being very ample and the $\Bbbk$-algebra $\bigoplus_{m \in \mathbb Z_{\geq0}} H^0(X, \mathcal O_X(mY))$ being generated in degree 1 are equivalent. A counterexample is given by any normal projective variety over an algebraically closed field that is linearly normal but not projectively normal.
\end{remark}

Even though $\mathcal O_{\mathbb P_Q(\bm w)}(d_{\bm w})$ is ample and generated by its global sections, it is important to note that it is not always very ample:

\begin{example}[{\cite[Remark~2.6]{Del75}}] \label{exa:Delorme}
Let $\bm w := (1,6,10,15)$. Then, $P_{\bm w}^{(60)}$ is generated in degree~1, while $P_{\bm w}^{(30)}$ is not.
The reason is that denoting $\bm u := (1, 4, 2, 1)$, the monomial $\bm x^{\bm u} := x_1 x_2^4 x_3^2 x_4$ of weighted-degree $\bm w \cdot \bm u = 60$ is not the product of two monomials of weighted-degree~$30$.
\end{example}

\begin{remark} \label{rem:Delorme}
In \cref{exa:Delorme}, if $Q$ is a normal domain, then $\mathcal O_{\mathbb P_Q(\bm w)}(30)$ is not very ample over $\operatorname{Spec}Q$ since the monomial $\bm x^{\bm u} / x_2^{10}$ does not belong to the coordinate ring of the affine open $D_+(x_2^5)$ of $Y := \operatorname{Proj} (Q \oplus \bigoplus_{m \in \mathbb Z_{\geq0}} P_{30}^m)$ but does belong to its integral closure, where $P := P_{\bm w}$.
If $Q$ is an algebraically closed field, then $\mathcal O_{\mathbb P_Q(\bm w)}(30)$ separates points since $a x_i^{30/w_i} - b x_j^{30/w_j}$ is a global section for all indices $i, j$ and elements $a, b \in Q$ but $\mathcal O_{\mathbb P_Q(\bm w)}(30)$ does not separate tangent vectors since there is a tangent vector $\operatorname{Spec} Q[\varepsilon] \to D_+(x_2)$ of $\mathbb P_Q(\bm w)$ given by $\bm x^{\bm u} / x_2^{10} \mapsto \varepsilon$ and $h / x_2^5 \mapsto 0$ for all $h \in P_{ 30}$.
\end{remark}

The problem of determining the smallest positive integer $r$ such that $P_{\bm w}^{(r)}$ is generated in degree $1$ or $\mathcal O_{\mathbb P_Q(\bm w)}(r)$ is very ample over $\operatorname{Spec} Q$ turns out to be very subtle.
It is surprising that the answers to many basic questions concerning this least integer $r$ were previously unknown.
This paper fills this gap in the literature.

Previously, the state of the art in determining which divisors on weighted projective spaces are very ample was the following:

\begin{proposition}[{\cite[Proposition~2.3]{Del75}, see \cite[Theorem~4B.7]{BR86} for a translation}] \label{thm:delorme}
Denote $\bm w = (w_1, \ldots, w_n)$. If $n = 1$, then define $G(\bm w) := -w_1$, and otherwise define
\[
G(\bm w) := -\sum_{i=1}^n w_i + \frac{1}{n-1} \sum_{v=2}^n \binom{n-2}{v-2}^{-1} \sum_{\substack{J \subseteq \{1, \ldots, n\}\\\abs{J} = v}} m_J,
\]
where for a subset $J$ of $\{1, \ldots, n\}$, we denote $m_J := \operatorname{lcm}\{w_i \mid i \in J\}$.
If $k$ is a positive integer satisfying $kd_{\bm w} > G(\bm w)$, then $P_{\bm w}^{(k d_{\bm w})}$ is generated in degree~1.
\end{proposition}

In this paper, we improve on it in several ways:
\begin{enumerate}
\item Instead of a one-way implication, we give a combinatorial condition equivalent to generation in degree 1 (see \cref{thm:bubble and generation in degree 1 - graded}) and a combinatorial condition equivalent to very ampleness (see \cref{thm:bubble satisfying condition implies not very ample - graded}).
\item Instead of a criterion which is sometimes inconclusive, we give a combinatorial algorithm that guarantees to determine generation in degree 1 (see \cref{app:algorithm}).
The algorithm is very fast, with an average runtime of 3.305\,\textmu s on a single CPU core for the roughly 2 trillion weight vectors that we checked (see \cref{rem:computational}).
\item The combinatorial criteria that we give are often stronger than \cref{thm:delorme}. For a concrete comparison, looking at strictly increasing vectors $\bm w = (w_1, \ldots, w_n)$ of positive integers at most 15 where $4 \leq n \leq 15$, \cref{cor:bound on u when all entries are positive} alone gives a bound $\bm w \cdot \bm v - d_{\bm w}$ (see \cref{cor:bubble and generation in degree 1 - graded,cor:bubble and generation in degree 1 - Rees} for how this bound is related to generation in degree 1) that is strictly smaller than $G(\bm w)$ from \cref{thm:delorme} roughly 98.9\% of the time. The stronger bound $\bm w \cdot \bm{\gamma} - d_{\bm w}$ from \cref{alg: Bubbles algorithm} is strictly smaller than $G(\bm w)$ roughly 99.997\% of the time.
\item We give several criteria of generation in degree 1 that are easier to use and compute by hand than \cref{thm:delorme}.
\item In contrast to \cref{thm:delorme}, even if some of our criteria are not conclusive, they still restrict the number of cases to check.
\item We describe what happens for uniformly randomly chosen weights (see \cref{thm:probability}).
\item In addition to graded rings and the corresponding weighted projective spaces, we also treat Rees rings and the corresponding weighted blowups.
\end{enumerate}

Similarly to \cite[Proposition~2.5]{Del75}, several statements in this paper also hold when we remove repeated weights or when we make the weight vector well-formed.
We make this precise below.

\begin{definition} \label{def:equivalences}
Consider the following operations on a nonempty vector $\bm w$ of positive integers:
\begin{enumerate}
\item permuting the entries,
\item adding or removing repeated entries,
\item multiplying all entries by the same positive integer or dividing all entries by a positive common divisor,
\item adding or removing entries equal to the least common multiple of the entries while not changing the least common multiple,
\item if the length of $\bm w$ is at least two and $\gcd(\bm w) = 1$, then choosing an index $i$ and either multiplying each entry $w_j$ where $j \neq i$ by the same positive integer coprime to $w_i$ or dividing each entry $w_j$ where $j \neq i$ by a positive common divisor of the entries $w_j$ where $j \neq i$.
\end{enumerate}
If two nonempty vectors of positive integers are equal after applying a suitable choice of any number of any of the operations
\begin{itemize}
\item (1)–(5), then we say they are \textbf{weakly equivalent},
\item (1)–(4), then we say they are \textbf{strongly equivalent}.
\end{itemize}
\end{definition}

Our first result is the following:

\begin{theorem}[{\cref{cor:bubble and generation in degree 1 - graded,thm:coprime weights}, \cref{lem:simplification bubbles}\labelcref{itm:weakly}}] \label{thm:introduction coprime graded}
Let $\bm w$ be weakly equivalent to a vector of pairwise coprime integers.
Then, $P^{(d_{\bm w})}_{\bm w}$ is generated in degree~$1$.
\end{theorem}

By \cref{rem:generation implies}, this gives the immediate corollary:

\begin{corollary} \label{cor:pairwise coprime}
Let $\bm w$ be weakly equivalent to a vector of pairwise coprime integers.
Let $\mathbb P_Q(\bm w)$ be the weighted projective space over~$Q$ with weights~$\bm w$.
Then, the line bundle $\mathcal O_{\mathbb P(\bm w)}(d_{\bm w})$ is very ample over~$\operatorname{Spec} Q$. Moreover, if $Q$ is a normal domain and $\bm w$ is well-formed, then the global sections of $\mathcal O_{\mathbb P(\bm w)}(d_{\bm w})$ define a projectively normal embedding.
\end{corollary}

\begin{remark}
We give two applications of \cref{cor:pairwise coprime}.
\begin{enumerate}[label=\textup{(\alph*)}, ref=\alph*]
\item In the special case of \cref{cor:pairwise coprime} where $n = 3$, we recover the result \cite[Remark~2.6]{Del75} (see \cite[Remark~3 in section 4B]{BR86} for an English translation).
Namely, since every weighted projective plane is isomorphic to a weighted projective plane with pairwise coprime weights, we find that every ample line bundle of the form $\mathcal O(r)$ on a weighted projective plane is very ample and moreover, defines a projectively normal embedding if the base ring is a normal domain and the weights are well-formed.
\item Let $n \geq 4$ be an integer, let $\bm w := (w_1, \ldots, w_n)$ be a vector of pairwise coprime positive integers and let $X$ be a smooth hypersurface of $\mathbb P_{\mathbb C}(\bm w)$ of weighted-degree divisible by~$d_{\bm w}$. Note that by \cite[Proposition~3.8.1]{PS26}, every smooth well-formed weighted hypersurface is of the form above. By \cref{cor:pairwise coprime}, $X$ is a very ample Cartier divisor of $\mathbb P_{\mathbb C}(\bm w)$ and defines a projectively normal embedding.
By \cite{Arm68}, the smooth locus $U$ of $\mathbb P_{\mathbb C}(\bm w)$ is simply connected.
Therefore, by the Lefschetz hyperplane theorem for quasiprojective varieties (\cite[Theorem~1.1.3]{HL85} or \cite[Theorem~9.3.1 in section~9.3.3]{HL20}), the complex manifold corresponding to a general $X$ is simply connected.
\end{enumerate}
\end{remark}

\subsection{Rees rings} \label[sec]{sec:intro Rees}

In this paper, we also study generation in degree~1 and very ampleness for Rees rings and weighted blowups. A \emph{Rees ring} is a graded subring $\bigoplus_{m \in \mathbb Z_{\geq0}} t^m I_m$ of the graded polynomial ring $A[t]$ that is finitely generated over~$A$, where $A$ is a ring, $t^m I_m$ is the $m$th graded component for every nonnegative integer $m$ and $A = I_0 \supseteq I_1 \supseteq I_2 \supseteq \cdots$ is a filtration of ideals of~$A$.

We only consider Rees rings of the form
\begin{equation} \label[eqn]{eqn:Rees ring R}
R_{\bm w} := Q[t^{-1}, t^{w_1} x_1, \ldots, t^{w_n} x_n]_{\geq0},
\end{equation}
where $\bm w := (w_1, \ldots, w_n)$ is a nonempty vector of positive integers, where $t$ denotes the degree and $S_{\geq0}$ denotes the nonnegative part $\bigoplus_{m \in \mathbb Z_{\geq0}} S_m$ of a $\mathbb Z$-graded ring $S = \bigoplus_{m \in \mathbb Z} S_m$.
Equivalently, $R_{\bm w}$ can be given as
\[
R_{\bm w} = Q\mleft[\mleft\{ t^d x_i \;\middle|\; i \in \{1, \ldots, n\},\, d \in \{0, \ldots, w_i\} \mright\}\mright].
\]
Note that the degree zero part is $(R_{\bm w})_0 = Q[x_1, \ldots, x_n]$.

\begin{remark} \label{rem:Rees subschemes}
Given a homogeneous ideal $I$ of~$R_{\bm w}$, if $R_{\bm w}$ is generated in degree~1, then $R_{\bm w}/I$ is generated in degree~1, and if $\mathcal O_{\operatorname{Proj} R_{\bm w}}(r)$ is very ample, then $\iota^* \mathcal O_{\operatorname{Proj} R_{\bm w}}(r)$ is very ample for every closed immersion $\iota$ into $\operatorname{Proj} R_{\bm w}$.
So, considering only Rees rings of the form \eqref{eqn:Rees ring R} does not restrict the setting.
\end{remark}

The natural morphism $\varphi\colon \operatorname{Proj} R_{\bm w} \to \operatorname{Spec} ((R_{\bm w})_0)$ is called the \emph{weighted blowup} of the affine \mbox{$n$-space} $\mathbb A_Q^n := \operatorname{Spec} ((R_{\bm w})_0)$ with weights~$\bm w$.
If $R_{\bm w}^{(r)}$ is generated in degree $1$ for some positive integer~$r$, then $\varphi$ is the blowup along~$I_r$, which is an ideal in~$(R_{\bm w})_0$.
If $Q$ is an integral domain, then $R_{\bm w}$ coincides with the integral closure of $Q[x_1, \ldots, x_n, t^{w_1} x_1, \ldots, t^{w_n} x_n]$ in $Q[t, x_1, \ldots, x_n]$.
If $Q$ is the field of complex numbers~$\mathbb C$, then $\varphi$ coincides with the toric morphism in \cite[Proposition-Definition~10.3]{KM92}.

\begin{example}
The weighted blowup with weights $(2, 3)$ of the affine plane $\mathbb A_Q^2 := \operatorname{Spec} Q[x_1, x_2]$ over $Q$ is the morphism $\varphi\colon \operatorname{Proj} R \to \mathbb A_Q^2$, where the Rees ring $R$ is given by
\[
R := R_{(2,3)} = Q[t^{-1}, t^2 x_1, t^3 x_2]_{\geq0} = Q[x_1, t x_1, t^2 x_1, x_2, t x_2, t^2 x_2, t^3 x_2].
\]
Note that $t^6 x_1^3$, $t^6 x_1^2 x_2$ and $t^6 x_2^2$ are all elements of~$R$.
The $6$th Veronese subring
\[
R^{(6)} = Q[x_1, x_2, t x_1^3, t x_1^2 x_2, t x_2^2]
\]
of $R$ is generated in degree~$1$, or equivalently, the graded rings $R_{0} \oplus \bigoplus_{m \in \mathbb Z_{\geq1}} t^m R_{6}^m$ and $\bigoplus_{m \in \mathbb Z_{\geq0}} t^m R_{6m}$ are equal. Therefore, $\varphi$ coincides with the blowup along the ideal $(x_1^3, x_1^2 x_2, x_2^2)$ of $Q[x_1, x_2]$, which is the integral closure of the ideal $(x_1^3, x_2^2)$ if $Q$ is an integral domain.
\end{example}

\begin{remark}
Note that the blowup of the affine $n$-space $\mathbb A_{\mathbb C}^{n}$ along a linear subspace $V(x_1, \ldots, x_k)$ can be written as the blowup along a point of $\mathbb A_{\mathbb C[x_{k+1}, \ldots, x_{n}]}^{k}$, the affine $k$-space over $\mathbb C[x_{k+1}, \ldots, x_n]$.
Analogously, every weighted blowup with a positive-dimensional centre can be written as a weighted blowup with centre a point with a possibly-larger coefficient ring.
So, since we allow general coefficient rings, considering only strictly positive weights and zero-dimensional centres does not restrict the setting.
\end{remark}

Every Rees ring $R_{\bm w}$ as in \cref{eqn:Rees ring R} is generated as a graded ring by homogeneous elements with degrees in the set $\{1, 2, \ldots, \max(\bm w)\}$. We have the following weak bound:

\begin{proposition}[{\cref{cor:bubble and generation in degree 1 - Rees,thm:graded Rees}}] \label{thm:introduction max}
Let $N := \operatorname{lcm}(1, 2, \ldots, \max(\bm w))$. Then, $R_{\bm w}^{(N)}$ is generated in degree~1.
\end{proposition}

We have analogues of \cref{rem:may assume divisible by d,rem:generation implies}:

\begin{remark}\label{rem:Rees divisible by d}
Let $k$ and $r$ be integers. If $n \geq 2$, then both of the following hold:
\begin{enumerate}
\item the sheaf $\mathcal O_{\operatorname{Proj} R_{\bm w}}(r)$ is invertible if and only if $d_{\bm w}$ divides~$r$,
\item the sheaf $\mathcal O_{\operatorname{Proj} R_{\bm w}}(kd_{\bm w})$ is generated by its global sections if and only if $k \geq 0$ and ample if and only if $k > 0$.
\end{enumerate}
\end{remark}

\begin{remark} \label{rem:Rees generation implies}
By \cref{thm:very ample}, if $R_{\bm w}^{(r)}$ is generated in degree~1, then $\mathcal O_{\operatorname{Proj} R_{\bm w}}(r)$ is invertible and very ample over~$\mathbb{A}^n_Q := \operatorname{Spec} ((R_{\bm w})_0)$.
\end{remark}

\begin{question} \label{que:very ample and generation equivalence - Rees}
Let $n \geq 2$ and $k$ be positive integers. If $\mathcal O_{\operatorname{Proj} R_{\bm w}}(kd_{\bm w})$ is very ample over $\mathbb{A}_Q^n$, then is $R_{\bm w}^{(kd_{\bm w})}$ generated in degree~1?
\end{question}

Similarly to \cref{exa:Delorme}, the $d_{\bm w}$th Veronese subring is not necessarily generated in degree~1:

\begin{example} \label{exa:Rees Delorme}
Let $\bm w := (1, 6, 10, 15)$.
Then, $R_{\bm w}^{(60)}$ is generated in degree~1, while $R_{\bm w}^{(30)}$ is not very ample over $\mathbb{A}^4_Q$.
\end{example}

The result analogous to \cref{thm:introduction coprime graded} for Rees rings is the following:

\begin{theorem}[{\cref{cor:bubble and generation in degree 1 - Rees,thm:coprime weights}, \cref{lem:simplification bubbles}\labelcref{itm:strongly}}] \label{thm:Rees coprime}
Let $\bm w$ be strongly equivalent to a vector of pairwise coprime integers.
Then, $R_{\bm w}^{(d_{\bm w})}$ is generated in degree~$1$.
\end{theorem}

By \cref{rem:Rees generation implies}, we have the immediate corollary:

\begin{corollary} \label{cor:Rees coprime}
Let $\bm w$ be strongly equivalent to a vector of pairwise coprime integers.
Then, the line bundle $\mathcal O_{\operatorname{Proj} R_{\bm w}}(d_{\bm w})$ is very ample over $\mathbb A_Q^n$.
\end{corollary}

Below, we give results on generation in degree 1 and very ampleness for both graded rings and Rees rings in various settings.

\subsection{Sharp upper bound} \label[sec]{sec:sharp}

In \cite[Lemma~(2.1.6)(v)]{EGAII}, it is shown that for all $k\geq \max(1, n-1)$, the Veronese subring $P_{\bm w}^{(kd_{\bm w})}$ is generated in degree~$1$. In the case $Q$ is a field, \cite[Theorem 1.3.3]{bgt97} and \cite[Theorem 1]{EW91} show that $P_{\bm w}^{(kd_{\bm w})}$ is generated in degree~$1$ for all $k\geq \max(1, n-2)$, improving Grothendieck's bound. By \cref{thm:bubble and generation in degree 1 - graded}, this actually holds for arbitrary nonzero coefficient rings~$Q$. We further strengthen this using our equivalence relations:

\begin{theorem}[{\cref{cor:bubble and generation in degree 1 - graded,thm:bound n-2 divisible}, \cref{lem:simplification bubbles}\labelcref{itm:weakly}}] \label{thm:sharp}
Let $n$ be a positive integer and let $\bm w \in \mathbb{Z}_{\geq1}^n$.
Then, for every integer $k \geq \max(1, N-2)$, where $N$ is the least length among vectors weakly equivalent to~$\bm w$, the graded ring $P_{\bm w}^{(kd_{\bm w})}$ is generated in degree~$1$.
\end{theorem}

\Cref{exa:sharp graded} shows the sharpness of \cref{thm:sharp} for every~$n$.

\begin{example}[{\cref{thm:bubble satisfying condition implies not very ample - graded,thm:infinitely many bubbles graded}, \cref{lem:condition}\labelcref{itm:condition weakly}, \cref{lem:representative}\labelcref{itm:weak representative}}] \label{exa:sharp graded}
Let $n \geq 2$ be an integer. For every vector $\bm p := (p_1, p_2, \ldots, p_n)$ of pairwise coprime integers greater than $1$ such that $k_{\bm p} = n-1$ (see \cref{lem:unique alpha} for the definition of $k_{\bm p}$), define
\[
\bm w_{\bm p} := \mleft( 1, \frac{\prod_{i=1}^{n} p_i}{p_1}, \ldots, \frac{\prod_{i=1}^{n} p_i}{p_{n}} \mright) \in \mathbb Z_{\geq1}^{n+1}.
\]
Note that the vector $\bm w_{\bm p}$ is well-formed.
Let $\bm w$ be a nonempty vector of positive integers such that a subsequence of $\bm w$ with the same least common multiple is weakly equivalent to~$\bm w_{\bm p}$.
Then, $\mathcal O_{\mathbb{P}_Q(\bm w)}((n-2)d_{\bm w})$ is not very ample over~$\operatorname{Spec} Q$.

For every $n \in \{2, 3, \ldots, 6\}$, \cref{tab:k} contains the unique vector $\bm p \in \mathbb Z_{\geq2}^n$ with the smallest vector $\bm w_{\bm p}$ in reverse lexicographic order such that $k_{\bm p} = n-1$. See \cref{lem:inductive,rem:after lemma inductive} how to inductively construct examples of $\bm p \in \mathbb Z_{\geq2}^n$ with $k_{\bm p} = n-1$ for every $n \geq 3$.
By \cref{thm:infinitely many for each k}, there exist infinitely many vectors $\bm p \in \mathbb Z_{\geq2}^n$ of pairwise different prime numbers with $k_{\bm p} = n-1$ for every $n \geq 2$.
\end{example}

\begin{table}[h!]
\centering
\caption{Examples with $k_{\bm p} = n-1$ and $k_{\bm q} = 1$\label{tab:k}}
\begin{tabular}{rrr}
\toprule
$n$ & $(p_1, p_2, \ldots, p_n)$ & $(q_1, q_2, \ldots, q_n)$\\*
\midrule
$2$ & $(3, 2)$ & $(3, 2)$\\
$3$ & $(5, 3, 2)$ & $(5, 4, 3)$\\
$4$ & $(8, 7, 5, 3)$ & $(11, 7, 4, 3)$\\
$5$ & $(19, 17, 5, 4, 3)$ & $(25, 11, 7, 4, 3)$\\
$6$ & $(29, 19, 13, 5, 4, 3)$ & $(17, 16, 13, 11, 9, 5)$\\
\bottomrule
\end{tabular}
\end{table}

To the best knowledge of the authors, this provides the first example of a graded ring $P$ generated by homogeneous elements of positive integer degrees $w_1, \ldots, w_n$ such that the $2d_{\bm w}$th Veronese subring of $P$ is not generated in degree~1.

The corresponding result for Rees rings is weaker:

\begin{theorem}[{\cref{cor:bubble and generation in degree 1 - Rees,thm:bound n-1}, \cref{lem:simplification bubbles}\labelcref{itm:strongly}}] \label{thm:Rees sharp bound}
Let $n$ be a positive integer and let $\bm w \in\mathbb{Z}_{\geq1}^n$.
Then, for every integer $k \geq \max(1, N-1)$, where $N$ is the least length among vectors strongly equivalent to~$\bm w$, the graded ring $R_{\bm w}^{(kd_{\bm w})}$ is generated in degree~$1$.
\end{theorem}

Again, the result is sharp:

\begin{example}[{\cref{thm:bubble satisfying condition implies not very ample - Rees}, \cref{thm:infinitely many bubbles rees}, \cref{lem:condition}\labelcref{itm:condition strongly}, \cref{lem:representative}\labelcref{itm:strong representative}}] \label{exa:sharp Rees}
Let $n \geq 2$ be an integer. For every vector $\bm q := (q_1, \ldots, q_n)$ of pairwise coprime integers greater than $1$ such that $k_{\bm q} = 1$ (see \cref{lem:unique alpha} for the definition of $k_{\bm q}$), define
\[
\bm v_{\bm q} := \mleft( \frac{\prod_{i=1}^{n} q_i}{q_1}, \ldots, \frac{\prod_{i=1}^{n} q_i}{q_n} \mright) \in \mathbb Z_{\geq1}^{n}.
\]
Let $\bm w$ be a nonempty vector of positive integers such that a subsequence of $\bm w$ with the same least common multiple is strongly equivalent to~$\bm v_{\bm q}$.
Then, $\mathcal O_{\operatorname{Proj} R_{\bm w}}((n-2)d_{\bm w})$ is not very ample over $\operatorname{Spec} (R_{\bm w})_0$.

For every $n \in \{2, 3, \ldots, 6\}$, \cref{tab:k} contains the unique vector $\bm q \in \mathbb Z_{\geq2}^n$ with the smallest vector $\bm v_{\bm q}$ in reverse lexicographic order such that $k_{\bm q} = 1$. See \cref{lem:inductive,rem:after lemma inductive} how to inductively construct examples of $\bm q \in \mathbb Z_{\geq2}^n$ with $k_{\bm q} = 1$ for every $n \geq 3$.
By \cref{thm:infinitely many for each k}, there exist infinitely many vectors $\bm q \in \mathbb Z_{\geq2}^n$ of pairwise different prime numbers with $k_{\bm q} = 1$ for every $n \geq 2$.
\end{example}

For weighted blowups with respect to well-formed weights, we have a stronger result:

\begin{theorem}[{\cref{cor:bubble and generation in degree 1 - Rees,thm:bound n-2 well-formed}, \cref{lem:simplification bubbles}\labelcref{itm:strongly}}] \label{thm:bound well-formed}
Let $n$ be a positive integer and let $\bm w \in\mathbb{Z}_{\geq1}^n$.
Then, for every integer $k \geq \max(1, N-2)$, where $N$ is the least length among well-formed vectors strongly equivalent to~$\bm w$, the graded ring $R_{\bm w}^{(kd_{\bm w})}$ is generated in degree~$1$.
\end{theorem}

The result is sharp by the following:

\begin{example}[{\cref{thm:bubble satisfying condition implies not very ample - Rees}, \cref{thm:infinitely many bubbles graded}, \cref{lem:condition}\labelcref{itm:condition weakly}, \cref{lem:representative}\labelcref{itm:weak representative}}] \label{exa:well-formed}
Define $\bm w_{\bm p} \in \mathbb Z_{\geq1}^{n+1}$ as in \cref{exa:sharp graded}.
Let $\bm w$ be a nonempty vector of positive integers such that a subsequence $\bm w$ with the same least common multiple is weakly equivalent to~$\bm w_{\bm p}$.
Then, $\mathcal O_{\operatorname{Proj} R_{\bm w}}((n-2)d_{\bm w})$ is not very ample over $\operatorname{Spec} (R_{\bm w})_0$.
\end{example}

\begin{remark}
\Cref{thm:Rees sharp bound,thm:bound well-formed} do not hold if we replace \emph{strongly equivalent} by \emph{weakly equivalent}. To see this, note that $\bm w := (6, 14, 21) = \bm v_{(7, 3, 2)}$ is weakly equivalent to the vector $(1) \in \mathbb Z^1$.
\end{remark}

\subsection{Small weights} \label[sec]{sec:intro small weights}

In this section, we discuss generation in degree 1 and very ampleness for small weights.

\newcommand\setAWeak{
    \mleft\{
        \begin{gathered}
            (1, 6, 10, 15),\, (1, 6, 22, 33),\, (1, 12, 20, 30),\\
            (1, 12, 21, 28),\, (1, 12, 22, 33),\, (1, 15, 21, 35),\\
            (1, 21, 28, 36),\, (1, 21, 30, 35),\, (2, 12, 15, 20),\\
            (2, 12, 20, 30),\, (2, 21, 30, 35),\, (4, 15, 24, 40),\\
            (4, 24, 30, 40),\, (1, 3, 24, 30, 40)
        \end{gathered}
    \mright\}
}

\newcommand\setBStrongUnionAWeak{
    \mleft\{
        \begin{gathered}
            (6, 14, 21),\, (6, 26, 39),\, (10, 14, 35),\, (12, 15, 20),\\
            (12, 26, 39),\, (14, 20, 35),\, (15, 20, 24),\, (15, 24, 40),\\
            (20, 28, 35),\, (24, 30, 40),\, (28, 35, 40),\\
            (1, 6, 10, 15),\, (1, 6, 22, 33),\, (1, 12, 20, 30),\\
            (1, 12, 21, 28),\, (1, 12, 22, 33),\, (1, 15, 21, 35),\\
            (1, 21, 28, 36),\, (1, 21, 30, 35),\, (2, 12, 20, 30),\\
            (2, 21, 30, 35),\, (3, 12, 20, 30)
        \end{gathered}
    \mright\}
}

\newcommand\setA{
    \mleft\{
        \begin{gathered}
            (1, 6, 10, 15),\, (1, 6, 22, 33),\, (1, 12, 21, 28),\\
            (1, 12, 22, 33),\, (1, 15, 21, 35),\, (1, 21, 28, 36),\\
            (1, 21, 30, 35),\, (2, 21, 30, 35),\, (1, 3, 24, 30, 40)
        \end{gathered}
    \mright\}
}

\newcommand\setB{
    \mleft\{
        \begin{gathered}
            (6, 14, 21),\, (6, 26, 39),\, (10, 14, 35),\, (12, 15, 20),\\
            (12, 26, 39),\, (14, 20, 35),\, (15, 20, 24),\, (15, 24, 40),\\
            (20, 28, 35),\, (28, 35, 40),\, (3, 12, 20, 30)
        \end{gathered}
    \mright\}
}

\begin{definition} \label{def:sets}
Define the sets
\begin{gather*}
    A := \setA{},\\
    B := \setB{}.
\end{gather*}
Let $A_{\mathrm{w}}$ denote the set of all the strictly increasing nonempty vectors of positive integers $\bm w$ satisfying $\max(\bm w) < \upperbound{}$ such that $\bm w$ is equivalent to some element of $A$ under the operations (1), (3) and (5) of \cref{def:equivalences}.
Let $B_{\mathrm{s}}$ denote the set of all the strictly increasing nonempty vectors of positive integers $\bm w$ satisfying $\max(\bm w) < \upperbound{}$ such that $\bm w$ is equivalent to some element of $B$ under the operations (1) and~(3).
Let $C$ denote the union of $B_{\mathrm{s}}$ and the vectors $\bm w$ in $A_{\mathrm{w}}$ such that no subsequence of $\bm w$ with the same least common multiple is in $B_{\mathrm{s}}$.
\end{definition}

\begin{remark} \label{rem:small large}
We have
\begin{gather*}
    A_{\mathrm{w}} = \setAWeak{},\\
    C = \setBStrongUnionAWeak{}.
\end{gather*}
\end{remark}

\begin{theorem}[{\cref{cor:bubble and generation in degree 1 - graded,thm:bubble satisfying condition implies not very ample - graded,thm:small weights}}] \label{thm:intro small graded}
Let $n$ be a positive integer and let $\bm w\in\mathbb{Z}^n_{\geq1}$.
Then, all of the following hold:
\begin{enumerate}[label=\textup{(\alph*)}, ref=\alph*]
\item if $\max(\bm w) < \upperbound{}$, then $P_{\bm w}^{(2d_{\bm w})}$ is generated in degree~$1$,
\item if $\max(\bm w) < \upperbound{}$, then $P_{\bm w}^{(d_{\bm w})}$ is not generated in degree $1$ if and only if, up to permutation, a subsequence of $\bm w$ with the same least common multiple is in~$A_{\mathrm{w}}$,
\item if a subsequence of $\bm w$ with the same least common multiple is weakly equivalent to an element of~$A$, then $\mathcal O_{\mathbb P_Q(\bm w)}(d_{\bm w})$ is not very ample over~$\operatorname{Spec} Q$.
\end{enumerate}
\end{theorem}

\begin{remark}
The general strategy we use to prove \cref{thm:intro small graded} and to obtain the set $A$ is the following:
\begin{itemize}
\item enumerate all strictly ordered vectors $\bm w \in \mathbb Z_{\geq1}^n$ with $\max(\bm w) < \upperbound{}$,
\item for each such vector $\bm w$, list all the monomials of weighted degree $d_{\bm w}$ and all the monomials of weighted degree $2d_{\bm w}$,
\item check one-by-one which monomials of weighted degree $2d_{\bm w}$ are a product of two monomials of weighted degree $d_{\bm w}$.
\end{itemize}
Directly applying this strategy is hopeless: even if we consider only one vector~$\bm w$, in the extremal case $\bm w := (1, 2, \ldots, 41)$, there are roughly up to $10^{596}$ monomials of weighted-degree $d_{\bm w}$ and roughly up to $10^{608}$ monomials of weighted-degree~$2d_{\bm w}$.
Instead, we develop a specialised algorithm, given in \cref{app:algorithm}, that uses various clever estimations to reduce the number of monomials to compute for a given vector~$\bm w$. We still loop through all the roughly two trillion vectors $\bm w$ one-by-one.
See \cref{rem:computational} for more remarks on computational complexity and \cref{algo: proof algorithm,algo: proof bubbles algorithm} for the proofs of the algorithm.
\end{remark}

\begin{remark}
\Cref{thm:intro small graded} can be used to describe very ample divisors on some known lists of weighted hypersurfaces.
\begin{enumerate}[label=\textup{(\alph*)}, ref=\alph*]
\item By \cite[Theorem~4.5]{Rei80} and \cite[13.3]{IF00}, there is a list of 95 families of canonical K3 surfaces that are well-formed quasismooth hypersurfaces in weighted projective spaces over the complex numbers. Since the largest weight that appears in the list is~33, it is easy to check using \cref{thm:intro small graded} that $\mathcal O_{\mathbb P_{\mathbb{C}}(\bm w)}(d_{\bm w})$ is very ample for every vector $\bm w := (w_1, w_2, w_3, w_4)$ in that list. Moreover, by \cref{rem:generation implies}\labelcref{itm:normal domain}, the global sections of $\mathcal O_{\mathbb P_{\mathbb{C}}(\bm w)}(d_{\bm w})$ define a projectively normal embedding.
\item Similarly, in \cite[16.6]{IF00} there is a list of 95 families of quasismooth terminal Fano varieties $X$ of dimension 3 that are well-formed quasismooth hypersurfaces in weighted projective spaces over the complex numbers with anticanonical class $\mathcal O_X(1)$ that are not linear cones. Its completeness was conjectured in \cite[15.2]{IF00} and proved in \cite[Theorem~6.1]{CCC11}, see \cite[Theorem~2.11]{Oka23} for an alternative proof. By \cref{rem:generation implies}\labelcref{itm:normal domain}, for all the weight vectors~$\bm w$, the line bundle $\mathcal O_{\mathbb P_{\mathbb{C}}(\bm w)}(d_{\bm w})$ is very ample and their global sections define projectively normal embeddings.
\end{enumerate}
\end{remark}

\begin{remark}
    According to \cite[Corollary 1.7.8]{PS26} a well-formed weighted projective space $\mathbb{P}_\mathbb C(w_1, \ldots, w_n)$ is Gorenstein if and only if $d_{\bm w}$ divides $\sum_{i=1}^nw_i$.
    For each positive integer~$n$, there is a bijection between the set of vectors $(w_1, \ldots, w_n) \in \mathbb Z_{\geq1}^n$ with $\gcd(\bm w)$ equal to $1$ and $\sum_{i=1}^nw_i$ divisible by $d_{\bm w}$ and the set of vectors $\left(x_1, \ldots, x_n\right) \in \mathbb Z_{\geq1}^n$ with sum of reciprocals equal to~$1$. Namely, we have the following bijection and its inverse:
    \begin{align*}
        (w_1, \ldots, w_n) & \mapsto \left(\frac{\sum_{i=1}^nw_i}{w_1},\ldots, \frac{\sum_{i=1}^nw_i}{w_n}\right),\\
        \left(x_1, \ldots, x_n\right) & \mapsto \left(\frac{d_{\bm x}}{x_1},\ldots, \frac{d_{\bm x}}{x_n}\right).
    \end{align*}
    Since for each positive integer $n$ there are only finitely many vectors $\left(x_1, \ldots, x_n\right) \in \mathbb Z_{\geq1}^n$ with sum of reciprocals equal to~$1$, there are only finitely many Gorenstein well-formed weighted projective spaces in each dimension.

    Our methods show that the anticanonical divisor of every Gorenstein well-formed weighted projective space of dimension up to 4 is very ample and that this is not true in dimension 5. In particular, the reverse lexicographically smallest example of a Gorenstein well-formed weighted projective space with an anticanonical divisor that is not very ample is $\mathbb{P}_Q(1,2,10,12,15,20)$.
    This answers a question of Kristin DeVleming.
\end{remark}

We have a corresponding result for Rees rings and weighted blowups.

\begin{theorem}[{\cref{cor:bubble and generation in degree 1 - Rees,thm:bubble satisfying condition implies not very ample - Rees,thm:small weights}}] \label{thm:introduction small weights Rees}
Let $n$ be a positive integer and let $\bm w\in\mathbb{Z}^n_{\geq1}$.
Then, all of the following hold:
\begin{enumerate}[label=\textup{(\alph*)}, ref=\alph*]
\item if $\max(\bm w) < \upperbound{}$, then $R_{\bm w}^{(2d_{\bm w})}$ is generated in degree~$1$,
\item if $\max(\bm w) < \upperbound{}$, then $R_{\bm w}^{(d_{\bm w})}$ is not generated in degree $1$ if and only if, up to permutation, a subsequence of $\bm w$ with the same least common multiple is in~$C$,
\item if a subsequence of $\bm w$ with the same least common multiple is weakly equivalent to an element of $A$ or strongly equivalent a element of~$B$, then $\mathcal O_{\operatorname{Proj} R_{\bm w}}(d_{\bm w})$ is not very ample over $\mathbb{A}^n_Q$.
\end{enumerate}
\end{theorem}

\begin{remark}
The example $(10, 14, 35)$ appeared already in \cite[Example~3.5]{AT14}.
\end{remark}

\subsection{Generic behaviour}

By \cref{exa:sharp graded}, there exist infinitely many well-formed vectors $\bm w$ such that $P_{\bm w}^{(d_{\bm w})}$ is not generated in degree~1.
On the other hand, \cref{thm:intro small graded} suggests that for most weight vectors~$\bm w$, the $d_{\bm w}$th Veronese subring $P_{\bm w}^{(d_{\bm w})}$ is generated in degree~$1$.
Below, we make this precise.
We let $\{1,2,\ldots, M\}^n$ denote the set of integer vectors $(w_1, w_2, \ldots, w_n)$ where for all $i \in \{1, \ldots, n\}$, we have $1 \leq w_i \leq M$.

\begin{theorem}[{\cref{cor:bubble and generation in degree 1 - graded,cor:bubble and generation in degree 1 - Rees,thm:no bubbles}}] \label{thm:probability}
    For all positive integers $n$ and~$M$, we define
    \[
    p(n, M) := \frac
        {
            \mleft|
                \mleft\{
                    \bm w\in \{1,2,\ldots, M\}^n
                    \;\middle|\;
                    \begin{gathered}
                        \text{$P_{\bm w}^{(d_{\bm w})}$ and $R_{\bm w}^{(d_{\bm w})}$ are}\\
                        \text{generated in degree~$1$}
                    \end{gathered}
                \mright\}
            \mright|
        }
        {M^n}.
    \]
    Then, for every positive integer~$n$,
    \[\lim_{M\to+\infty}p(n, M) = 1.\]
\end{theorem}

The proof of \cref{thm:probability} shows that the following criterion works for almost all of the vectors~$\bm w$.

\begin{theorem}[{\cref{cor:bubble and generation in degree 1 - graded,cor:bubble and generation in degree 1 - Rees,thm:maximum all weights}}]
    Let $n$ be a positive integer and let $\bm w\in \mathbb{Z}_{\geq1}^n$.
    If $2d_{\bm w} > (\max(\bm w)-2)\sum_{i=1}^nw_i$, then $P_{\bm w}^{d_{\bm w}}$ and $R_{\bm w}^{d_{\bm w}}$ are generated in degree~$1$.
\end{theorem}

\begin{remark}
    The various criteria given above and in \cref{sec:tools} can be used to prove very ampleness of $\mathcal O(d_{\bm w})$ for almost all weighted projective spaces $\mathbb P_Q(\bm w)$. For the cases where these criteria do not apply, we have \cref{alg: Bubbles algorithm}. For example, for the weighted projective space $\mathbb P_\mathbb C(2, 12, 21, 49)$ that recently appeared in \cite[Theorem~1.9]{DLT26}, \cref{alg: Bubbles algorithm} shows that $\mathcal{O}(588)$ is very ample and defines a projectively normal embedding.
\end{remark}

\subsection{Overview of the paper}

In \cref{sec:relating}, we introduce the notion of a \emph{bubble}.
An example of a bubble is the vector $\bm u$ in \cref{exa:Delorme}.
We prove equivalence of generation in degree 1 and the existence of bubbles satisfying \cref{cond:combinatorial generation}, as well as very ampleness and the existence of bubbles satisfying \cref{cond:combinatorial very ampleness}.
In \cref{sec:tools}, we give sufficient conditions for the nonexistence of bubbles.
In \cref{sec:bubbles,sec:small}, we use these conditions to prove generation in degree~1. In \cref{sec:nullity}, we prove that almost no weight vectors admit a bubble.
In \cref{sec:infinitude}, we show that there are nonetheless infinitely many weight vectors admitting a bubble and we construct series of examples showing that the bounds in \cref{sec:bubbles} are sharp.

In \cref{app:algorithm}, we describe the algorithms used to compute the finite sets of vectors of weights less than \upperbound{} in \cref{sec:intro small weights}.
In \cref{app:tables}, we give the tables of vectors of weights less than \upperbound{} that admit a bubble with positive coordinates.

\begin{remark}
    The paper \cite{ABL25}, developed independently and concurrently, has proved certain similar results in the context of simplices, also discussing various relations with polytopes and almost 1-normality of rectangular simplices. In particular, slightly weaker versions of \cref{thm:introduction coprime graded,exa:sharp graded} appear in \cite{ABL25}.
\end{remark}

\section{Relating bubbles to rings and very ampleness} \label[sec]{sec:relating}

In this section, we reduce the problem of determining whether the line bundles we are interested in are very ample to a combinatorial problem.
For this, we introduce the notion of \emph{bubbles}; the rest of the paper is dedicated to thoroughly studying these bubbles, determining when they exist, and under which circumstances the existence of bubbles is an obstruction to generation in degree 1 of certain Veronese subrings or to very ampleness.

\subsection{Notation}

We remind that $Q$ is a fixed nonzero ring and $P_{\bm w}$ and $R_{\bm w}$ are defined as in the beginning of \cref{sec:intro polynomial,sec:intro Rees}, respectively, namely by equations \eqref{eqn:intro polynomial} and \eqref{eqn:Rees ring R}, respectively.

\Cref{tab:notation} contains the notation that we use throughout the paper.
Let $n$ be a positive integer and $\bm w := (w_1, w_2, \ldots, w_n)$ a vector of positive integers.

\begin{longtable}{ll}
\caption{List of notation\label{tab:notation}}\\
\toprule
\endfirsthead
\endhead
\bottomrule
\endlastfoot
$d_{\bm w}$ & The least common multiple of $w_1, \ldots, w_n$.\\
$\lfloor x\rfloor$ & The greatest integer not greater than the real number~$x$.\\
$\lceil x \rceil$ & The least integer not smaller than the real number~$x$.\\
$\abs{S}$ & The number of elements of a finite set~$S$.\\
$\mathbb Z_{\geq k}^n$ & $\{(v_1, \ldots, v_n) \in \mathbb Z^n \mid \forall i \in \{1, \ldots, n\}\colon v_i \geq k\}$.\\
$\bm a \cdot \bm b$ & The scalar product $\sum_{i \in \{1, \ldots, n\}} a_i b_i$ of\\*
    & $\bm a := (a_1, \ldots, a_n) \in \mathbb Z_{\geq0}^n$ and $\bm b := (b_1, \ldots, b_n) \in \mathbb Z_{\geq0}^n$.\\
$\preceq$ & The product order on $\mathbb Z_{\geq0}^n$, meaning that the partial\\*
    & order on $\mathbb Z_{\geq0}^n$ such that $\bm a \preceq \bm b$ holds if and only if the\\*
    & inequality $a_i \leq b_i$ holds for every $i \in \{1, \ldots, n\}$.\\
$\bm e_i$ & The unit vector in $\mathbb Z_{\geq0}^n$ with $1$ in the $i$th coordinate.\\
$\pi_i$ & The projection from $\mathbb Z_{\geq0}^n$ to $\mathbb Z_{\geq0}^{n-1}$ which drops the $i$th\\*
    & coordinate.\\
$\pi_I$ & The projection $\mathbb Z_{\geq0}^n \to \mathbb Z_{\geq0}^{n - \abs{I}}$ which drops each\\*
    & coordinate $i \in I$, where $I \subseteq \{1, \ldots, n\}$.\\
$\kappa_i$ & The inverse of the restriction of $\pi_i$ to $\{\bm u \in \mathbb Z_{\geq0}^n \mid u_i = 0\}$.\\
$\kappa_I$ & The inverse of the restriction of $\pi_I$ to\\*
    & $\{\bm u \in \mathbb Z_{\geq0}^n \mid \forall i \in I\colon u_i = 0\}$, where $I \subseteq \{1, \ldots, n\}$.\\
$V_{\bm w}(m)$ & $\mleft\{ \bm v \in \mathbb Z_{\geq0}^n \;\middle|\; \bm w \cdot \bm v = m \mright\}$, where $m$ is a positive integer.\\
$V_{\bm w}^+(m)$ & $\mleft\{ \bm v \in \mathbb Z_{\geq0}^n \;\middle|\; \bm w \cdot \bm v \geq m \mright\}$, where $m$ is a positive integer.\\
$mS$ & The Minkowski sum $S + \ldots + S$ with $m$ summands,\\*
    & meaning the set of vectors $\bm s_1 + \ldots + \bm s_m$ in $\mathbb Z_{\geq0}^n$ with\\*
    & each $\bm s_i$ in $S$, where $m$ is a positive integer and $S \subseteq \mathbb Z_{\geq0}^n$.
\end{longtable}

\begin{definition}
A \textbf{\bubble{$\bm w$}} is an element $\bm u$ of $\mathbb Z_{\geq0}^n$ such that the inequality $\bm w \cdot \bm u \geq 2 d_{\bm w}$ holds and there is no vector $\bm v \in \mathbb Z_{\geq0}^n$ satisfying $\bm w \cdot \bm v = d_{\bm w}$ and $\bm v \preceq \bm u$.
\end{definition}

We define the following two conditions for every triple $(\bm w, \bm u, k)$ where $\bm u\in\mathbb{Z}_{\geq 0}^n$ and $k$ is a positive integer:

\begin{condition} \label{cond:combinatorial generation}
There exist vectors $\bm v_1, \ldots, \bm v_r \in V_{\bm w}(d_{\bm w})$, where $m := \max(1, \lceil \lfloor\frac{\bm w \cdot \bm u}{d_{\bm w}}\rfloor / k \rceil)$
and $r := mk - \lfloor\frac{\bm w \cdot \bm u}{d_{\bm w}}\rfloor$, such that
\begin{equation} \label{eqn:combinatorial generation}
\bm u + \bm v_1 + \ldots + \bm v_r \notin m V_{\bm w}^+(k d_{\bm w}).
\end{equation}
\end{condition}

\begin{condition} \label{cond:combinatorial very ampleness}
There exists $i \in \{1, \ldots, n\}$ such that for all positive integers~$m$, we have
\begin{equation} \label{eqn:combinatorial very ampleness}
\bm u + \frac{(km - \lfloor\frac{\bm w \cdot \bm u}{d_{\bm w}}\rfloor)d_{\bm w}}{w_i} \bm e_i \notin mV_{\bm w}^+(k d_{\bm w}).
\end{equation}
\end{condition}

We will see in \cref{thm:bubble and generation in degree 1 - graded,thm:bubble and generation in degree 1 - Rees} that bubbles satisfying \cref{cond:combinatorial generation} are an obstruction to generation in degree 1 and we will see in \cref{thm:bubble satisfying condition implies not very ample - graded,thm:bubble satisfying condition implies not very ample - Rees} that bubbles satisfying \cref{cond:combinatorial very ampleness} are an obstruction to very ampleness.

\begin{remark} \label{rem:conditions}
We make the following remarks on \cref{cond:combinatorial generation,cond:combinatorial very ampleness}:
\begin{enumerate}[label=\textup{(\alph*)}, ref=\alph*]
    \item $(\bm w, \bm u, k)$ satisfies \cref{cond:combinatorial very ampleness} if and only if for every positive integer~$m$, there exist $\bm v_1, \ldots, \bm v_r \in V_{\bm w}(d_{\bm w})$, where $r := mk - \lfloor\frac{\bm w \cdot \bm u}{d_{\bm w}}\rfloor$, such that \eqref{eqn:combinatorial generation} holds,
    \item \label{itm:trivial} if $(\bm w, \bm u, k)$ satisfies \cref{cond:combinatorial very ampleness}, then it also satisfies \cref{cond:combinatorial generation},
    \item \label{itm:obstruction weight} if $(\bm w, \bm u, k)$ satisfies \cref{cond:combinatorial generation} or~\labelcref{cond:combinatorial very ampleness}, then $\bm w \cdot \bm u \geq (k+1)d_{\bm w}$ and $\bm u \notin \lfloor\frac{\bm w \cdot \bm u}{d_{\bm w}}\rfloor V_{\bm w}^+(d_{\bm w})$,
    \item \label{itm:small m} if $\bm u$ is a \bubble{$\bm w$}, then \eqref{eqn:combinatorial very ampleness} automatically holds for all $i \in \{1, \ldots, n\}$ and positive integers $m < \lfloor\frac{\bm w \cdot \bm u}{d_{\bm w}}\rfloor / k$,
    \item if $\bm u$, $k$ and $i$ are such that \eqref{eqn:combinatorial very ampleness} holds for all positive integers $m$ that are at most $\max(1,\, \sum_{j \in \{1, \ldots, n\} \setminus \{i\}} u_j)$, then \eqref{eqn:combinatorial very ampleness} holds for all positive integers~$m$.
\end{enumerate}
\end{remark}

\Cref{thm:a sufficient condition for condition 2.2} will give a criterion for satisfying \cref{cond:combinatorial very ampleness} when $k = 1$, the most important case.

\subsection{Relating generation in degree 1 to the existence of bubbles}

In the following list of results, we relate the existence of bubbles to a Veronese subring being generated in degree 1.

\begin{proposition}\label{thm:bubble and generation in degree 1 - graded}
Let $k$ be a positive integer.
Then, $P_{\bm w}^{(kd_{\bm w})}$ is generated in degree $1$ if and only if there is no \bubble{$\bm w$}~$\bm u$ with $d_{\bm w} \mid \bm w \cdot \bm u$ such that $(\bm w, \bm u, k)$ satisfies \cref{cond:combinatorial generation}.
\end{proposition}

\begin{proof}
    In order to lighten the notation, denote $P := P_{\bm w}$. For every vector $\bm a \in \mathbb Z_{\geq0}^n$, denote
    \[
    \bm x^{\bm a} := x_1^{a_1} \cdot \ldots \cdot x_n^{a_n} \in P.
    \]

    If the subring $P^{(kd_{\bm w})}$ is not generated in degree~$1$, then there exists a vector $\bm a\in\mathbb{Z}_{\geq0}^n$ satisfying $\bm w\cdot\bm a = mkd_{\bm w}$ for some integer $m\geq 2$ such that $\bm x^{\bm a}\in P_{mkd_{\bm w}}\setminus P_{ kd_{\bm w}}^m$. Without loss of generality, such a vector can be chosen such that there exists no vector $\bm v\in V_{\bm w}(kd_{\bm w})$ such that $\bm v\preceq \bm a$.

    If the vector $\bm a$ is a \bubble{$\bm w$}, then the triple $(\bm w, \bm a, k)$ trivially satisfies \cref{cond:combinatorial generation}, by definition of $\bm a$. Otherwise, we can find a sequence $\bm v_1,\ldots,\bm v_N\in V_{\bm w}(d_{\bm w})$ of maximal length such that $\sum_{i=1}^N\bm v_i\preceq \bm a$. By the choice of $\bm a$, the inequality $N\leq k-1$ necessarily holds. It follows that the vector defined as $\bm u := \bm a-\sum_{i=1}^N\bm v_i\in\mathbb{Z}_{\geq0}^n$ is a \bubble{$\bm w$} which satisfies
    \[\mleft\lfloor \frac{\bm w\cdot \bm u}{d_{\bm w}}\mright\rfloor = mk - N > (m-1)k.\]
    In particular, the triple $(\bm w, \bm u, k)$ satisfies \cref{cond:combinatorial generation}.

    Conversely, suppose that there exists a \bubble{$\bm w$}~$\bm u$ with $d_{\bm w} \mid \bm w \cdot \bm u$ such that $(\bm w, \bm u, k)$ satisfies \cref{cond:combinatorial generation}. Then, from the fact that $l := \frac{\bm w\cdot \bm u}{d_{\bm w}}$ is an integer, we know that there exist vectors $\bm v_1,\ldots,\bm v_r\in V_{\bm w}(d_{\bm w})$, where $r = mk-l$ and $m = \lceil \frac{l}{k}\rceil$, satisfying
    \[\bm a := \bm u+\bm v_1+\ldots +\bm v_r\in V_{\bm w}(mkd_{\bm w})\setminus mV_{\bm w}(kd_{\bm w}).\]
    Hence by the existence of $\bm a$, the Veronese subring $P^{(kd_{\bm w})}$ cannot be generated in degree~$1$. This concludes the proof.
\end{proof}

\begin{corollary} \label{cor:bubble and generation in degree 1 - graded}
Let $k$ be a positive integer.
If no \bubble{$\bm w$} $\bm u$ satisfies both $\bm w \cdot \bm u \geq (k+1) d_{\bm w}$ and $d_{\bm w} \mid \bm w \cdot \bm u$, then $P_{\bm w}^{(kd_{\bm w})}$ is generated in degree~$1$. The converse holds if $k = 1$.
\end{corollary}

\begin{proof}
    The first part follows from \cref{thm:bubble and generation in degree 1 - graded} and \cref{rem:conditions}\labelcref{itm:obstruction weight}. For the last statement, suppose that $P_{\bm w}^{(d_{\bm w})}$ is generated in degree~$1$. It follows that any vector $\bm u\in\mathbb{Z}_{\geq0}^n$ with $\bm w\cdot \bm u = md_{\bm w}$ for some $m\geq 1$ lies in $mV_{\bm w}(d_{\bm w})$. Hence none of the \bubbles{$\bm w$}~$\bm u$ satisfy $d_{\bm w} \mid \bm w \cdot \bm u$.
\end{proof}

\begin{proposition}\label{thm:bubble and generation in degree 1 - Rees}
Let $k$ be a positive integer.
Then, $R_{\bm w}^{(kd_{\bm w})}$ is generated in degree $1$ if and only if there is no \bubble{$\bm w$}~$\bm u$ such that $(\bm w, \bm u, k)$ satisfies \cref{cond:combinatorial generation}.
\end{proposition}

\begin{proof}
    Assume that $R_{\bm w}^{(kd_{\bm w})}$ is not generated in degree~$1$. Similarly to the proof of \cref{thm:bubble and generation in degree 1 - graded}, one can show the existence of a vector $\bm a\in \mathbb{Z}_{\geq0}^n$ and an integer $m\geq 2$ satisfying $\bm w\cdot\bm a\geq mkd_{\bm w}$ and $\bm a\notin mV_{\bm w}^+(kd_{\bm w})$, and such that there exists no vector $v\in V_{\bm w}(kd_{\bm w})$ satisfying $\bm v\preceq \bm a$. If $\bm a$ is a \bubble{$\bm w$}, we set $\bm u := \bm a$ and $N := 0$. Otherwise, we find a sequence of vectors $\bm v_1,\ldots, \bm v_N\in V_{\bm w}(d_{\bm w})$, where $1\leq N\leq k-1$, such that the vector $\bm u := \bm a-\sum_{i=1}^N\bm v_i$ is a \bubble{$\bm w$}. In both cases, the following inequalities hold:
    \[\bm w \cdot \bm u \geq mkd_{\bm w}-Nd_{\bm w} > (m-1)kd_{\bm w},\]
    and the triple $(\bm w, \bm u, k)$ satisfies \cref{cond:combinatorial generation}.
    The rest of the proof follows similarly to that of \cref{thm:bubble and generation in degree 1 - graded}.
\end{proof}

\begin{corollary} \label{cor:bubble and generation in degree 1 - Rees}
Let $k$ be a positive integer.
If no \bubble{$\bm w$} $\bm u$ satisfies both $\bm w \cdot \bm u \geq (k+1) d_{\bm w}$ and $\bm u \notin \lfloor\frac{\bm w \cdot \bm u}{d_{\bm w}}\rfloor V_{\bm w}^+(d_{\bm w})$, then $R_{\bm w}^{(kd_{\bm w})}$ is generated in degree~$1$.
The converse holds if $k = 1$.
\end{corollary}

\begin{proof}
The first part follows from \cref{thm:bubble and generation in degree 1 - Rees} and \cref{rem:conditions}\labelcref{itm:obstruction weight}. For the last statement, we simply remark that for any vector $\bm u\in\mathbb{Z}_{\geq0}^n$ with $\bm w\cdot \bm u\geq d_{\bm w}$, the triple $(\bm w, \bm u, 1)$ satisfies \cref{cond:combinatorial generation} if and only if $\bm u \notin \lfloor\frac{\bm w \cdot \bm u}{d_{\bm w}}\rfloor V_{\bm w}^+(d_{\bm w})$. Note that by definition, any \bubble{$\bm w$}~$\bm u$ satisfies $\bm w\cdot \bm u \geq 2d_{\bm w}$.
\end{proof}

In this section, we connected generation in degree 1 and the existence of bubbles.
Before we relate very ampleness and bubbles, we state some general results on very ampleness.

\subsection{Very ampleness} \label{sec:schemes}

We give some basic results on very ampleness in a general setting.

\begin{remark}
    In this paper, we use the definition of relative very ampleness of an invertible sheaf given in \cite[Definition \href{https://stacks.math.columbia.edu/tag/01VM}{01VM}]{Stacks}.
    Note that this differs from \cite[definition above Remark~II.5.16.1]{Har77}.
    In particular, by \cite[Exercise II.7.14]{Har77}, \cref{thm:ample} would not hold using the definition of relative very ampleness in \cite{Har77}.
    The two definitions coincide when the base is affine and the schemes are Noetherian.
\end{remark}

\Cref{thm:ample} gives a characterisation of very ampleness that is useful for our purposes, namely relating very ampleness to the generation in degree 1 of a certain graded algebra.

\begin{proposition} \label{thm:ample}
Let $\varphi \colon Y \to X$ be a proper morphism of schemes and let $\mathcal L$ be an invertible $\mathcal O_Y$-module. Then, all of the following hold:
\begin{enumerate}[label=\textup{(\alph*)}, ref=\alph*]
\item \label{itm:scheme ample} $\mathcal L$ is $\varphi$-ample if and only if $Y$ is isomorphic over $X$ to $\operatorname{Proj}_X \mathcal R$, where $\mathcal R := \bigoplus_{m \in \mathbb Z_{\geq0}} \varphi_* (\mathcal L^{\otimes m})$,
\item \label{itm:scheme very ample} $\mathcal L$ is $\varphi$-very ample if and only if $Y$ is isomorphic over $X$ to $\operatorname{Proj}_X \mathcal S$, where $\mathcal S$ is the $\mathbb Z_{\geq0}$-graded $\mathcal O_X$-subalgebra of $\mathcal R$ generated over $\mathcal O_X$ by~$\varphi_* \mathcal L$,
\item \label{itm:scheme localisation} $\mathcal L$ is $\varphi$-very ample if and only if for every affine open $U \subseteq X$, we have $\operatorname{Proj} \mathcal R(U) = \bigcup_{f \in \mathcal R_1(U)} \operatorname{Spec} \mathcal R(U)_{(f)}$ and the inclusions $\iota_{(f)}\colon \mathcal S(U)_{(f)} \to \mathcal R(U)_{(f)}$ are surjective for all $f \in \mathcal R_1(U)$, where $\iota\colon \mathcal S \to \mathcal R$ is the inclusion map,
\item \label{itm:scheme ideal sheaf} if $\mathcal L$ is an $\mathcal O_Y$-submodule of the sheaf of rational functions~$\mathcal K_Y$, then we have a natural isomorphism of $\mathbb Z_{\geq0}$-graded $\mathcal O_X$-algebras $\mathcal R \to \mathcal O_Y \oplus \bigoplus_{m \in \mathbb Z_{\geq1}} \varphi_* (\mathcal L^m)$, where the product on the right-hand side is the product in~$\mathcal K_Y$.
\end{enumerate}
\end{proposition}

\begin{proof}
\labelcref*{itm:scheme ample}
``$\Longrightarrow$''. This is \cite[Lemma \href{https://stacks.math.columbia.edu/tag/0C6J}{0C6J}]{Stacks}.

``$\Longleftarrow$''. It suffices to prove the case where $X$ is affine. By \cite[Lemma \href{https://stacks.math.columbia.edu/tag/01K4}{01K4}]{Stacks}, $Y$ is quasi-compact over $X$. The scheme $Y$ is covered by affine opens $D_+(f)$ where
\[
f \in H^0(X, \varphi_*(\mathcal L^{\otimes m})) = H^0(Y, \mathcal L^{\otimes m})
\]
and $m$ is a positive integer.
By \cite[Remark~13.46(4)]{GW20}, for each such $f$ we have an equality $D_+(f) = Y_f$.
By \cite[Definition \href{https://stacks.math.columbia.edu/tag/01PS}{01PS}]{Stacks}, it follows that $\mathcal L$ is $\varphi$-ample.

\labelcref*{itm:scheme very ample}
``$\Longrightarrow$''.
By \cite[Lemma \href{https://stacks.math.columbia.edu/tag/01VR}{01VR}]{Stacks}, we obtain an immersion $Y \to \mathbb P(\varphi_* \mathcal L)$ which corresponds to the $\mathbb Z_{\geq0}$-graded $\mathcal O_X$-algebra homomorphism $\operatorname{Sym} \varphi_* \mathcal L \to \mathcal R$.
Since $\varphi$ is proper, $Y \to \mathbb P(\varphi_* \mathcal L)$ is a closed immersion.
By \cite[Lemma \href{https://stacks.math.columbia.edu/tag/01R8}{01R8}]{Stacks}, its scheme-theoretic image is $\operatorname{Proj}_X \mathcal S$.

``$\Longleftarrow$''.
It suffices to prove the case where $X$ is affine.
The surjective graded $\mathcal O_X$-algebra homomorphism $\operatorname{Sym}(\varphi_* \mathcal L) \to \mathcal S$ induces a closed embedding $\iota\colon \operatorname{Proj}_X(\mathcal S) \to \mathbb P(\varphi_* \mathcal L)$ over~$X$.
Let $S$ denote the graded $\mathcal O_X(X)$-algebra $\operatorname{Sym} H^0(Y, \mathcal L)$.
There exists a quasi-coherent saturated homogeneous ideal $I$ of $S$ such that $\operatorname{Proj}_X \mathcal S$ is isomorphic over $X$ to $\operatorname{Proj} (S/I)$.
By \cite[Lemma \href{https://stacks.math.columbia.edu/tag/01N1}{01N1}]{Stacks}, the equality $\iota^*(\mathcal O_{\operatorname{Proj}(S)}(1)) = \mathcal O_{\operatorname{Proj}(S/I)}(1)$ holds.
Since $Y = \operatorname{Proj}_X(\mathcal S)$ is covered by affine opens $D_+(f)$ for $f \in H^0(Y, \mathcal L)$, the invertible sheaf $\mathcal L$ is generated by its global sections.
The invertible sheaf $\mathcal O_Y(1)$ is also generated by its global sections, which coincide with the global sections of~$\mathcal L$, showing that $\mathcal O_Y(1) = \mathcal L$.
So, $\mathcal L$ is $\varphi$-very ample.

\labelcref*{itm:scheme localisation}
Follows from \labelcref{itm:scheme ample,itm:scheme very ample} since the morphism $\operatorname{Proj}_X \mathcal R \to \operatorname{Proj}_X \mathcal S$ from the proof of direction ``$\Longrightarrow$'' of part \labelcref{itm:scheme very ample} is induced by the inclusion $\iota\colon \mathcal S \to R$.

\labelcref*{itm:scheme ideal sheaf}
Since $\mathcal L$ is locally principal, for every positive integer~$m$, every stalk $(\mathcal L^{\otimes m})_P$ of $\mathcal L^{\otimes m}$ is generated by $f \cdot 1^{\otimes m}$ over $\mathcal O_{Y, P}$ for some $f \in \mathcal K_{Y, P}$. So, we have a natural map between the stalks of $\mathcal L^{\otimes m}$ and $\mathcal L^m$, sending $f \cdot 1^{\otimes m}$ to~$f$. This induces the map $\mathcal R \to \mathcal O_Y \oplus \bigoplus_{m \in \mathbb Z_{\geq1}} \varphi_* (\mathcal L^m)$.
\end{proof}

The following useful \lcnamecref{thm:very ample} shows that generation in degree 1 implies very ampleness.

\begin{corollary} \label{thm:very ample}
Let $\varphi \colon Y \to X$ be a proper morphism of schemes and let $\mathcal L$ be a $\varphi$-ample invertible $\mathcal O_Y$-module such that $\bigoplus_{m \in \mathbb Z_{\geq0}} \varphi_* (\mathcal L^{\otimes m})$ is generated in degree~$1$. Then, $\mathcal L$ is $\varphi$-very ample.
\end{corollary}

\begin{proof}
Follows by applying \cref{thm:ample} parts \labelcref{itm:scheme ample} and~\labelcref{itm:scheme very ample} and by using that the sheaves of algebras $\mathcal R$ and $\mathcal S$ coincide.
\end{proof}

\subsection{Relating very ampleness to the existence of bubbles}

In this section, we relate very ampleness to \cref{cond:combinatorial very ampleness}.

\begin{proposition}\label{thm:bubble satisfying condition implies not very ample - graded}
    Let $k$ be a positive integer. Then, the following are equivalent:
    \begin{enumerate}[label=\textup{(\roman*)}, ref=\roman*]
        \item\label{itm:condition on all bubbles} for every \bubble{$\bm w$}~$\bm u$ with $d_{\bm w}\ |\ \bm w\cdot\bm u$, the triple $(\bm w, \bm u, k)$ does not satisfy \cref{cond:combinatorial very ampleness},
        \item\label{itm:statement very ampleness graded} the line bundle $\mathcal O_{\mathbb{P}_Q(\bm w)}(kd_{\bm w})$ is very ample over $\operatorname{Spec}Q$.
    \end{enumerate}
\end{proposition}
\begin{proof}
    Part~\labelcref{itm:condition on all bubbles} is equivalent to the following statement:
    \begin{itemize}
        \item[\textdaggerdbl{}] for every $\bm u\in\mathbb{Z}_{\geq 0}^n$  with $d_{\bm w}\ |\ \bm w\cdot\bm u$, the triple $(\bm w, \bm u, k)$ does not satisfy \cref{cond:combinatorial very ampleness}.
    \end{itemize}
    To see this, note that for every vector $\bm u \in \mathbb Z_{\geq0}^n$ and integer $i \in \{1, \ldots, n\}$, if there exists a positive integer $m$ such that
    \[
    \bm u + \frac{(km - \lfloor\frac{\bm w \cdot \bm u}{d_{\bm w}}\rfloor)d_{\bm w}}{w_i} \bm e_i \in m V_{\bm w}^+(k d_{\bm w}),
    \]
    then for all positive integers $r$ and vectors $\bm v_1, \ldots, \bm v_r \in V_{\bm w}(d_{\bm w})$, we have
    \[
    \bm u' + \frac{(km' - \lfloor\frac{\bm w \cdot \bm u'}{d_{\bm w}}\rfloor)d_{\bm w}}{w_i} \bm e_i \in m' V_{\bm w}^+(k d_{\bm w}),
    \]
    where $\bm u' := \bm u + \bm v_1 + \ldots + \bm v_r$ and $m' := m + \lceil\frac{r}{k}\rceil$.

    Now, we show that \textdaggerdbl{} and part~\labelcref{itm:statement very ampleness graded} of the statement are equivalent. For the rest of the proof, let us denote $P := P_{\bm w}$.
    Let moreover $S := Q\oplus\bigoplus_{m\geq 1}P_{kd_{\bm w}}^m$ be the subring of $P^{(kd_{\bm w})}$ generated in degree~$1$.
    According to \cref{thm:ample}\labelcref{itm:scheme localisation}, the line bundle $\mathcal O_{\mathbb{P}_Q(\bm w)}(kd_{\bm w})$ is very ample over $\operatorname{Spec}Q$ if and only if the local rings $S_{(x_i)}$ and $P_{(x_i)}$ are equal for all $i\in\{1,\ldots, n\}$.
    Fix $i \in \{1, \ldots, n\}$.
    For every $\bm u \in \mathbb Z_{\geq0}^n$, let $o_{\bm u}$ denote $\lfloor\frac{\bm w \cdot \bm u}{d_{\bm w}}\rfloor$.
    Let $U$ denote the set of $\bm u\in \mathbb{Z}_{\geq 0}^n$ such that $\bm w\cdot\bm u = o_{\bm u} d_{\bm w}$.
    For every $\bm u \in U$, let $p_{\bm u}$ denote the element $\displaystyle\frac{\bm x^{\bm u}}{x_i^{o_{\bm u}d_{\bm w}/w_i}}$ of $P_{(x_i)}$.
    Note that $P_{(x_i)}$ is generated as a $Q$-module by the elements $p_{\bm u}$ where $\bm u \in U$.
    The equality $S_{(x_i)} = P_{(x_i)}$ is equivalent to the containment $p_{\bm u} \in S_{(x_i)}$ for all $\bm u \in U$.
    We have $p_{\bm u} \in S_{(x_i)}$ if and only if there exist a positive integer $m$ and vectors $\bm a_1,\ldots, \bm a_m\in V_{\bm w}(kd_{\bm w})$ such that the following equality holds:
    \[\frac{\bm x^{\bm u}}{x_i^{o_{\bm u}d_{\bm w}/w_i}} = \frac{\bm x^{\bm a_1}\bm x^{\bm a_2}\cdots \bm x^{\bm a_m}}{x_i^{kmd_{\bm w}/w_i}}\in P_{(x_i)}.\]
    The previous equality is equivalent to:
    \[\bm u+\frac{(km-o_{\bm u})d_{\bm w}}{w_i}\bm e_i = \bm a_1+\ldots+\bm a_m.\]
    Thus, the tuple $(\bm w, \bm u, k)$ satisfies \cref{cond:combinatorial very ampleness} for the index~$i$ if and only if $p_{\bm u}$ does not belong to $S_{(x_i)}$.
    The result follows.
\end{proof}

\begin{proposition}\label{thm:bubble satisfying condition implies not very ample - Rees}Let $k$ be a positive integer. Then, the following are equivalent:
    \begin{enumerate}[label=\textup{(\roman*)}, ref=\roman*]
        \item for every \bubble{$\bm w$}~$\bm u$, the triple $(\bm w, \bm u, k)$ does not satisfy \cref{cond:combinatorial very ampleness},
        \item the line bundle $\mathcal O_{\operatorname{Proj} R_{\bm w}}(kd_{\bm w})$ is very ample over $\mathbb{A}^n_Q$.
    \end{enumerate}
\end{proposition}

\begin{proof}
    Similar to the proof of \cref{thm:bubble satisfying condition implies not very ample - graded}.
\end{proof}

We have shown how to relate the obstruction of generation in degree 1 with the combinatorics of bubbles.
In this paper, all the results on generation in degree 1, respectively not very ampleness, rely on \cref{cor:bubble and generation in degree 1 - graded,cor:bubble and generation in degree 1 - Rees}, respectively on \cref{thm:bubble satisfying condition implies not very ample - graded,thm:bubble satisfying condition implies not very ample - Rees}.
In the next section, we discuss how the existence of bubbles and \cref{cond:combinatorial generation,cond:combinatorial very ampleness} change under the equivalences defined in \cref{def:equivalences}.

\section{Bubbles and very ampleness under equivalences}

\subsection{Existence of bubbles under equivalences}

We show that the existence of bubbles with weight divisible by~$d$ (resp.\ any bubbles) is stable under weak (resp.\ strong) equivalence:

\begin{lemma} \label{lem:simplification bubbles}
Let $n, m, k$ be positive integers. Let $\bm v \in \mathbb Z_{\geq1}^m$. Let $\bm t$ be a \bubble{$\bm v$}. Then, both of the following hold:
\begin{enumerate}[label=\textup{(\alph*)}, ref=\alph*]
\item \label{itm:weakly} if $d_{\bm v}$ divides $\bm v \cdot \bm t$ and $\bm v$ can be obtained from $\bm w \in \mathbb Z_{\geq1}^n$ using operations (1)–(5) of \cref{def:equivalences} and the operation of removing a weight dividing the least common multiple of the other weights, then there exists a \bubble{$\bm w$} $\bm u$ with $\bm w \cdot \bm u = \frac{\bm v \cdot \bm t}{d_{\bm v}} d_{\bm w}$,
\item \label{itm:strongly} if $\bm v$ can be obtained from $\bm w \in \mathbb Z_{\geq1}^n$ using operations (1)–(4) of \cref{def:equivalences} and the operation of removing a weight dividing the least common multiple of the other weights, then there exists a \bubble{$\bm w$} $\bm u$ with $\bm w \cdot \bm u = \frac{\bm v \cdot \bm t}{d_{\bm v}} d_{\bm w}$.
\end{enumerate}
\end{lemma}

\begin{proof}
Denote $\bm w = (w_1, \ldots, w_n)$. Part~\labelcref*{itm:strongly} follows from the implications and equivalences below:
\begin{enumerate}
    \item if $\sigma$ is a permutation on $n$ symbols, then $\bm u$ is a \bubble{$\bm w$} if and only if $\sigma(\bm u)$ is a \bubble{$\sigma(\bm w)$},
    \item if $\widehat{\bm w} := (w_1, \ldots, w_{n-1}, w_n, w_n)$, then $\widehat{\bm u}:=(u_1, \ldots, u_{n-1}, u_n, u_{n+1})$ is a \bubble{$\widehat{\bm w}$} if and only if $(u_1, \ldots, u_{n-1}, u_n + u_{n+1})$ is a \bubble{$\bm w$},
    \item if $\widehat{\bm w} = (g w_1, \ldots, g w_n)$ for some positive integer~$g$, then $\bm u$ is a \bubble{$\bm w$} if and only if $\bm u$ is a \bubble{$\widehat{\bm w}$},
    \item if $w_i = \operatorname{lcm}(\bm w)$, then every \bubble{$\bm w$} $\bm u$ satisfies $u_i = 0$,
    \item[$(*)$] if $\widehat{\bm w} = (w_1, \ldots, w_n, f)$ with $f \mid d_{\bm w}$ and $\bm u$ is a \bubble{$\bm w$}, then $\kappa_{n+1}(\bm u)$ is a \bubble{$\widehat{\bm w}$},
\end{enumerate}
To see~\labelcref*{itm:weakly}, we additionally note the following equivalence and implication:
\begin{enumerate}[resume]
    \item if $\widehat{\bm w} = (g w_1, \ldots, g w_{n-1}, w_n)$ for some positive integer $g$ coprime to~$w_n$, $n$ is at least~$2$ and $\gcd(\bm w)$ is equal to~$1$, then both of the following hold:
    \begin{enumerate}[label=\textup{(\roman*)}, ref=\roman*]
        \item $\bm u$ is a \bubble{$\bm w$} if and only if $(u_1, \ldots, u_{n-1}, g u_n)$ is a \bubble{$\widehat{\bm w}$},
        \item if $\bm u'$ is a \bubble{$\widehat{\bm w}$} satisfying $d_{\widehat{\bm w}} \mid \widehat{\bm w} \cdot \bm u'$, then $g \mid u_n'$.\qedhere
    \end{enumerate}
\end{enumerate}
\end{proof}

If we also require \cref{cond:combinatorial generation} or \cref{cond:combinatorial very ampleness}, then we have weaker statements:

\begin{lemma} \label{lem:condition}
Let $n, m, k$ be positive integers. Let $\bm v \in \mathbb Z_{\geq1}^m$. Let $\bm t$ be a \bubble{$\bm v$} such that $(\bm v, \bm t, k)$ satisfies \cref{cond:combinatorial generation}, respectively \cref{cond:combinatorial very ampleness}. Then, both of the following hold:
\begin{enumerate}[label=\textup{(\alph*)}, ref=\alph*]
\item \label{itm:condition weakly} if $d_{\bm v}$ divides $\bm v \cdot \bm t$ and $\bm v$ can be obtained from $\bm w \in \mathbb Z_{\geq1}^n$ using operations (1)–(3) and (5) of \cref{def:equivalences} and the operation of removing a weight dividing the least common multiple of the other weights, then there exists a \bubble{$\bm w$} $\bm u$ with $\bm w \cdot \bm u = \frac{\bm v \cdot \bm t}{d_{\bm v}} d_{\bm w}$ such that $(\bm w, \bm u, k)$ satisfies \cref{cond:combinatorial generation}, respectively \cref{cond:combinatorial very ampleness},
\item \label{itm:condition strongly} if $\bm v$ can be obtained from $\bm w \in \mathbb Z_{\geq1}^n$ using operations (1)–(3) of \cref{def:equivalences} and the operation of removing a weight dividing the least common multiple of the other weights, then there exists a \bubble{$\bm w$} $\bm u$ with $\bm w \cdot \bm u = \frac{\bm v \cdot \bm t}{d_{\bm v}} d_{\bm w}$ such that $(\bm w, \bm u, k)$ satisfies \cref{cond:combinatorial generation}, respectively \cref{cond:combinatorial very ampleness}.
\end{enumerate}
\end{lemma}

\begin{proof}
Similarly to the proof of \cref{lem:simplification bubbles}, consider each of the operations separately.
\end{proof}

For bubbles with positive coordinates, we obtain even weaker statements:

\begin{lemma} \label{lem:positive equivalence}
Let $n, m, k$ be positive integers. Let $\bm w \in \mathbb Z_{\geq1}^n$. Let $\bm u$ be a \bubble{$\bm w$} with positive coordinates. Then, both of the following hold:
\begin{enumerate}[label=\textup{(\alph*)}, ref=\alph*]
\item if $d_{\bm w}$ divides $\bm w \cdot \bm u$ and $\bm v\in\mathbb{Z}_{\geq1}^{m}$ can be obtained from $\bm w\in\mathbb{Z}_{\geq1}^n$ using operations (1), (3) and (5) of \cref{def:equivalences}, the operation of removing a weight equal to another weight (this is one direction of operation~(2)) and the operation of removing a weight equal to the least common multiple of the other weights (this is one direction of operation~(4)), then there exists a \bubble{$\bm v$} $\bm t$ with positive coordinates satisfying $\bm v \cdot \bm t = \frac{\bm w \cdot \bm u}{d_{\bm w}} d_{\bm v}$ and moreover, if $(\bm w, \bm u, k)$ satisfies \cref{cond:combinatorial generation}, respectively \cref{cond:combinatorial very ampleness}, then we can choose the $\bm v$ and $\bm t$ such that $(\bm v, \bm t, k)$ satisfies \cref{cond:combinatorial generation}, respectively \cref{cond:combinatorial very ampleness},
\item if $\bm v\in\mathbb{Z}_{\geq1}^{m}$ can be obtained from $\bm w\in\mathbb{Z}_{\geq1}^n$ using operations (1) and (3) of \cref{def:equivalences}, the operation of removing a weight equal to another weight (this is one direction of operation~(2)) and the operation of removing a weight equal to the least common multiple of the other weights (this is one direction of operation~(4)), then there exists a \bubble{$\bm v$} $\bm t$ with positive coordinates satisfying $\bm v \cdot \bm t = \frac{\bm w \cdot \bm u}{d_{\bm w}} d_{\bm v}$ and moreover, if $(\bm w, \bm u, k)$ satisfies \cref{cond:combinatorial generation}, respectively \cref{cond:combinatorial very ampleness}, then we can choose the $\bm v$ and $\bm t$ such that $(\bm v, \bm t, k)$ satisfies \cref{cond:combinatorial generation}, respectively \cref{cond:combinatorial very ampleness}.
\end{enumerate}
\end{lemma}

\begin{proof}
Similar to \cref{lem:simplification bubbles,lem:condition}.
\end{proof}

Both weak and strong equivalence classes have a unique minimal representative.

\begin{lemma} \label{lem:representative}
Let $n$ be a positive integer and let $\bm w\in\mathbb{Z}_{\geq1}^n$. Then, both of the following hold:
\begin{enumerate}[label=\textup{(\alph*)}, ref=\alph*]
\item \label{itm:weak representative} the vector $\bm w \in \mathbb Z_{\geq1}^n$ is weakly equivalent to a unique strictly increasing well-formed vector $\bm v \in \mathbb Z_{\geq1}^m$ such that $m = 1$ or $m>1$ and $v_m \neq \operatorname{lcm}(v_1, v_2, \ldots, v_{m-1})$,
\item \label{itm:strong representative} the vector $\bm w \in \mathbb Z_{\geq1}^n$ is strongly equivalent to a unique strictly increasing vector $\bm v \in \mathbb Z_{\geq1}^{m}$ with $\gcd(\bm v) = 1$ such that $m = 1$ or $m>1$ and $v_m \neq \operatorname{lcm}(v_1, v_2, \ldots, v_{m-1})$.
\end{enumerate}
\end{lemma}

\begin{proof}
We say two vectors in $\mathbb Z_{\geq0}^n \setminus \{\bm 0\}$ are \emph{permutation-equivalent} if they are the same up to permutation. We say two vectors of nonnegative integers, neither equal to a zero vector, are \emph{permutation-zero-equivalent} if they are the same up to permutation and removing the zero entries.
A \emph{reducer} is one of the following operators on vectors $\bm w \in \mathbb Z_{\geq0}^n \setminus \{\bm 0\}$, where $i \in \{1, \ldots, n\}$ and $g$ and $h$ are positive integers such that $g$ divides $\gcd(\bm w)$ and $h$ divides $\gcd(\pi_i(\bm w))$:
\begin{align*}
Z_i\colon \bm w & \mapsto \mleft\{\begin{aligned}
    & (\kappa_i \circ \pi_i)(\bm w) && \begin{aligned}
        & \text{if there exists $j \neq i$ such that $w_i = w_j$ or $i$}\\
        & \text{is the unique index such that $w_i$ is positive}\\
        & \text{and equal to the least common multiple}\\
        & \text{of all the other positive entries},
    \end{aligned}\\
    & \bm w && \text{otherwise},
\end{aligned}\mright.\\
G_g\colon \bm w & \mapsto \frac{\bm w}{g},\\
W_{i, h}\colon \bm w & \mapsto \mleft\{\begin{aligned}
    & \bm w && \text{if $w_i = 0$ or $\pi_i(\bm w) = \bm 0$},\\
    & \kappa_i\mleft(\frac{\gcd(\bm w)}{h}\pi_i(\bm w)\mright) + w_i \bm e_i && \text{otherwise}.
\end{aligned}\mright.
\end{align*}
We denote also $G := G_{\gcd(\bm w)}$ and $W_{i} := W_{i, \gcd(\pi_i(\bm w))}$.
Reducers together with their inverses generate every operation (1)--(5) in \cref{def:equivalences} up to permutation-zero-equivalence.

A \emph{column} is a sequence $\bm C := (\bm C_1, \ldots, \bm C_r)$ of pairwise different elements of $\mathbb Z_{\geq0}^n \setminus \{\bm 0\}$ such that both of the following hold:
\begin{enumerate}
\item $\bm C_r$ is permutation-zero-equivalent to a strictly increasing well-formed vector $\bm v \in \mathbb Z_{\geq1}^m$ for some positive integer $m$ such that $m = 1$ or $v_m \neq \operatorname{lcm}(v_1, v_2, \ldots, v_{m-1})$,
\item for every $k \in \{1, \ldots, r-1\}$, there exists a unique reducer $O_k$ such that $\bm C_{k+1} = O_k(\bm C_k)$ and $O_k$ is one of $Z_i, G, W_i$ for some~$i$.
\end{enumerate}
Given a column $\bm C := (\bm C_1, \ldots, \bm C_r)$, for every $k \in \{1, \ldots, r-1\}$, the \emph{$k$th reducer} of $\bm C$ is the unique reducer $O_k$ above.

\labelcref*{itm:weak representative} For every vector $\bm w \in \mathbb Z_{\geq0}^n \setminus \{\bm 0\}$, there exists a column $\bm C$ with $\bm C_1 = \bm w$.
Therefore, it suffices to show that for every $\bm w \in \mathbb Z_{\geq0}^n \setminus \{\bm 0\}$, given two columns $\bm C := (\bm C_1, \ldots, \bm C_r)$ and $\bm C' := (\bm C_1', \ldots, \bm C_{r'}')$ with $\bm C_1 = \bm C_1' = \bm w$, the vectors $\bm C_r$ and $\bm C_{r'}'$ are permutation-equivalent.
We prove this by induction on $r + r'$.
Since the reducer $G$ commutes with every reducer $Z_i$ and~$W_i$, it suffices to consider the case where $\gcd(\bm w) = 1$.
It follows that each reducer in both of the columns is either $Z_i$ or $W_i$ for some~$i$.

First, we prove the case where there exist $\alpha < \beta$ such that $w_{\alpha} = w_{\beta} > 0$.
If $\bm t \in \mathbb Z_{\geq0}^n \setminus \{\bm 0\}$ is such that $t_\alpha = t_\beta > 0$, then we have $Z_i(\bm t)_{\alpha} = Z_i(\bm t)_{\beta} > 0$ for every $i \notin \{\alpha, \beta\}$ and we have $W_j(\bm t)_{\alpha} = W_j(\bm t)_{\beta} > 0$ for every~$j\in\{1,\ldots, n\}$.
Therefore, there exists $k \in \{1, \ldots, r-1\}$ and $i \in \{\alpha, \beta\}$ such that the $k$th reducer of $\bm C$ is~$Z_i$.
Without loss of generality, we let $k$ be the least possible.
The operation $Z_i$ commutes with each $l$th reducer in $\bm C$ for every $l < k$.
Therefore, without loss of generality, the first reducer for both of the columns is $Z_i$ for some $i \in \{\alpha, \beta\}$.
Since $Z_{\alpha}(\bm w)$ and $Z_{\beta}(\bm w)$ are permutation-equivalent, the claim follows by induction.

Next, we prove the case where $\bm w$ has at least two positive entries and there exists $\alpha$ such that $w_\alpha$ is positive and $\gcd(\pi_\alpha(\bm w)) > 1$.
Let $\bm t \in \mathbb Z_{\geq0}^n \setminus \{\bm 0\}$ be such that $\bm t$ has at least two positive entries, $t_\alpha$ is positive, $\gcd(\bm t) = 1$ and $\gcd(\pi_\alpha(\bm t)) > 1$.
Since $\gcd(\gcd(\pi_\alpha(\bm t)), \gcd(\pi_j(\bm t))) = 1$ for all $j \neq \alpha$, we find that $\gcd(\pi_\alpha(W_j(\bm t))) = \gcd(\pi_\alpha(\bm t))$ and $W_j(W_\alpha(\bm t)) = W_\alpha(W_j(\bm t))$.
Note that for all $i \neq \alpha$, we have $t_i \neq t_\alpha$.
Moreover, if $t_i$ is positive and equal to the least common multiple of all the other positive entries for some $i \neq \alpha$, then $\bm t$ has necessarily at least three positive entries.
Therefore, for all $i \neq \alpha$, we find that $\gcd(\pi_\alpha(Z_i(\bm t))) = \gcd(\pi_\alpha(\bm t))$ and $Z_i(W_\alpha(\bm t)) = W_\alpha(Z_i(\bm t))$.
It follows that there exists $k \in \{1, \ldots, r-1\}$ such that the $k$th reducer of $\bm C$ is $W_\alpha$ and moreover that for every $l < k$, this reducer $W_\alpha$ commutes with the $l$th reducer of~$\bm C$.
As before, the claim now follows by induction.

Finally, we prove the case where there $\bm w$ has at least two positive entries and $\alpha$ is the unique index such that $w_\alpha$ is positive and equal to the least common multiple of all the other positive entries.
If $\bm t \in \mathbb Z_{\geq0}^n \setminus \{\bm 0\}$ is such that $\bm t$ has at least two positive entries and $t_\alpha$ is positive and equal to the least common multiple of all the other positive entries, then for all $i$ with $t_i \neq t_\alpha$ and all~$j$, both $Z_i(\bm t)$ and $W_j(\bm t)$ still have at least two positive entries and $t_\alpha$ is still positive and equal to the least common multiple of the other positive entries.
If $t_i = t_\alpha$, then $Z_i(\bm t)$ is permutation-equivalent to $Z_\alpha(\bm t)$.
So, without loss of generality, there exists $k \in \{1, \ldots, r-1\}$ such that $\bm C_{k+1} = Z_\alpha(\bm C_k)$.
Without loss of generality, $k$ is such that for every $l < k$, if the $l$th reducer of $\bm C$ is $Z_i$ for some~$i$, then $(v_l)_i \neq (v_l)_\alpha$.
It follows that this $k$th reducer $Z_\alpha$ of $\bm C$ commutes with each $l$th reducer in the column where $l < k$.
As above, the claim now follows by induction.

\labelcref*{itm:strong representative} Obvious.
\end{proof}

We use \cref{lem:mild representative,cor:mild equivalence} to prove \cref{lem:dispatch essential support}.

\begin{lemma} \label{lem:mild representative}
Let $n$ be a positive integer and let $\bm w\in\mathbb{Z}_{\geq1}^n$. Then, the vector $\bm w $ is equivalent under the operations (1)–(3) and (5) of \cref{def:equivalences} to a unique strictly increasing well-formed vector $\bm v \in \mathbb Z_{\geq1}^m$ for some positive integer $m\leq n$.
\end{lemma}

\begin{proof}
Similar to \cref{lem:representative}\labelcref{itm:weak representative}.
\end{proof}

\begin{corollary} \label{cor:mild equivalence}
The vector $\bm v$ in \cref{lem:representative}\labelcref{itm:weak representative} (resp. $\bm v'$ in \cref{lem:representative}\labelcref{itm:strong representative}) can be reached by applying a sequence of operations (1)–(5) (resp. (1)--(4)) of \cref{def:equivalences} to $\bm w$ such that every operation in the sequence, except possibly the last, is one of (1)–(3) and (5) (resp. (1)--(3)).
\end{corollary}

\begin{proof}
For weak equivalence, this follows from \Cref{lem:mild representative}. The other case is easy.
\end{proof}

\begin{lemma}\label{lem:dispatch essential support}
Let $n$ be a positive integer. Then, both of the following hold:
\begin{enumerate}[label=\textup{(\alph*)}, ref=\alph*]
\item \label{itm:mild equivalence 1} if $\bm w \in \mathbb Z_{\geq1}^n$ is weakly equivalent to $(1) \in \mathbb Z^{1}$, then every nonempty subsequence of $\bm w$ is equivalent to $(1) \in \mathbb Z^{1}$ under the operations (1)–(3) and (5) of \cref{def:equivalences}.
\item \label{itm:weak lcm} If $\bm w \in \mathbb Z_{\geq1}^n$ is not weakly equivalent to the vector $(1) \in \mathbb Z^{1}$, then every nonempty subsequence $\bm w'$ of $\bm w$ that is weakly equivalent to $\bm w$ satisfies $d_{\bm w'} = d_{\bm w}$.
\end{enumerate}
\end{lemma}

\begin{proof}
Let $\bm v$ be the unique vector from \cref{lem:representative}\labelcref{itm:weak representative}.
Let $\bm w'$ be a subsequence of $\bm w$.
By \cref{cor:mild equivalence}, there exists a sequence of operations $O_1, \ldots, O_r$, each of type (1)–(3) or (5), such that applying them to $\bm w$ to get a vector $\overline{\bm w}$ and then possibly removing one weight that is equal to the least common multiple of the other weights, we reach~$\bm v$.
Without loss of generality, $\overline{\bm w}$ has no repeating weights.
Without loss of generality, no operation $O_i$ removes a weight from $\bm w'$ that is not a repeated weight in~$\bm w'$.
For every $i\in\{1,\ldots, r\}$, let $(O_1 \circ \cdots \circ O_i)(\bm w')$ denote the corresponding subsequence of $(O_1 \circ \cdots \circ O_i)(\bm w)$.
Note that $(O_1 \circ \cdots \circ O_i)(\bm w')$ is weakly equivalent to $\bm w'$ for every~$i$.

\labelcref*{itm:mild equivalence 1} We necessarily have $\overline{\bm w} = (1) \in \mathbb Z^1$. Therefore, $(O_1 \circ \cdots \circ O_r)(\bm w') = (1) \in \mathbb Z^1$.

\labelcref*{itm:weak lcm}
If $\bm w'$ is weakly equivalent to $\bm w$, then $(O_1 \circ \cdots \circ O_r)(\bm w')$ is equal to $\overline{\bm w}$ or $\bm v$. Since $\bm w$ is not weakly equivalent to~$(1) \in \mathbb Z^1$, for every~$i$, the length of $(O_1 \circ \cdots \circ O_i)(\bm w')$ is at least two.
Therefore, the ratios $\frac{\operatorname{lcm}{\bm w}}{\operatorname{lcm}{\bm w'}}$ and $\frac{\operatorname{lcm}((O_1 \circ \cdots \circ O_i)(\bm w))}{\operatorname{lcm}((O_1 \circ \cdots \circ O_i)(\bm w'))}$ are equal for every~$i$.
Since $\operatorname{lcm}{((O_1 \circ \cdots \circ O_r)(\bm w))} = d_{\bm v} = \operatorname{lcm}((O_1 \circ \cdots \circ O_r)(\bm w'))$, this proves the \lcnamecref{lem:dispatch essential support}.
\end{proof}

\subsection{Very ampleness and generation in degree 1 under equivalences}

\Cref{lem:simplification bubbles,lem:condition}, together with \cref{thm:bubble and generation in degree 1 - graded,thm:bubble and generation in degree 1 - Rees,thm:bubble satisfying condition implies not very ample - graded,thm:bubble satisfying condition implies not very ample - Rees}, almost give the invariance of generation in degree 1 and very ampleness under weak and strong equivalence.
The only missing parts are the following:

\begin{question} \label{que:adding d}
Let $n$ and $k$ be positive integers. Let $\bm w' \in \mathbb Z_{\geq1}^{n+1}$ and $\bm w \in \mathbb Z_{\geq1}^n$ satisfy $\bm w' = (w_1, w_2, \ldots, w_n, \operatorname{lcm}(\bm w))$.
\begin{enumerate}[label=\textup{(\alph*)}, ref=\alph*]
\item \label{itm:add d ring graded} If $P_{\bm w}^{(kd_{\bm w})}$ is generated in degree 1, then is $P_{\bm w'}^{(kd_{\bm w'})}$ necessarily generated in degree~1?
\item \label{itm:add d ring Rees} If $R_{\bm w}^{(kd_{\bm w})}$ is generated in degree 1, then is $R_{\bm w'}^{(kd_{\bm w'})}$ necessarily generated in degree~1?
\item \label{itm:add d line bundle graded} If $\mathcal O_{\mathbb{P}_Q(\bm w)}(kd_{\bm w})$ is very ample over $\operatorname{Spec}Q$, then is $\mathcal O_{\mathbb{P}_Q(\bm w')}(kd_{\bm w'})$ necessarily very ample over $\operatorname{Spec}Q$?
\item \label{itm:add d line bundle Rees} If $\mathcal O_{\operatorname{Proj} R_{\bm w}}(kd_{\bm w})$ is very ample over $\mathbb{A}^n_Q$, then is $\mathcal O_{\operatorname{Proj} R_{\bm w'}}(kd_{\bm w'})$ necessarily very ample over $\mathbb{A}^{n+1}_Q$?
\end{enumerate}
\end{question}

\begin{remark} \label{rem:positive answer}
By \cref{cor:bubble and generation in degree 1 - graded,cor:bubble and generation in degree 1 - Rees}, the answers to \cref{que:adding d} parts \labelcref{itm:add d ring graded} and~\labelcref{itm:add d ring Rees} is \emph{yes} if $k = 1$.
\end{remark}

We describe a situation where generation in degree 1 and very ampleness are equivalent.

\begin{proposition} \label{prop:equivalence when one weight is d}
Let $n$ be a positive integer and let $k \in \{1, 2\}$.
Let $\bm w \in \mathbb Z_{\geq1}^n$ satisfy $w_n = \operatorname{lcm}(\bm w)$.
Let $S$ be one of $P_{\bm w}$ or $R_{\bm w}$.
Then, $S^{(kd_{\bm w})}$ is generated in degree 1 if and only if $\mathcal O_{\operatorname{Proj} S}(kd_{\bm w})$ is very ample.
\end{proposition}

\begin{proof}
Let $\bm u$ be a \bubble{$\bm w$}.
If $(\bm w, \bm u, k)$ satisfies \cref{cond:combinatorial very ampleness}, then by \cref{rem:conditions}\labelcref{itm:trivial}, $(\bm w, \bm u, k)$ satisfies \cref{cond:combinatorial generation}.
For the rest of the proof, let $(\bm w, \bm u, k)$ satisfy \cref{cond:combinatorial generation}.
By \cref{rem:conditions}\labelcref{itm:obstruction weight}, we have $\bm w \cdot \bm u \geq (k+1)d_{\bm w}$.
Denote $m_0 := \lceil \lfloor\frac{\bm w \cdot \bm u}{d_{\bm w}}\rfloor / k \rceil$ and $r_0 := m_0 k - \lfloor\frac{\bm w \cdot \bm u}{d_{\bm w}}\rfloor$.
Then, $\bm u + r_0 \bm e_n \notin m_0 V_{\bm w}^+(kd_{\bm w})$.
\Cref{cond:combinatorial very ampleness} is satisfied for all $m < m_0$ by \cref{rem:conditions}\labelcref{itm:small m}.
Assume there exists an integer $m > m_0$ such that $\bm u + r \bm e_n \in m V_{\bm w}^+(kd_{\bm w})$, where $r := mk - \lfloor\frac{\bm w \cdot \bm u}{d_{\bm w}}\rfloor$.
To prove the \lcnamecref{prop:equivalence when one weight is d}, we will derive a contradiction.
Without loss of generality,
$m$ is the smallest integer satisfying the above.
We find that there exist $\bm a_1, \ldots, \bm a_m \in V_{\bm w}^+(k d_{\bm w})$ such that
\[
\bm u + r \bm e_n = \bm a_1 + \ldots + \bm a_m,
\]
where for every $i \in \{1, \ldots, m\}$, we have $(a_i)_n \leq k - 1$.
Since $m > m_0$, we find $r \geq k$.
This proves the case $k = 1$.
From now on, let $k \geq 2$.
The inequality $\bm w \cdot \bm u \geq (k+1)d_{\bm w}$ implies that $m_0 \geq 2$.
Therefore, $m \geq 3$.

If there exist distinct $i, j \in \{1, \ldots, m\}$ such that $(a_i)_n + (a_j)_n \geq k$, then defining
\[
\bm a_l' := \mleft\{\begin{aligned}
    & \bm a_l && \text{if $l \notin \{i, j\}$},\\
    & \bm a_i + \bm a_j - k \bm e_n && \text{if $l = i$},\\
    & k \bm e_n && \text{otherwise},
\end{aligned}\mright.
\]
we find that
\[
\bm u + (r-k)\bm e_n = \sum_{l \in \{1, \ldots, m\} \setminus \{j\}} \bm a_l' \in (m-1)V_{d_{\bm w}}^+(kd_{\bm w}),
\]
contradicting the minimality of~$m$. This proves the case $k = 2$.
\end{proof}

\begin{remark}
By \cref{prop:equivalence when one weight is d,rem:positive answer}, in the case $k = 1$, \cref{que:adding d} parts \labelcref{itm:add d line bundle graded} and \labelcref{itm:add d line bundle Rees} are respectively equivalent to \cref{que:very ample and generation equivalence - graded,que:very ample and generation equivalence - Rees}.
\end{remark}

\subsection{Subsequences} \label[sec]{sec:inductive}

Throughout, let $n$ and $\bm w\in\mathbb{Z}_{\geq1}^n$ be positive integers. In this section, we relate the existence of \bubbles{$\bm w$} to the existence of \bubbles{$\bm w'$}, where $\bm w'$ is a subsequence of~$\bm w$.

\begin{lemma} \label{lem:projecting zeros}
Let $I$ be a proper subset of $\{1, \ldots, n\}$ such that $d_{\bm w} = d_{\pi_I(\bm w)}$.
Let $\bm u \in \mathbb Z_{\geq0}^n$ be such that $u_i$ is zero for all $i \in I$.
Then, $\bm u$ is a \bubble{$\bm w$} if and only if $\pi_I(\bm u)$ is a \bubble{$\pi_I(\bm w)$}.
Moreover, for every positive integer integer~$k$, if $(\pi_I(\bm w), \pi_I(\bm u), k)$ satisfies \cref{cond:combinatorial generation}, respectively \cref{cond:combinatorial very ampleness}, then so does $(\bm w, \bm u, k)$.
\end{lemma}

\begin{proof}
Clear from the definition of a bubble and from the statements of \cref{cond:combinatorial generation,cond:combinatorial very ampleness}.
\end{proof}

\begin{remark}
We make the following remarks on \cref{lem:projecting zeros}.
\begin{enumerate}[label=\textup{(\alph*)}, ref=\alph*]
\item The condition that $d_{\bm w} = d_{\pi_I(\bm w)}$ is necessary. For example, there are no \bubbles{$\bm w$} for $\bm w := (1, 6, 10, 15, 60)$ but there is a \bubble{$\bm w'$} for $\bm w' := (1, 6, 10, 15)$.
\item The condition that $u_i$ is zero for all $i \in I$ is necessary. For example, there are no \bubbles{$\bm w'$} for $\bm w' := (6, 10, 15)$ but there is a \bubble{$\bm w$} for $\bm w := (1, 6, 10, 15)$.
\end{enumerate}
\end{remark}

\Cref{thm:coordinates are positive} below is an extension of \cref{lem:projecting zeros} to the case where $d_{\bm w}$ might not equal $d_{\pi_I(\bm w)}$.
This enables us to reduce to the case of bubbles with positive coordinates and use the strong bound \cref{thm:bound on u when all entries are positive}.

\begin{lemma} \label{thm:coordinates are positive}
    Let $\bm u$ be a \bubble{$\bm w$} and $I$ a subset of $\{1, \ldots, n\}$ such that $u_i$ is zero for all $i \in I$.
    Then, there exists a \bubble{$\pi_I(\bm w)$} $\bm u' \preceq \pi_I(\bm u)$ satisfying $\pi_I(\bm w) \cdot \bm u' \geq \bm w \cdot \bm u + d_{\pi_I(\bm w)} - d_{\bm w}$.
\end{lemma}

\begin{proof}
The case $d_{\pi_I(\bm w)} = d_{\bm w}$ follows from \cref{lem:projecting zeros}.
For the rest of the proof, let $d_{\bm w} = m d_{\pi_I(\bm w)}$ for some integer $m \geq 2$. In this case, $d_{\pi_I(\bm w)} - d_{\bm w} = -(m-1) d_{\pi_I(\bm w)}$.

Seeking a contradiction, assume there is no \bubble{\(\pi_I(\bm w)\)} $\bm u' \preceq \pi_I(\bm u)$ such that the inequality $\pi_I(\bm w) \cdot \bm u' \geq \bm w \cdot \bm u - (m-1) d_{\pi_I(\bm w)}$ holds. We recursively construct $\bm v_1, \ldots, \bm v_m \in \mathbb Z_{\geq0}^{n-\abs{I}}$. Let $\bm u_1 := \pi_I(\bm u)$ and assume we have defined $\bm u_k \in \mathbb Z_{\geq0}^{n-\abs{I}}$ with
\[
\pi_I(\bm w) \cdot \bm u_k = \bm w \cdot \bm u - (k-1) d_{\pi_I(\bm w)}
\]
for some $k \in \{1, \ldots, m\}$.
Since $\pi_I(\bm w) \cdot \bm u_k$ is at least $\bm w \cdot \bm u - (m-1) d_{\pi_I(\bm w)}$, there exists $\bm v_k \in \mathbb Z_{\geq0}^{n-\abs{I}}$ such that $\bm v_k \preceq \bm u_k$ and $\pi_I(\bm w) \cdot \bm v_k = d_{\pi_I(\bm w)}$. Define $\bm u_{k+1} := \pi_I(\bm u) - \bm v_k$.
This way, we can define $\bm v_1, \ldots, \bm v_m$ such that $\bm v := \kappa_I(\bm v_1 + \ldots + \bm v_m) \preceq \bm u$ and $\bm w \cdot \bm v = d_{\bm w}$. This contradicts $\bm u$ being a \bubble{\(\bm w\)}.
\end{proof}

\begin{corollary} \label{cor:coordinates are positive}
    Let $I$ be a proper subset of $\{1, \ldots, n\}$ such that $d_{\bm w} > d_{\pi_I(\bm w)}$.
    If there exists a \bubble{$\bm w$} $\bm u$ such that $u_i = 0$ for all $i \in I$, then there exists a \bubble{$\pi_I(\bm w)$} $\bm u'$ satisfying $\pi_I(\bm w) \cdot \bm u' \geq 3 d_{\pi_I(\bm w)}$.
\end{corollary}

\begin{proof}
Follows from \cref{thm:coordinates are positive} and the inequalities $\bm w \cdot \bm u \geq 2d_{\bm w} \geq d_{\bm w} + 2d_{\pi_I(\bm w)}$.
\end{proof}

\section{Tools for showing nonexistence of bubbles} \label[sec]{sec:tools}

Our goal is to describe as accurately as possible when the line bundles of the previous sections are very ample. In particular, we are interested in determining when a given weight vector has no bubbles. According to \cref{thm:bubble and generation in degree 1 - graded,thm:bubble and generation in degree 1 - Rees}, the latter implies that there are no obstructions for generation in degree 1 for the Veronese subrings we have defined. For what follows, we fix again a positive integer $n$ and $\bm w := (w_1, \ldots, w_n)\in\mathbb{Z}^n_{\geq1}$.

\subsection{Nonexistence of big bubbles} \label[sec]{sec:nonexistence}

The main results of this section are \cref{thm:bound n-1,thm:maximum all weights}. We show that the coordinates of a \bubble{$\bm w$} are bounded, and we deduce upper bounds for the scalar product of $\bm w$ and any \bubble{$\bm w$}. We start by observing the following.

\begin{lemma} \label{lem:firstupperbound}
    If \(\bm u\) is a \bubble{\(\bm w\)}, then \(u_i \leq \dfrac{d_{\bm w}}{w_i} - 1\) for every \(i \in \{1, \ldots, n\}\).
\end{lemma}

\begin{proof}
    Let \(\bm u\) be a \bubble{\(\bm w\)} and suppose that $u_i \geq \dfrac{d_{\bm w}}{w_i}$ for some \(1\leq i\leq n\). It gives the inequality \(w_iu_i\geq d_{\bm w}\), where both sides are divisible by \(w_i\). Hence there exists \( \gamma \in \{1, \ldots u_i\}\) such that \(w_i\gamma = d_{\bm w}\),  giving rise to an element \(\bm v := \gamma \bm e_i\preceq \bm u\) satisfying \(\bm w\cdot \bm v = d_{\bm w}\). This is a contradiction. Therefore, $u_i \leq \displaystyle\frac{d_{\bm w}}{w_i}-1$ for every \(1\leq i\leq n\).
\end{proof}

This gives us an upper bound on the scalar product between $\bm w$ and any \bubble{$\bm w$}.

\begin{corollary} \label{thm:bound n-1}
    Every \bubble{$\bm w$} $\bm u$ satisfies $\bm w \cdot \bm u < n d_{\bm w}$.
\end{corollary}

\begin{proof}
    Follows from \cref{lem:firstupperbound}.
\end{proof}

We show in \cref{thm:infinitely many bubbles graded,thm:infinitely many bubbles rees} that the bound in \cref{thm:bound n-1} is sharp.

\Cref{thm:lcm} is a generalisation of \cref{lem:firstupperbound} that we use for proving \cref{thm:bound n-2 well-formed}.

\begin{lemma} \label{thm:lcm}
Let $\bm u$ be a \bubble{$\bm w$} and let $I$ be a nonempty subset of $\{1, \ldots, n\}$. Then,
\[
\sum_{i \in I} w_i u_i \leq d_{\bm w} + (\abs{I} - 1) \operatorname{lcm}(\{w_i \mid i \in I\}) - \sum_{i \in I} w_i.
\]
\end{lemma}

\begin{proof}
Denote $g := \operatorname{lcm}(\{w_i \mid i \in I\})$. For all $i\in I$, let $a_i, b_i$ be the unique nonnegative integers such that
\[
u_i = a_i \frac{g}{w_i} + b_i \quad \text{ and } \quad b_i \leq \frac{g}{w_i} - 1.
\]
Consider the finite set $V$ of vectors $\bm v \in \mathbb Z_{\geq0}^n$ with $\bm v \preceq \bm u$ such that $v_j = 0$ if $j \notin I$ and $v_iw_i$ is divisible by $g$ for every $i \in I$. Then,
\[
\{\bm w \cdot \bm v \mid \bm v \in V\} = \{0,\, g,\, 2g,\, \ldots,\, \sum_{i \in I} a_i g\}.
\]
Since $\bm u$ is a \bubble{$\bm w$}, the inequality $\sum_{i \in I} a_i g \leq d_{\bm w} - g$ holds. Therefore, we obtain the inequality
\[
\sum_{i \in I} w_i u_i = \sum_{i \in I} a_i g + \sum_{i \in I} w_i b_i \leq d_{\bm w} - g +  \abs{I} g - \sum_{i \in I} w_i,
\]
proving the \lcnamecref{thm:lcm}.
\end{proof}

We use of the following well-known lemma in the proofs of \cref{thm:maximum one weight}.

\begin{lemma} \label{thm:combinatorics}
Let $l$ and $m$ be positive integers. Then, both of the following hold:
\begin{enumerate}[label=\textup{(\alph*)}, ref=\alph*]
    \item \label{itm:combinatorics small subsequence} for every sequence \(a_1, \ldots, a_{m}\) of positive integers, there exists a positive integer $k \leq m$ and a \(k\)-element subset \(I\) of \(\{1, \ldots, m\}\) such that \(m\) divides \(\sum_{i\in I}a_i\), and
    \item \label{itm:combinatorics big subsequence} for every sequence $a_1, \ldots, a_l$ of positive integers, there exists a (possibly empty) subsequence $b_1, \ldots, b_k$, where $k \geq l - m + 1$, such that $m$ divides $b_1 + \ldots + b_k$.
\end{enumerate}
\end{lemma}

\begin{lemma} \label{thm:maximum one weight}
Assume that $n \geq 2$ and that $w_1 \leq \ldots \leq w_n$.
If $\bm u \in \mathbb Z_{\geq0}^n$ is a \bubble{$\bm w$}, then for all $i \in \{1, \ldots, n\}$, we have
\[
u_i \leq \mleft\{\begin{aligned}
    & w_n - 2 && \text{if $i \neq n$},\\
    & w_{n-1} - 2 && \text{if $i = n$}.
\end{aligned}\mright.
\]
\end{lemma}

\begin{proof}
Let us fix $i \in \{1, \ldots, n\}$. Define
\[
\alpha_i := \mleft\{\begin{aligned}
    & w_n && \text{if $i \neq n$},\\
    & w_{n-1} &&  \text{if $i = n$}.
\end{aligned}\mright.
\]
Let \(\bm u' := \bm u-u_i\bm e_i\). Our strategy to prove the \lcnamecref{thm:maximum one weight} consists of constructing an element \(\bm v' \in \mathbb{Z}_{\geq 0}^n\) satisfying \(\bm v' \preceq \bm u'\) such that the inequalities
\[d_{\bm w} - w_i (\alpha_i-1) \leq \bm w \cdot \bm v' <  d_{\bm w} +w_i\]
hold and \(\bm w \cdot \bm v'\) is divisible by \(w_i\). If such an element \(\bm v'\) exists, then the element
\[\bm v := \bm v' + \mleft(\frac{d_{\bm w}-\bm w\cdot \bm v'}{w_i}\mright)\bm e_i\in \mathbb{Z}_{\geq0}^n\]
satisfies \(\bm w\cdot \bm v = d_{\bm w}\). Thus, since \(\bm u\) is a \bubble{$\bm w$}, we cannot have that \(\bm v\preceq \bm u\). This implies the following
\[u_i \leq \frac{d_{\bm w} - \bm w \cdot \bm v'}{w_i} - 1 \leq \frac{w_i (\alpha_i-1)}{w_i} - 1,\]
giving the desired inequality.
Let us show now that such an element \(\bm v'\) exists to conclude the proof.
\smallskip

Let  \(s := \sum_{1\leq j\leq n}u'_j\) and let $j_1, \ldots, j_s \in \{1, \ldots, n\}$ be such that $\bm u' = \sum_{1\leq k\leq s} \bm e_{j_k}$. Since the inequalities $\bm w \cdot \bm u \geq 2d_{\bm w}$ and $w_i u_i \leq d_{\bm w}-w_i$ hold respectively by definition and by \cref{lem:firstupperbound}, we find $\bm w \cdot \bm u' \geq d_{\bm w}+w_i$. Now, let $r$ be the greatest integer such that the inequality $\bm w \cdot \sum_{1\leq k \leq r} \bm e_{j_k} < d_{\bm w}+w_i$ holds.
By maximality of $r$, we obtain the following inequalities:
\[
d_{\bm w} -\alpha_i+w_i \leq \bm w \cdot \smashoperator{\sum_{k \in \{1, \ldots, r\}}} \bm e_{j_k} < d_{\bm w} + w_i.
\]
By \cref{thm:combinatorics}, there exists a subset $J$ of $\{1, \ldots, r\}$ satisfying that $\abs{J} \geq r - w_i + 1$ and such that $w_i$ divides $\bm w \cdot \sum_{k \in J} \bm e_{j_k}$. We have the following:
\begin{equation*}
d_{\bm w} - w_i (\alpha_i-1) \leq \bm w \cdot \sum_{k \in J} \bm e_{j_k} <  d_{\bm w} +w_i.
\end{equation*}
Define
\(
\bm v' :=  \sum_{k \in J} \bm e_{j_k}
\)
and note that \(\bm v' \preceq \bm u'\), finishing the proof.
\end{proof}

We give two corollaries that are weaker than \cref{thm:maximum one weight} but slightly easier to use.

\begin{corollary} \label{thm:maximum all weights}
Every \bubble{$\bm w$} $\bm u$ satisfies $\bm w \cdot \bm u \leq (\max(\bm w) - 2) \cdot \sum_i w_i$.
\end{corollary}

\begin{proof}
    Direct consequence of \cref{thm:maximum one weight}.
\end{proof}

If $w_n$ is big compared to~$w_{n-1}$, then \cref{thm:maximum product weights} may give a better bound than \cref{thm:maximum all weights}.

\begin{corollary} \label{thm:maximum product weights}
Assume that $n \geq 2$ and $w_1 \leq \ldots \leq w_n$.
Then, every \bubble{$\bm w$}~$\bm u$ satisfies $\bm w \cdot \bm u \leq n w_{n-1} (w_n - 2)$.
\end{corollary}

\begin{proof}
By \cref{thm:maximum one weight}, the following inequality holds:
\begin{equation} \label[eqn]{eqn:strong}
\bm w \cdot \bm u \leq (w_{n-1} - 2)w_n + (w_n - 2) \smashoperator{\sum_{i \in \{1, \ldots, n-1\}}} w_i.
\end{equation}
By assumption, we have $(w_{n-1} - 2)w_n \leq (w_n - 2)w_{n-1}$, so the \lcnamecref{thm:maximum product weights} follows.
\end{proof}

\subsection{Divisibility chains} \label[sec]{sec:divisibility}

The main results of this section are \cref{thm:multiple chains,thm:special multiple chains}, which show nonexistence of \bubbles{$\bm w$} in special cases concerning divisibility.

We use \cref{not:graph} in \cref{thm:one chain,thm:multiple chains,thm:special one chain}.

\begin{notation} \label{not:graph}
    Let \(G_{\bm w}\) be the directed graph with set of vertices \(\{1, \ldots, n\}\) and such that there exists an arrow from \(i\) to \(j\) if \(w_i \mid w_j\).
\end{notation}

\begin{lemma} \label{thm:one chain}
Let \(\bm u\) be a \bubble{\(\bm w\)}.
Then, for every directed path in \(G_{\bm w}\) with vertices \(I\), the following is satisfied:
\[
\smashoperator{\sum_{i \in I}} w_i u_i < d_{\bm w}.
\]
\end{lemma}

\begin{proof}
Let \(I = (i_1, \ldots, i_N)\) be the vertex sequence of a directed path in \(G_{\bm w}\). By definition, for all $j \in \{1, \ldots, N-1\}$, we have that $w_{i_j}$ divides~$w_{i_{j+1}}$.

If there exists $j \in \{1, \ldots, N-1\}$ such that $w_{i_j} u_{i_j} \geq w_{i_{j+1}}$, then we define
\[
\bm u' := \bm u + \bm e_{i_{j+1}} - \frac{w_{i_{j+1}}}{w_{i_j}} \bm e_{i_j}.
\]
We remark that \(\bm u'\) has nonnegative coordinates and it satisfies $\sum_{i \in I} w_i \cdot u_i' = \sum_{i \in I} w_i \cdot u_i$. Suppose that $\bm u'$ is not a \bubble{\(\bm w\)} and let \(\bm v'\preceq \bm u'\) be an element of \(\mathbb Z_{\geq0}^n\) such that \(\bm w\cdot \bm v' = d_{\bm w}\). We define
\[
\bm v := \mleft\{\begin{aligned}
  & \bm v' && \text{if $v_{i_{j+1}}' \leq u_{i_{j+1}}$},\\
  & \bm v' - \bm e_{i_{j+1}} + \frac{w_{i_{j+1}}}{w_{i_j}} \bm e_{i_j} && \text{otherwise}.
\end{aligned}\mright.
\]
Then, $\bm v\preceq \bm u$ is an element of  \(\mathbb Z_{\geq0}^n\) and $\bm w\cdot \bm v = \bm w \cdot \bm v' = d_{\bm w}$, a contradiction. Therefore, $\bm u'$ is a \bubble{\(\bm w\)}.
By induction, we find that for all $j \in \{1, \ldots, N-1\}$, the inequality
\[
w_{i_j} u_{i_j} \leq w_{i_{j+1}} - w_{i_j}
\]
holds. We deduce the following inequality:
\[
\sum_{i \in I} w_i u_i \leq (u_{i_N} + 1) w_{i_N} - w_{i_1}.
\]
By \cref{lem:firstupperbound}, we know that \((u_{i_N} + 1) w_{i_N}\leq d_{\bm w}\). Therefore, the inequalities
\[
(u_{i_N} + 1) w_{i_N} - w_{i_1} \leq d_{\bm w} - w_{i_1} < d_{\bm w},
\]
hold, proving the \lcnamecref{thm:one chain}.
\end{proof}

\begin{proposition} \label{thm:multiple chains}
If there exists a partition $\{I_1, \ldots, I_N\}$ of $\{1, \ldots, n\}$ such that each $I_j$ is the vertex set of a directed path in $G_{\bm w}$, then every \bubble{$\bm w$} $\bm u$ satisfies $\bm w \cdot \bm u < N d_{\bm w}$.
\end{proposition}

\begin{proof}
Follows from \cref{thm:one chain}.
\end{proof}

\begin{lemma} \label{thm:special one chain}
Let $\bm u$ be a \bubble{\(\bm w\)}.
If there exists a directed path in \(G_{\bm w}\), with vertex sequence $I:= (i_1,\ldots, i_N)$, and integers $k, l \in \{1, \ldots, n\} \setminus I$ such that $w_{i_N} \mid w_k + w_l$, then
\[
\smashoperator{\sum_{i \in I \cup \{k, l\}}} w_i u_i < 2d_{\bm w}.
\]
\end{lemma}

\begin{proof}
Let \(I:= (i_1,\ldots, i_N)\) be the vertex sequence of a directed path in \(G_{\bm w}\), and suppose there exist $k, l \in \{1, \ldots, n\} \setminus I$ satisfying $w_{i_N} \mid w_k + w_l$. Let $\{\alpha, \beta\} = \{k, l\}$ be such that $u_\alpha \leq u_\beta$. According to \cref{lem:firstupperbound}, we observe the inequalities:
\[
0 < w_\alpha+w_\beta \leq d_{\bm w}-(w_\alpha + w_\beta)u_{\alpha}.
\]
For all \(j \in \{1, \ldots, N\}\), we define
\[
  c_{i_j} :=
 \frac{d_{\bm w} - (w_\alpha + w_\beta) u_\alpha - \sum_{l \in \{j+1, \ldots, N\}} w_{i_l}u_{i_l}}{w_{i_j}}.
\]

Suppose that there exists \(j \in \{1, \ldots, N\}\) such that \(u_{i_j} \geq c_{i_j}\), and without loss of generality, suppose that $j$ is the greatest satisfying this property. The existence and maximality of \(j\) imply that the inequalities
\[
\smashoperator{\sum_{l \in \{j+1, \ldots, N\}}} w_{i_l}u_{i_l}
< d_{\bm w} - (w_\alpha + w_\beta) u_\alpha
\leq \smashoperator{\sum_{l \in \{j, \ldots, N\}}} w_{i_l}u_{i_l}
\]
hold, and all three are divisible by \(w_{i_j}\). Hence, there exists \(\gamma \in \{1, \ldots, u_{i_j}\}\) such that
\[\gamma w_{i_j} + \smashoperator{\sum_{l \in \{j+1, \ldots, N\}}} w_{i_l}u_{i_l} = d_{\bm w} - (w_\alpha + w_\beta) u_\alpha.\]
In particular, $\bm c := \gamma e_{i_j} + \sum_{l \in \{j+1, \ldots, N\}} u_{i_l} \bm e_{i_l} +  u_\alpha \bm e_\alpha + u_\alpha \bm e_\beta\preceq \bm u$ is an element of \(\mathbb Z_{\geq0}^n\)  such that \(\bm w\cdot \bm c = d_{\bm w}\). This contradicts our assumption on~$\bm u$ being a \bubble{\(\bm w\)}.

Therefore, \(u_{i_j} < c_{i_j}\) for every \(1\leq j\leq N\). In particular, \(\sum_{i\in I}w_iu_i < d_{\bm w}-(w_\alpha + w_\beta)u_\alpha\). Since $\sum_{i\in I}w_iu_i$, $d_{\bm w}$ and $w_\alpha + w_\beta$ are all divisible by $w_{i_1}$, we obtain the inequality:
\begin{equation}
\label[ineq]{ineq:sum chain u-i} \sum_{i \in I} w_i u_i \leq d_{\bm w} - (w_\alpha + w_\beta) u_\alpha - w_{i_1}.
\end{equation}
We also have the following two inequalities:
\begin{align}
  \label[ineq]{ineq:sum chain u-alpha} w_\alpha u_\alpha & \leq w_\alpha u_\alpha,\\
  \label[ineq]{ineq:sum chain u-beta} w_\beta u_\beta & \leq d_{\bm w} - w_\beta.
\end{align}
Adding the \cref{ineq:sum chain u-i,ineq:sum chain u-alpha,ineq:sum chain u-beta} gives
\[
\smashoperator{\sum_{i \in I \cup \{k, l\}}} w_i \cdot u_i \leq 2d_{\bm w} - w_\beta (u_\alpha + 1) - w_{i_1} < 2d_{\bm w},
\]
which concludes the proof.
\end{proof}

\begin{proposition} \label{thm:special multiple chains}
If $n \geq 3$ and we have a divisibility chain
\[
w_1 \mid w_2 \mid \ldots \mid w_{n-2} \mid w_{n-1} + w_n
\]
after a suitable permutation of the weights, then there are no \bubbles{\(\bm w\)}.
\end{proposition}

\begin{proof}
Follows from \cref{thm:special one chain}.
\end{proof}

\begin{remark}
For an example application, by \cite[Theorem~1.13]{Kaw03} and \cite[Theorem~2.6]{Yam18}, every 3-dimensional divisorial contraction over $\mathbb C$ with centre an isolated $cA_k$ singularity is locally analytically the restriction $Y \to X$ to a hypersurface of the weighted blowup of~$\mathbb A^4_{\mathbb C}$ with weights $\bm w$, where $\bm w$ is either one of the following:
\begin{enumerate}
\item $(r_1, r_2, a, 1)$, where $a$ divides $r_1 + r_2$,
\item $(1, 5, 3, 2)$, or
\item $(4, 3, 2, 1)$.
\end{enumerate}
As a corollary of \cref{thm:special multiple chains,cor:bubble and generation in degree 1 - Rees,thm:very ample}, in each of the three cases, the line bundle $\mathcal O_Y(d_{\bm w})$ is very ample.
\end{remark}

\subsection{Technical inequality} \label[sec]{sec:technical}

\Cref{thm:bound on u when all entries are positive} is the most important technical inequality for determining the small weights that admit bubbles in \cref{sec:small}.
We use the following notation in \cref{thm:bound on u when all entries are positive}.

\begin{notation} \label{not:bound on u when all entries are positive}
    For all \(i \in \{1, \ldots, n\}\) and \(r \in \{0, \ldots, w_i-1\}\), define
    \begin{align*}
        B_{\bm w, i,r} & := \bigl\{\bm w\cdot \bm b \;\big|\; \bm b\in \{0, 1\}^n,\, b_i=0,\, \bm w\cdot \bm b\equiv r\mod w_i\bigr\},\\
        b_{\bm w, i,r} & := \mleft\{\begin{aligned}
            & \min(B_{\bm w, i,r}) && \text{if $B_{\bm w, i,r}$ is nonempty},\\
            & (w_i-1) \cdot \max(\pi_i(\bm w)) && \text{otherwise},
        \end{aligned}\mright.\\
        c_{\bm w, i} & := \max\{b_{\bm w, i,r} \mid r \in \{0, \ldots, w_i-1\}\}.
    \end{align*}
\end{notation}

\begin{lemma} \label{thm:bound on u when all entries are positive}
    Let \(\bm u\) be a \bubble{\(\bm w\)} with positive coordinates. Then, for every \(i \in \{1, \ldots, n\}\) such that the inequality
    \(\sum_{j \in \{1, \ldots, n\} \setminus \{i\}} w_j < d_{\bm w}+w_i\)
    holds, we have an inequality:
    \[
        u_i \leq \mleft\lfloor \frac{\max(\pi_i(\bm w)) + c_{\bm w, i}}{w_i}\mright\rfloor - 2.
    \]
\end{lemma}

\begin{proof}
    Let \(i\in \{1,\ldots, n\}\) be such that the inequality \(\sum_{j \in \{1, \ldots, n\} \setminus \{i\}} w_j < d_{\bm w}+w_i\) holds. Define $\alpha_i := \max(\pi_i(\bm w))$.
    Let \(\bm u' := \bm u - u_i\bm e_i\) and note that \(\bm w\cdot \bm u'\geq d_{\bm w}+w_i\) since \(\bm u\) is a \bubble{\(\bm w\)}. We follow a strategy similar to the one in the proof of \cref{thm:maximum one weight}.
    \smallskip

    Let \(\bm a' := \sum_{j \in \{1, \ldots, n\} \setminus \{i\}} \bm e_j\). Note that \(\bm a'\preceq \bm u'\) and the inequality
    \[\bm w\cdot \bm a' < d_{\bm w}+w_i\]
    holds by the assumption on \(i\). Let \(s := \sum_{j \in \{1, \ldots, n\}}(u'_j-a'_j)\) and let $j_1, \ldots, j_s \in \{1, \ldots, n\}$ be such that $\bm u'-\bm a' = \sum_{k \in \{1, \ldots, s\}} \bm e_{j_k}$. We define \(t \in \{0, \ldots, s\}\) to be the greatest integer such that the following holds:
    \[
    \bm w\cdot (\bm a' + \smashoperator{\sum_{k \in \{1, \ldots, t\}}} \bm e_{j_k}) < d_{\bm w}+w_i.
    \]
    Let \(\bm a := \bm a' + \sum_{k \in \{1, \ldots, t\}} \bm e_{j_k}\).
    It follows that $\bm a' \preceq \bm a \preceq \bm u'$ and
    \[d_{\bm w}-\alpha_i+w_i \leq \bm w\cdot \bm a < d_{\bm w} + w_i.\]
    Let \(0\leq r\leq w_i-1\) be such that \(\bm w\cdot \bm a\equiv r\mod w_i\). There are two cases.
    \begin{enumerate}
        \item If \(B_{\bm w, i,r}\) is nonempty, then let \(\bm b\in \{0, 1\}^n\) with \(b_i = 0\) be such that \(\bm w\cdot \bm b = b_{\bm w, i, r}\). We define \(\bm v' := \bm a- \bm b\). Note that \(\bm v'\) is an element of \(\mathbb{Z}_{\geq0}^n\) that satisfies both \(\bm v'\preceq \bm u'\) and \(w_i\mid \bm w\cdot \bm v'\). Moreover, the inequalities
        \[
        d_{\bm w} - \alpha_i - b_{\bm w, i, r} + w_i \leq \bm w \cdot \bm v' < d_{\bm w} + w_i
        \]
        hold. Similarly to the proof of \cref{thm:maximum one weight}, we find that the following upper bound on~$u_i$:
        \[u_i \leq \frac{\alpha_i + b_{\bm w, i, r} - w_i}{w_i} - 1.\]
        Since $b_{\bm w, i,r} \leq c_{\bm w, i}$, we arrive at the desired inequality.
        \item If \(B_{\bm w, i,r}\) is empty, then we obtain the equality \(b_{\bm w, i,r} = (w_i-1)\alpha_i\) and in particular, \(c_{\bm w, i}\geq (w_i-1) \alpha_i\). This together with \cref{thm:maximum one weight} shows that the following inequalities hold:
        \[ u_i \leq \alpha_i-2 \leq \frac{\alpha_i + c_{\bm w, i}}{w_i} -2 .\]
        This concludes the proof.\qedhere
    \end{enumerate}
\end{proof}

\begin{corollary} \label{cor:bound on u when all entries are positive}
Let $\bm w \in \mathbb Z_{\geq1}^n$ satisfy $\sum_{j \in \{1, \ldots, n\}} w_j < d_{\bm w} + 2 \min(\bm w)$.
Let $\bm v \in \mathbb Z^n$ be the vector with $i$th entry
\[
\mleft\lfloor \frac{\max(\pi_i(\bm w)) + c_{\bm w, i}}{w_i}\mright\rfloor - 2
\]
for every $i \in \{1, \ldots, n\}$. Then, every \bubble{$\bm w$} $\bm u$ with positive coordinates satisfies
\[
\bm w \cdot \bm u \leq \bm w \cdot \bm v.
\]
\end{corollary}

\begin{proof}
Follows from \cref{thm:bound on u when all entries are positive}.
\end{proof}

\begin{remark} \label{rem:negative bound}
We make the following remarks on \cref{thm:bound on u when all entries are positive}.
\begin{enumerate}[label=\textup{(\alph*)}, ref=\alph*]
\item \label{itm:negative bound} It is possible for the bound in \cref{thm:bound on u when all entries are positive} to be negative. Let $\bm w \in \mathbb Z_{\geq1}^n$ be any vector of positive integers such that $\sum_j w_j < d_{\bm w} - \max(\bm w) + 3$. Let $i$ be an index such that $w_i = \max(\bm w)$. Let $\bm w' := (w_1, w_2, \ldots, w_n, 1, 1, \ldots, 1) \in \mathbb Z_{\geq1}^{n + w_i - 1}$.
Then,
\[
\mleft\lfloor \frac{\max(\pi_i(\bm w')) + c_{\bm w', i}}{w_i'}\mright\rfloor - 2 = -1,
\]
implying by \cref{thm:bound on u when all entries are positive} that there are no \bubbles{$\bm w'$} with positive coordinates.
\item \Cref{thm:bound on u when all entries are positive} does not always hold when $\bm u$ is allowed to have coordinates equal to zero. For an example, take $\bm w := (1, 12, 22, 33)$ and use \labelcref{itm:negative bound} and the fact that there exists a \bubble{$\bm w$}, namely $(1, 10, 2, 3)$.
\end{enumerate}
\end{remark}

\Cref{thm:at least as strong} shows that the bound in \cref{thm:bound on u when all entries are positive} is always at least as strong as the bound in \cref{thm:maximum one weight}.

\begin{lemma} \label{thm:at least as strong}
For every $i \in \{1, \ldots, n\}$, we have the following inequality:
\[
c_{\bm w, i} \leq (w_i-1) \cdot \max(\pi_i(\bm w)).
\]
\end{lemma}

\begin{proof}
To prove the \lcnamecref{thm:at least as strong}, it suffices to show that if $B_{\bm w, i, r}$ is nonempty, then the inequality $\min(B_{\bm w, i, r}) \leq (w_i - 1) \cdot \max(\pi_i(\bm w))$ holds.

Let $\bm b \in \{0, 1\}^n$ be such that $b_i = 0$ and $\bm w \cdot \bm b = \min(B_{\bm w, i, r})$. Let $N, i_1, \ldots, i_N$ be positive integers such that $\bm b = \sum_{j \in \{1, \ldots, N\}} \bm e_{i_j}$.
If $N \geq w_i$, then let $1 \leq k < l \leq n$ be integers such that $\bm w \cdot \sum_{j \in \{1, \ldots, k\}} \bm e_{i_j}$ and $\bm w \cdot \sum_{j \in \{1, \ldots, l\}} \bm e_{i_j}$ are congruent modulo~$w_i$. We find that the vector $\bm b' := \bm b - \sum_{j \in \{k+1, \ldots, l\}} \bm e_{i_j}$ satisfies $\bm b' \in \{0, 1\}^n$, $b_i' = 0$ and $\bm w \cdot \bm b' \equiv r \mod{w_i}$. Since $\bm w \cdot \bm b' < \bm w \cdot \bm b$ holds by construction, this contradicts the statement that $\bm w \cdot \bm b = \min(B_{\bm w, i,r})$.
Therefore, $N < w_i$ and $\min(B_{\bm w, i,r}) \leq (w_i - 1) \cdot \max(\pi_i(\bm w))$ hold.
\end{proof}

\section{Sharp upper bound for bubbles} \label[sec]{sec:bubbles}

We have shown in \cref{thm:maximum product weights} that given positive integers $n$ and $\bm w\in\mathbb{Z}_{\geq1}^n$, any \bubble{$\bm w$}~$\bm u$ satisfies $\bm w\cdot \bm u < nd_{\bm w}$. In this section we show that if $\bm w\cdot \bm u$ is divisible by $d_{\bm w}$ or if $\bm w$ is well-formed, then one obtains a sharper bound on the scalar product of $\bm w$ and~$\bm u$.
In the process, we prove that pairwise coprime weights do not admit bubbles, which is an interesting statement on its own.

\begin{proposition} \label{thm:coprime weights}
If $n$ is a positive integer and $w_1, \ldots, w_n$ are pairwise coprime positive integers, then there are no \bubbles{$\bm w$}.
\end{proposition}

\begin{proof}
By \cref{lem:simplification bubbles}\labelcref{itm:strongly}, it suffices to consider the case where $\bm w$ is strictly increasing.
Moreover, by \cref{thm:bound n-1}, we only have to prove the cases where $n \geq 3$.

We consider the case $n = 3$.
We have $d_{\bm w} = w_1 w_2 w_3$.
There are thus two cases:
\begin{itemize}
    \item if $w_1 = 1$, then \cref{thm:multiple chains} tells us directly that there are no \bubbles{$\bm w$},
    \item if $w_1 > 1$, remark that every \bubble{$\bm w$}~$\bm u$ satisfies:
    \[
        2 w_1 w_2 w_3 \leq \bm w \cdot \bm u < 3 w_2 w_3,
    \]
    where the first inequality holds since $\bm u$ is a \bubble{$\bm w$} and the second holds by \cref{thm:maximum product weights}. Since $w_1 > 1$, we find that there are no \bubbles{$\bm w$}.
\end{itemize}

Next, we consider the case $n \geq 4$.
By \cref{thm:maximum product weights}, it suffices to prove that the inequality $2d_{\bm w} \geq n w_{n-1} w_n$ holds.
Since $\bm w$ is strictly increasing, we have the inequality:
\[
2 d_{\bm w} = 2 (w_1 \cdot \ldots \cdot w_{n-2}) w_{n-1} w_n \geq 2 (n-2)! w_{n-1} w_n.
\]
Since $n \geq 4$, we find that $2(n-2)! \geq n$, proving the \lcnamecref{thm:coprime weights}.
\end{proof}

\begin{theorem} \label{thm:bound n-2 well-formed}
Let $n$ be a positive integer and let $\bm w \in \mathbb Z_{\geq1}^n$ be well-formed (\cref{def:well-formed}).
Then, every \bubble{$\bm w$} $\bm u$ satisfies $\bm w \cdot \bm u < (n-1) d_{\bm w}$.
\end{theorem}

\begin{proof}
The cases $n \in \{1,2\}$ follow by \cref{thm:bound n-1}. In the case $n = 3$, the weights are pairwise coprime and the result follows from \cref{thm:coprime weights}.
From now on, let $n \geq 4$ and assume that there exists a \bubble{$\bm w$}~$\bm u$ satisfying $\bm w \cdot \bm u \geq (n-1)d_{\bm w}$. We prove two claims.

\begin{claim} \label{cla:four}
For all pairwise distinct integers $i, j, k, l \in \{1, \ldots, n\}$, we have $\operatorname{lcm}(w_i, w_j) = d_{\bm w}$ or $\operatorname{lcm}(w_k, w_l) = d_{\bm w}$.
\end{claim}

\begin{proof}
Assume that there exist pairwise distinct $i, j, k, l \in \{1, \ldots, n\}$ such that both $\operatorname{lcm}(w_i, w_j)$ and $\operatorname{lcm}(w_k, w_l)$ are less than $d_{\bm w}$.
By applying \cref{lem:firstupperbound} to every $\alpha\in\{1,\ldots, n\}\setminus\{i,j,k,l\}$ and by applying \cref{thm:lcm} twice, namely to $I = \{i, j\}$ and $I = \{k, l\}$, we find that the inequality
\[\bm w\cdot \bm u \leq (n-2)d_{\bm w} + \operatorname{lcm}(w_i, w_j) + \operatorname{lcm}(w_k, w_l) - \sum_{\alpha=1}^nw_{\alpha}\]
holds.
Since $\operatorname{lcm}(w_i, w_j)$ and $\operatorname{lcm}(w_k, w_l)$ divide $d_{\bm w}$, we necessarily have
\[
\begin{aligned}
  \operatorname{lcm}(w_i, w_j) & \leq \frac{d_{\bm w}}{2} \quad \text{and}\\
  \operatorname{lcm}(w_k, w_l) & \leq \frac{d_{\bm w}}{2}.
\end{aligned}
\]
But this contradicts our assumption on~$\bm u$.
This proves the claim.
\end{proof}

\begin{claim} \label{cla:three}
For all pairwise different $i, j, k$, one of the equalities $\operatorname{lcm}(w_i, w_j) = d_{\bm w}$, $\operatorname{lcm}(w_i, w_k) = d_{\bm w}$ or $\operatorname{lcm}(w_j, w_k) = d_{\bm w}$ holds.
\end{claim}

\begin{proof}
Assume that the inequalities $\operatorname{lcm}(w_i, w_j) < d_{\bm w}$, $\operatorname{lcm}(w_i, w_k) < d_{\bm w}$ and $\operatorname{lcm}(w_j, w_k) < d_{\bm w}$ all hold.
Choose any $l \in \{1, \ldots, n\} \setminus \{i, j, k\}$. We find that for every $m \in \{1, \ldots, n\} \setminus \{l\}$, we have $\operatorname{lcm}(w_l, w_m) = d_{\bm w}$.
Since $\gcd$ and $\operatorname{lcm}$ are distributive over each other, we have $\operatorname{lcm}(w_l, \gcd(\pi_m(\bm w))) = d_{\bm w}$.
Since $\gcd(\pi_m(\bm w)) = 1$, we find $w_l = d_{\bm w}$.
Therefore, $u_l = 0$.
By \cref{lem:projecting zeros}, $\pi_l(\bm u)$ is a \bubble{$\pi_l(\bm w)$} with $\pi_l(\bm w) \cdot \pi_l(\bm u) \geq (n-1)d_{\bm w}$. This contradicts \cref{thm:bound n-1}, proving the claim.
\end{proof}

Next, we show that there exists a unique integer $i_0 \in \{1, \ldots, n\}$ such that for all $j, k \in \{1, \ldots, n\} \setminus \{i_0\}$ with $j\neq k$, we have $\operatorname{lcm}(w_j, w_k) = d_{\bm w}$.
The existence follows from the two claims above.
Now we show uniqueness.
It suffices to show that for every $i_0$ as above, we have $\operatorname{lcm}(w_{i_0}, w_j) < d_{\bm w}$ for all $j \in \{1, \ldots, n\} \setminus \{i_0\}$.
Assume that there exists $j \in \{1, \ldots, n\} \setminus \{i_0\}$ such that $\operatorname{lcm}(w_{i_0}, w_j) = d_{\bm w}$. Then, $\operatorname{lcm}(w_j, w_k) = d_{\bm w}$ for all $k \in \{1, \ldots, n\} \setminus \{j\}$.
Similarly to the proof of \cref{cla:three}, we find a contradiction, proving the claim.

Denote $J := \{1, \ldots, n\} \setminus \{i_0\}$ and for every $j \in J$, denote $q_j := \frac{d_{\bm w}}{w_j}$.

We show that the integers $q_j$ with $j \in J$ are pairwise coprime and with product~$d_{\bm w}$.
Choose different $j, k \in J$.
Since $\operatorname{lcm}(w_j, w_k) = d_{\bm w}$, we find that $\gcd(\frac{d_{\bm w}}{w_j}, \frac{d_{\bm w}}{w_k}) = 1$.
Now, let $m$ denote the least common multiple of the elements $\frac{d_{\bm w}}{w_j}$ where $j \in J$. Then, there exist positive integers $a_1, \ldots, a_n$ satisfying $\gcd(a_1, \ldots, a_n) = 1$ and $m w_j = d_{\bm w} a_j$ for all~$j$. Since $\pi_{i_0}(\bm w) = 1$, we find
\[
m = \gcd(\{m w_j \mid j \in J\}) = d_{\bm w}.
\]
This proves the claim.

It follows that for all $j \in J$, we have $w_j = \prod_{k \in J \setminus \{j\}} q_k$.
Since $\pi_j(\bm w) = 1$ for all $j \in J$, we find that $\gcd(w_{i_0}, q_j) = 1$ for all $j \in J$.
We find $w_{i_0} = \gcd(w_{i_0}, d_{\bm w}) = 1$, where the second equality comes from the fact that $\gcd$ and $\operatorname{lcm}$ are distributive over each other. By \cref{thm:multiple chains}, $\bm w \cdot \bm u < (n-1)d_{\bm w}$, a contradiction.
\end{proof}

\begin{corollary} \label{thm:bound n-2 divisible}
Every \bubble{$\bm w$} $\bm u$ such that $d_{\bm w}$ divides $\bm w \cdot \bm u$ satisfies $\bm w \cdot \bm u \leq (n-2) d_{\bm w}$.
\end{corollary}

\begin{proof}
By \cref{thm:bound n-1}, in order to prove the \lcnamecref{thm:bound n-2 divisible}, it is enough to show that there is no \bubble{$\bm w$}~$\bm u$ satisfying $\bm w \cdot \bm u = (n-1) d_{\bm w}$. Moreover, according to \cref{lem:simplification bubbles}\labelcref{itm:weakly} we may assume without loss of generality that $\bm w$ is well-formed.
The \lcnamecref{thm:bound n-2 divisible} now follows from \cref{thm:bound n-2 well-formed}.
\end{proof}

\section{Small weights} \label[sec]{sec:small}

In this section we treat the main computational part of the paper, describing an efficient algorithm for enumerating all the bubbles with positive coordinates of a given weight vector. We apply this algorithm to list all the weight vectors with maximum less than \upperbound{} which admit a bubble.

By a \textbf{list} we mean a finite sequence.

\begin{definition}
    Let $W\subseteq \mathbb{Z}_{\geq 0}^n$ be a list of vectors.
    \begin{itemize}
    \item We say $W$ is \textbf{arranged} if $W$ is lexicographically ordered and no two consecutive elements are comparable for the product order.
    \item We say $W$ is \textbf{marked} if there exists $\bm c \in \mathbb Z_{\geq0}^n$ such that $W$ contains $c_i\bm e_i$ for every $i\in\{1, \ldots, n\}$.
    \item We call $\widetilde{W}$ an \textbf{arrangement} of $W$ if $\widetilde{W}$ is an \emph{arranged list} containing all the minimal elements of $W$.
    \end{itemize}
\end{definition}

\begin{proposition}\label{algo: proof algorithm}
    \Cref{alg: Recursive algorithm} in \Cref{app:algorithm} returns the correct output.
\end{proposition}

\begin{proof}
    Let $n$ and $\bm w:=(w_1, \ldots, w_n)\in\mathbb{Z}_{\geq 1}^n$ be positive integers, let $\variable{mins},\variable{maxs}\in\mathbb{Z}^n$ be vectors, let $\variable{min\_product}\in\mathbb{Z}$ be an integer and let $\variable{thresholds}\subseteq \mathbb{Z}_{\geq0}^n$ be an arranged marked list of nonnegative vectors. We let $\variable{a}$ and $\variable{b}$ be defined as in lines~\ref{line: lower bound a} and \ref{line: upper bound b}, and we denote by $V$ the output of
    \[\textsc{Recursive}(\bm w,\, \variable{mins},\,\variable{maxs},\, \variable{min\_product},\, \variable{thresholds},\,\variable{false}).\]

    Suppose that there exists a vector $\bm u\in\mathbb{Z}^n$ such that $\variable{mins}\preceq \bm u\preceq \variable{maxs}$ and $\variable{min\_product}\leq \bm w\cdot \bm u$ hold, and such that there are no vectors $\bm v\in \variable{thresholds}$ with $\bm v\preceq \bm u$. Since $\variable{thresholds}$ is marked and arranged, we know that its first element $\bm v$ is of the form $\bm v = c_n\bm e_n$ for some $c_n\geq 0$. We thus deduce, by assumption, that $u_n\leq \variable{b}$. Moreover,  we know that
    \[w_nu_n\geq \variable{min\_product}-\pi_n(\bm w)\cdot \pi_n(\bm u)\]
    and together with $\pi_n(\bm u)\preceq\pi_n(\variable{maxs)}$ we obtain that $u_n\geq \variable{a}$. In particular $\variable{a} \leq \variable{b}$, which justifies line~\ref{line: return empty list}.

    \begin{enumerate}
        \item If $n=1$ (line~\ref{line: terminating case}), then $\variable{thresholds}$ consists of a single vector $\bm v$ which is a nonnegative multiple of $\bm e_1$. Moreover $\bm u$ consists of one coordinate $u_1$ which must be contained, by assumption, between $\max\mleft(\variable{mins}_1, \lceil\frac{\variable{min\_product}}{w_1}\rceil\mright)$ and $\min(\variable{maxs}_1, v_1-1)$. In this situation, the former coincides with $\variable{a}$ and the latter with $\variable{b}$. Thus $\bm u$ is in $V$ and by construction, all the elements in $V$ satisfy the same assumptions as $\bm u$. In particular, $\variable{a}\bm e_1\in V$ is the smallest element for the lexicographic order to satisfy the wanted assumptions.
        \item In the case where $n\geq 2$ is a positive integer, we observe as before that
        \[\pi_n(\bm w)\cdot \pi_n(\bm u)\geq \variable{min\_product}-w_nu_n.\]
        If one denotes $\bm u':= \pi_n(\bm u)$, then  there are no vectors
        \[\bm v'\in \variable{thrs}:=\{\pi_n(\bm v)\mid \bm v\in\variable{thresholds},\, v_n\leq u_n\}\]
        with $\bm v'\preceq \bm u'$. Note that the arrangement $\widetilde{\variable{thrs}}$ of $\variable{thrs}$ is again marked. Moreover, by the assumptions on $\bm u$, we already know that the inequalities $\pi_n(\variable{mins})\preceq\bm u'\preceq\pi_n(\variable{maxs})$ hold. Hence, by induction, $\bm u'$ lies in the output $V'$ of
        \[\textsc{Recursive}(\pi_n(\bm w),\, \pi_n(\variable{mins}),\,\pi_n(\variable{maxs}),\, \variable{min\_product}-u_nw_n,\, \widetilde{\variable{thrs}},\,\variable{false})\]
        and $\bm u = \kappa_n(\bm u')+u_n\bm e_n \in V$. Moreover, once again, all the elements in $V$ satisfy the wanted assumptions.
    \end{enumerate}
    Note that by induction, if $V$ is nonempty, the output of
    \[\textsc{Recursive}(\bm w,\, \variable{mins},\,\variable{maxs},\, \variable{min\_product},\, \variable{thresholds},\,\variable{true})\]
    consists of the smallest element of $V$, for the reverse lexicographic order.
\end{proof}

\begin{proposition}\label{algo: proof bubbles algorithm}
    \Cref{alg: Bubbles algorithm} in \Cref{app:algorithm} returns the correct output.
\end{proposition}

\begin{proof}
    Let $n$ and $\bm w:=(w_1, \ldots, w_n)\in\mathbb{Z}_{\geq 1}^n$ be positive integers.

     \textbf{Step 1 (line~\ref{line: step 1})}. According to \cref{thm:bound n-1}, if $n < 3$, then there are no \bubbles{$\bm w$} and the procedure terminates. In what follows, we will therefore assume that $n\geq 3$.

    \textbf{Step 2 (line~\ref{line: step 2})}. According to \cref{lem:simplification bubbles}\labelcref{itm:strongly}, if $g := \gcd(w_1, \ldots, w_n) > 1$, then we may replace $\bm w$ with $(w_1/g, \ldots, w_n/g)$ since both vectors have the same set of bubbles. In what follows, we will therefore assume that $g = 1$.

    \textbf{Step 3 (line~\ref{line: step 3})}. According to \cref{thm:maximum one weight}, if the inequality
    \[\bm w\cdot \bm \alpha < 2d_{\bm w}\]
    holds, then there are no \bubbles{$\bm w$} and the procedure terminates.

    \textbf{Step 4 (line~\ref{line: step 4})}.
    By \cref{lem:firstupperbound,thm:maximum one weight}, if $\bm u$ is a \bubble{$\bm w$}, then $u_i\leq \varepsilon_i$ for every $i\in\{1,\ldots, n\}$. Moreover, following an argument similar to the beginning of the proof of \cref{thm:one chain}, if such a \bubble{$\bm w$} exists, then there exists a \bubble{$\bm w$}~$\bm u'$ such that $u'_i \leq \zeta_i-1$ for every $i\in\{1,\ldots, n\}$. By definition, every \bubble{$\bm w$}~$\bm u$ satisfies the inequality $\bm w \cdot \bm u \geq 2d_{\bm w}$. Hence if $\bm w \cdot \bm{\beta} < 2d_{\bm w}$, then there are no \bubbles{$\bm w$} and the procedure terminates. Note that for all $i\in\{1,\ldots, n\}$, we add $\varepsilon_i+1$ to the definition of $\zeta_i$ since $w_i$ might divide none of the $w_j$'s for all $j > i$.

    \textbf{Step 5 (line~\ref{line: step 5})}.
    Note that by \cref{thm:at least as strong}, the inequality $\delta_i \leq \alpha_i$ holds for every $i\in\{1, \ldots, n\}$. The vector $\bm \delta$ is harder to compute than $\bm \alpha$ but it gives upper bounds that are often stronger. Again, by \cref{lem:firstupperbound,thm:bound on u when all entries are positive}, the inequality $u_i\leq \delta_i$ holds for every $i\in\{1,\ldots, n\}$ and every \bubble{$\bm w$}~$\bm u$ with positive coordinates. Similarly to before, if there exists a \bubble{$\bm w$} with positive coordinates, then there exists a \bubble{$\bm w$}~$\bm u$ such that $u_i\leq \gamma_i$ for every $i\in\{1, \ldots n\}$. Hence, if $\bm w \cdot \bm{\gamma} < 2d_{\bm w}$, then there are no \bubbles{$\bm w$} with positive coordinates.

    \textbf{Step 6 (line~\ref{line: step 6})}. If there exists a \bubble{$\bm w$} with positive coordinates, then there exists a \bubble{$\bm w$}~$\bm u$ with positive coordinates satisfying $\bm u \preceq \bm{\gamma}$. Such a \bubble{$\bm w$}~$\bm u$ satisfies $\bm v \npreceq \bm u$ for every $\bm v\in V_{\bm w}(d_{\bm w})$.
    If $\bm v \npreceq \bm{\gamma}$, then we automatically have $\bm v \npreceq \bm u$.
    Hence, by \Cref{algo: proof algorithm}, if such a \bubble{$\bm w$}~$\bm u$ exists, the output of
    \[\textsc{Recursive}(\bm w,\, \variable{mins},\,\bm \gamma,\, 2d_{\bm w},\, \widetilde{\variable{W}_{\bm \gamma}},\,\bm \beta,\, \variable{true})\]
    is nonempty. If it is empty, then there are no \bubbles{$\bm w$} satisfying the wanted assumptions, and the procedure terminates.

    \textbf{Step 7 (line~\ref{line: step 7})}. We have previously shown that for every \bubble{$\bm w$}~$\bm u$ satisfying the wanted assumptions, we have $\bm u \preceq \bm{\delta}$. Hence, by \Cref{algo: proof algorithm}, the output of
    \[\textsc{Recursive}(\bm w,\, \variable{mins},\,\bm \gamma,\, 2d_{\bm w},\, \widetilde{\variable{W}_{\bm \delta}},\,\bm \beta,\, \variable{false})\]
    consists exactly of the \bubbles{$\bm w$} with positive coordinates.
\end{proof}

\begin{remark} \label{rem:computational}
We make the following remarks on computational complexity.
\begin{enumerate}[label=\textup{(\alph*)}, ref=\alph*]
\item Out of the $2^{41} \approx 2$ trillion weight vectors, there were 82 vectors that took longer than half an hour to compute on a single CPU core.
One possibility to improve the computation time is to develop stronger inequalities on the coordinates of the bubbles.
The five weight vectors $\bm w$ that took the longest were
\begin{itemize}
\item 47\,h~\phantom{0}2\,min – (2, 5, 9, 21, 28, 30, 35, 40),
\item 36\,h~28\,min – (1, 16, 24, 26, 30, 39, 40),
\item 20\,h~51\,min – (2, 5, 7, 12, 14, 16, 21, 28, 35, 40),
\item 19\,h~53\,min – (2, 7, 9, 21, 28, 30, 35, 40),
\item 13\,h~29\,min – (2, 21, 28, 30, 32, 35, 40).
\end{itemize}

\item The average computation time for computing \cref{tab:positive coordinates} for a single weight vector $\bm w$ on a single CPU core
was 3.305\,\textmu s.
Another possibility to further improve the computation time is to reduce the amount of weight vectors we look through.
For example, the vectors $\bm w$ where the three biggest entries are $(w_{n-2}, w_{n-1}, w_n) = (39, 40, 41)$ could be immediately ruled out by \cref{thm:maximum all weights},

\item Given a list of vectors $W \subseteq \mathbb{Z}_{\geq 0}^n$, computing the set $W'$ of minimal elements for the product order can be expensive.
The set $W'$ or the process of computing it has various names in the literature:
\begin{itemize}
    \item Pareto front or Pareto frontier or Pareto set,
    \item skyline operator or skyline query,
    \item maxima of a point set,
    \item multi-objective optimisation.
\end{itemize}
Instead of computing~$W'$, we compute an approximation to it, namely what we call an \emph{arrangement}.

\item Computations for \cref{thm:bubbles with positive coordinates} have shown that condition in line~\ref{line: step 3} of \cref{alg: Bubbles algorithm} is often more than ten times faster compared to line~\ref{line: step 4}. In the computations for \cref{tab:positive coordinates}, we have observed that \cref{alg: Bubbles algorithm} terminates at line~\ref{line: terminate step 3} for most of the strictly increasing vectors in $\mathbb{Z}_{\geq 1}^n$ with bounded coordinates.

\item
The upper bounds defined by $\variable{maxs}$ in \cref{alg: Recursive algorithm} are often not sharp after a few iterations. For instance, if after $k$ iterations we have chosen particular values for $u_{n-k+1}, \ldots, u_n$ such that $\widetilde{\variable{thrs}}$ (defined in line~\ref{line: arrangement thrs}) contains an element of the form $c\bm e_{n-k}$ with $c \leq \variable{maxs}_{n-k}$, then we would obtain a strictly sharper upper bound $u_{n-k} \leq c-1$. This drastically reduces the amount of cases to check.

\item
Since the list $\variable{thresholds}$ in \Cref{alg: Recursive algorithm} is lexicographically ordered and consists of nonnegative vectors, its first element is of the form $(0, \ldots, 0, v_n)$.
At line~\ref{line: upper bound b}, we use this to give a bound $u_n \leq v_n - 1$.
Computationally, there is almost no performance loss if instead of taking the first element of $\variable{thresholds}$, we loop through the entire list $\variable{thresholds}$ to find all elements of the form $(0, \ldots, 0, v_n)$ and then take the minimum over all~$v_n$.
\item \cref{alg: Bubbles algorithm} only deals with bubbles with positive coordinates so that we can benefit from the strong bounds given by \cref{thm:bound on u when all entries are positive}. The algorithm could be adapted to enumerate all bubbles of a given weight vector $\bm w$. However, it might be faster to enumerate all bubbles with positive coordinates for every subsequence of $\bm w$ of length at least equal to three, and construct all bubbles from there.
\end{enumerate}
\end{remark}

We use the following criterion for \cref{cond:combinatorial very ampleness}

\begin{lemma}\label{thm:a sufficient condition for condition 2.2}
    Let $n\geq 3$ and $\bm w := (w_1,\ldots, w_n)\in\mathbb{Z}_{\geq 1}^n$ be positive integers. Let $\bm u\in \mathbb{Z}_{\geq 0}^n$ be a \bubble{$\bm w$} satisfying both of the following conditions:
    \begin{enumerate}[label=\textup{(\arabic*)}, ref=\arabic*]
        \item $\bm w\cdot \bm u < 3d_{\bm w}$,
        \item \label{itm:existence of i} there exists $i\in \{1, \ldots, n\}$ such that $u_i = \dfrac{d_{\bm w}}{w_i}-1$.
    \end{enumerate}
    Then, $(\bm w, \bm u, 1)$ satisfies \cref{cond:combinatorial very ampleness} if and only if $\bm u\notin 2V^+_{\bm w}(d_{\bm w})$.
\end{lemma}

\begin{proof}
    Note that if already $\bm u\in 2V^+_{\bm w}(d_{\bm w})$, then $(\bm w, \bm u, 1)$ does not satisfy \cref{cond:combinatorial very ampleness}.

    From now on, let us assume that $\bm u\notin 2V^+_{\bm w}(d_{\bm w})$, let $i\in \{1, \ldots, n\}$ be as in part (\ref{itm:existence of i}) of the statement, and let $m\geq 2$. Seeking a contradiction, we suppose that there exist $\bm v^1,\ldots, \bm v^m\in V_{\bm w}^+( d_{\bm w})$ such that the following holds:
    \[\bm u + \frac{(m-2)d_{\bm w}}{w_i}\bm e_i = \sum_{j=1}^m\bm v^j.\]
    Since the scalar product $\bm w\cdot \sum_{j=1}^m\bm v^j$ is equal to $\bm w\cdot \bm u + (m-2)d_{\bm w}$, for every $j\in\{1, \ldots, m\}$, the following inequalities hold:
    \begin{equation} \label{ineqs:bounded product}
    d_{\bm w}\leq \bm w\cdot \bm v^j \leq \bm w\cdot \bm u - d_{\bm w}.
    \end{equation}
    Now, by assumption, we observe that for all $j\in\{1, \ldots, m\}$ and for all $l\in\{1, \ldots, n\}\setminus\{i\}$, the inequality $v_l^j\leq u_l$ holds and moreover:
    \[(m-1)\frac{d_{\bm w}}{w_i}-1 = \mleft(\bm u + \frac{(m-2)d_{\bm w}}{w_i}\bm e_i\mright)_i = \sum_{j=1}^{m} v_i^j.\]
    In particular, there exists $j\in \{1, \ldots, m\}$ such that $v_i^j\leq \displaystyle\frac{d_{\bm w}}{w_i}-1 = u_i$, which implies that $\bm v^j\preceq \bm u$. Let us denote $\bm u' := \bm u-\bm v^j$: it is a nonnegative vector, and inequalities \eqref{ineqs:bounded product} give the following inequality:
    \[\bm w\cdot \bm u' = \bm w\cdot \bm u - \bm w\cdot \bm v^j \geq d_{\bm w}.\]
    Hence, $\bm u = \bm v^j + \bm u'\in 2V^+_{\bm w}(d_{\bm w})$, which is a contradiction. Hence $(\bm w, \bm u, 1)$ satisfies \cref{cond:combinatorial very ampleness}.
\end{proof}

\begin{proposition} \label{thm:bubbles with positive coordinates}
\Cref{tab:positive coordinates} in \cref{app:tables} contains all the \bubbles{$\bm w$} with positive coordinates for all the nonempty vectors $\bm w \in \mathbb Z_{\geq1}^n$ such that the inequalities $\max(\bm w) < \upperbound{}$ and $w_1 < \ldots < w_n$ hold. Moreover, every such \bubble{$\bm w$}~$\bm u$ satisfies both of the following:
\begin{enumerate}[label=\textup{(\alph*)}, ref=\alph*]
\item $\bm w \cdot \bm u < 3 d_{\bm w}$,
\item $(\bm w, \bm u, 1)$ satisfies \cref{cond:combinatorial very ampleness} if and only if $\bm u\notin 2V^+_{\bm w}(d_{\bm w})$.
\end{enumerate}
\end{proposition}

\begin{proof}
    (a) Since $\bm w$ is strictly ordered and $\max(\bm w) < \upperbound{}$, we have $n \leq \upperboundminusone$.
    For every $1 \leq n \leq \upperboundminusone$, we consider each vector $(w_1, \ldots, w_n) \in \mathbb Z_{\geq1}^n$ separately, where $1 \leq w_1 < \ldots < w_n \leq \upperboundminusone$.
    For a fixed~$n$, there are exactly $\binom{\upperboundminusone}{n}$ such vectors, so there are exactly $2^{\upperboundminusone} - 1$ vectors to consider in total.
    According to \Cref{algo: proof bubbles algorithm}, for each vector $\bm w$ as above, the output of the function \textsc{PositiveBubbles}($\bm w$) of \Cref{alg: Bubbles algorithm} in \Cref{app:algorithm} consists exactly of the \bubbles{$\bm w$} with positive coordinates.
    \Cref{tab:positive coordinates} in \cref{app:tables} is now obtained by a computer calculation.

    (b) If $(\bm w, \bm u, 1)$ satisfies \cref{cond:combinatorial very ampleness}, then from the definition, we find $\bm u \notin 2V^+_{\bm w}(d_{\bm w})$. Conversely, we can check that for every \bubble{$\bm w$} $\bm u$ from \cref{tab:positive coordinates}, there exists an index $i$ such that $u_i = \frac{d_{\bm w}}{w_i} - 1$.
    The claim now follows from \cref{thm:a sufficient condition for condition 2.2}.
\end{proof}

Below, we use \cref{thm:bubbles with positive coordinates} to prove \cref{thm:small weights}.
We remind that the sets $A, B, A_{\mathrm{w}}$ and $C$ are defined in \cref{def:sets}.
The sets $A_{\mathrm{w}}$ and $C$ are explicitly described in \cref{rem:small large}.

\begin{theorem} \label{thm:small weights}
Let $n$ and $\bm w\in\mathbb{Z}_{\geq 1}^n$ be positive integers. Then, all of the following hold:
\begin{enumerate}[label=\textup{(\alph*)}, ref=\alph*]
\item \label{itm:small weights - 3d} if $\max(\bm w) < \upperbound{}$, then every \bubble{$\bm w$} $\bm u$ satisfies $\bm w \cdot \bm u < 3 d_{\bm w}$,
\item \label{itm:small weights - graded} if $\max(\bm w) < \upperbound{}$, then there exists a \bubble{$\bm w$} $\bm u$ such that $d_{\bm w} \mid \bm w \cdot \bm u$ if and only if, up to permutation, a subsequence of $\bm w$ with the same least common multiple is in~$A_{\mathrm{w}}$,
\item \label{itm:small weights - large} if a nonempty subsequence of $\bm w$ with the same least common multiple is weakly equivalent to an element of~$A$, then there exists a \bubble{$\bm w$} $\bm u$ such that $\bm w \cdot \bm u = 2d_{\bm w}$ and $(\bm w, \bm u, 1)$ satisfies \cref{cond:combinatorial very ampleness},
\item \label{itm:small weights - Rees} if $\max(\bm w) < \upperbound{}$, then there exists a \bubble{$\bm w$} $\bm u \notin 2V^+_{\bm w}(d_{\bm w})$ if and only if, up to permutation, a subsequence of $\bm w$ with the same least common multiple is in~$C$,
\item \label{itm:small weights - small} if a nonempty subsequence of $\bm w$ with the same least common multiple is weakly equivalent to an element of $A$ or strongly equivalent to an element of~$B$, then there exists a \bubble{$\bm w$} $\bm u$ such that $(\bm w, \bm u, 1)$ satisfies \cref{cond:combinatorial very ampleness}.
\end{enumerate}
\end{theorem}

\begin{proof}
For $\bm u \in \mathbb Z_{\geq0}^n$, let $I_{\bm u}$ denote the set of all the indices $i$ such that $u_i$ is zero.
By \cref{thm:bubbles with positive coordinates} and \cref{cor:coordinates are positive}, if $\bm u$ is a \bubble{$\bm w$} where $\max(\bm w) < \upperbound{}$, then $d_{\bm w} = d_{\pi_{I_{\bm u}}(\bm w)}$.

\labelcref*{itm:small weights - 3d} By \cref{lem:projecting zeros}, $\pi_{I_{\bm u}}(\bm u)$ is a \bubble{$\pi_{I_{\bm u}}(\bm w)$}. By \cref{thm:bubbles with positive coordinates}, $\pi_{I_{\bm u}}(\bm w) \cdot \pi_{I_{\bm u}}(\bm u) < 3 d_{\bm w}$.

\labelcref*{itm:small weights - large} Let $\bm w'$ be such a nonempty subsequence of $\bm w$. Every vector $\bm v \in A$ appears in \cref{tab:positive coordinates} with a \bubble{$\bm v$} $\bm u$ with positive coordinates satisfying $\bm v \cdot \bm u = 2d_{\bm v}$.
By \cref{thm:bubbles with positive coordinates} for every pair $(\bm v, \bm u)$, the triple $(\bm v, \bm u, 1)$ satisfies \cref{cond:combinatorial very ampleness}.
Note that $\bm v$ is the unique minimal representative in its weak equivalence class (see \cref{lem:representative}\labelcref{itm:weak representative}).
The claim now follows from \cref{lem:condition}.

\labelcref*{itm:small weights - graded} By \labelcref{itm:small weights - large}, it suffices to prove the direction ``$\Longrightarrow$''.
Let $\bm w'$ be a subsequence of $\pi_{I_{\bm u}}(\bm w)$ obtained after repeatedly removing a weight which is equal to another weight, until no repeating weights exist.
By \cref{lem:positive equivalence}, there exists a \bubble{$\bm w'$} $\bm u'$ with positive coordinates satisfying $\bm w' \cdot \bm u' = 2 d_{\bm w'}$.
Therefore, up to permutation, $\bm w'$ appears in \cref{tab:positive coordinates} with a \bubble{$\bm w'$} $\bm u'$ satisfying $\bm w' \cdot \bm u' = 2d_{\bm w'}$.
We can check one-by-one that every such vector $\bm w'$ in \cref{tab:positive coordinates} has a subsequence in $A_{\mathrm{w}}$ with the same least common multiple.

\labelcref*{itm:small weights - small} Let $\bm w'$ be such a nonempty subsequence of $\bm w$. The case where $\bm w'$ is weakly equivalent to an element of $A$ is done in~\labelcref*{itm:small weights - large}. Every $\bm v\in B$ appears in \cref{tab:positive coordinates} with a \bubble{$\bm v$}~$\bm u$ satisfying both $\bm v\cdot \bm u < 3d_{\bm v}$ and $\bm u\notin V_{\bm v}^+(2d_{\bm v})$. By \cref{thm:bubbles with positive coordinates} for every pair $(\bm v, \bm u)$, the triple $(\bm v, \bm u, 1)$ satisfies \cref{cond:combinatorial very ampleness}.
Note that $\bm v$ is the unique minimal representative in its strong equivalence class (see \cref{lem:representative}\labelcref{itm:strong representative}).
The claim now follows from \cref{lem:condition}.

\labelcref*{itm:small weights - Rees} Similar to the proof of~\labelcref{itm:small weights - graded}.
\end{proof}

To conclude this section, we give a small application of \cref{thm:small weights}.

\begin{proposition} \label{thm:graded Rees}
Let $n$ and $\bm w\in\mathbb{Z}_{\geq 1}^n$ be positive integers. Every \bubble{$\bm w$} $\bm u$ satisfies
\[
\bm w \cdot \bm u < \operatorname{lcm}(1, 2, \ldots, \max(\bm w)).
\]
\end{proposition}

\begin{proof}
By \cref{lem:simplification bubbles}\labelcref{itm:strongly}, it suffices to prove the case where the weights are pairwise distinct. If $\max(\bm w) < 7$, then the result follows from \cref{thm:small weights}. Otherwise, by \cref{thm:maximum all weights}, it suffices to note that for all integers $k \geq 7$, we have the following inequality:
\[
\operatorname{lcm}(1, 2, \ldots, k) > (k - 2) (1 + 2 + \ldots + k). \qedhere
\]
\end{proof}

\section{Nullity of the set of weights with a bubble} \label[sec]{sec:nullity}

In this section, we show that the asymptotic density of weight vectors which admit bubbles is zero.
In what follows, given a positive integer $n$ we denote by $\{1, \ldots, M\}^n$ the set of all vectors $(w_1, \ldots, w_n)$ of positive integers such that the inequality $w_i \leq M$ holds for every $i \in \{1, \ldots, n\}$.
Let $\log(x)$ denote the natural logarithm of~$x$.

\begin{lemma}\label{lem:upperboundforgcds}
    Let $\psi\colon \mathbb R_{>0} \to \mathbb R_{>0}$ be an unbounded increasing function.
    Then, for all positive integers $n\geq 2$ and sufficiently large positive integers~$M$, the following holds:
    \[\mleft|\mleft\{ \bm w\in \{1, \ldots, M\}^n \;\middle|\; \begin{gathered}
      \exists\,1\leq i< j \leq n\colon\\
      \gcd(w_i, w_j) \geq n^2 \psi(M) \log M
    \end{gathered} \mright\}\mright| < \frac{M^n}{2 \psi(M)}.\]
\end{lemma}

\begin{proof}
Let $B(M)$ denote the average greatest common divisor of two positive integers between $1$ and~$M$, that is, $B(M) := \frac{\sum_{a, b \in \{1, \ldots, M\}} \gcd(a, b)}{M^2}$.
If the \lcnamecref{lem:upperboundforgcds} did not hold, then there would exist $i < j$ and arbitrarily large $M$ such that there are at least $\frac{2}{n(n-1)} \cdot \frac{M^2}{2 \psi(M)}$ pairs $(w_i, w_j) \in \{1, \ldots, M\}^2$ with $\gcd(w_i, w_j) \geq n^2 \psi(M) \log M$. Therefore, for arbitrarily large~$M$, we would have $B(M) \geq \log M$. This contradicts \cite[equation~(16)]{Tót10} which says that $B(M)$ is $\frac{6 \log M}{\pi^2} + \mathcal O(1)$.
\end{proof}

\begin{lemma}\label{lem:lowerboundforweights}
    Let $\psi\colon \mathbb R_{>0} \to \mathbb R_{>0}$ be a function.
    Then, for all positive integers $n$ and~$M$, the following holds:
    \[\mleft| \mleft\{ \bm w\in \{1, \ldots, M\}^n \;\middle|\; \begin{gathered}
      \exists i\in\{1,\ldots n\}\colon\\
      w_i \leq \frac{M}{2n\psi(M)}
    \end{gathered} \mright\}\mright| \leq \frac{M^n}{2\psi(M)}.\]
\end{lemma}

\begin{proof}
    The set in the left-hand side is in the union of the sets~$S_i$, where $S_i$ denotes the set of vectors $\bm w \in \{1, \ldots, M\}^n$ with $w_i \leq \frac{M}{2n \psi(M)}$.
    Since $S_i$ has cardinality at most $\frac{M^n}{2n\psi(M)}$, the \lcnamecref{lem:lowerboundforweights} follows.
\end{proof}

\begin{theorem}\label{thm:no bubbles}
    For all positive integers $n$ and~$M$, we define
    \[p(n, M) := \frac{\mleft|\left\{\bm w\in \{1,2,\ldots, M\}^n \mid \text{there are no \(\bubbles{\bm w}\)}\right\}\mright|}{M^n}.\]
    Then, the following holds:
    \[\lim_{M\to+\infty}p(n, M) = 1.\]
\end{theorem}

\begin{proof}
    The cases $n\in \{1,2\}$ follow from \cref{thm:multiple chains}.
    Now, let $n \geq 3$, let $\varepsilon < 1/6$ be a positive real number and let $\psi(x) := x^\varepsilon$.
    For all positive integers~$M$, define
    \[C(n, \varepsilon, M) := \left\{\bm w\in\{1,\ldots, M\}^n \;\middle|\; \begin{gathered}
      \forall\,1 \leq i < j \leq n\colon\\
      \gcd(w_i, w_j) < n^2 \psi(M) \log M,\\
      w_i, w_j > \frac{M}{2n\psi(M)}
    \end{gathered}\right\}.\]
    Let $\bm w\in C(n, \varepsilon, M)$. We compute
    \begin{align*}
        d_{\bm w} &= \operatorname{lcm}(w_1, \dots, w_n)\\
                  &\geq \operatorname{lcm}(w_1, w_2, w_3)\\
                  &=\frac{w_1w_2w_3\gcd(w_1,w_2,w_3)}{\gcd(w_1,w_2)\gcd(w_1,w_3)\gcd(w_2,w_3)}\\
                  &> \frac{M^3}{(2 n^3 \psi(M)^2 \log M)^3}
    \end{align*}
    Since $\psi(M) \leq M^\varepsilon$, there exists a positive integer $M_0$ depending only on $n$ such that for all $M \geq M_0$, we have
    \begin{equation} \label{eqn:d}
    d_{\bm w} > nM^{2+1/6-\varepsilon}.
    \end{equation}
    We have the inequality
    \begin{equation} \label{eqn:sum}
    (\max(\bm w)-2)\sum_{1\leq i\leq n}w_i < nM^2.
    \end{equation}
    By the inequalities \eqref{eqn:d} and \eqref{eqn:sum} and \cref{thm:maximum all weights}, for all $M \geq M_0$, there are no \bubbles{\(\bm w\)}.
    By \cref{lem:upperboundforgcds,lem:lowerboundforweights} and since $\psi(M) \geq M^\varepsilon$, for all positive real numbers $\varepsilon < 1/6$ and large enough integers~$M$, the inequalities
    \begin{equation}
    1 - M^{-\varepsilon} \leq \frac{|C(n, \varepsilon, M)|}{M^n}\leq p(n, M) \leq 1
    \end{equation}
    hold.
    Thus, the limit $\lim_{M\to+\infty}p(n, M)$ exists and is equal to~1.
\end{proof}

\section{Infinitude of the set of weights with a bubble} \label[sec]{sec:infinitude}

In this section, we construct infinitely many examples of weight vectors which admit a bubble. These examples are at the heart of the proof of \cref{thm:infinitely many bubbles graded,thm:infinitely many bubbles rees}.

\begin{lemma}\label{lem:unique alpha}
    Let $n\geq 2$ and let $\bm p:=(p_1, \ldots, p_n)$ be pairwise coprime positive integers greater than~1. Then, there exists a unique $\bm a \in \mathbb{Z}_{\geq0}^n$ with $a_i < p_i$ for all $i \in \{1, \ldots, n\}$ such that
    \[\sum_{i=1}^n\frac{a_i}{p_i}+\prod_{i=1}^n\frac{1}{p_i}\]
    is an integer, which we denote by~$k_{\bm p}$. Moreover, for every $i \in \{1, \ldots, n\}$, $a_i$ is coprime to $p_i$ and the following holds:
    \[
    \left\lceil\sum_{j\in\{1,\ldots, n\}\setminus\{i\}}\frac{a_j}{p_j}\right\rceil = k_{\bm p}.
    \]
\end{lemma}

\begin{proof}
    Denote $d:= d_{\bm p} = \prod_{i=1}^np_i$ and for all $i \in \{1, \ldots, n\}$, denote $w_i := \prod_{j\in\{1,\ldots, n\}\setminus\{i\}}p_j$.

    We prove the existence and uniqueness of~$\bm a$.
    The rational number
    \[k := \sum_{i=1}^n\frac{a_i}{p_i}+\prod_{i=1}^n\frac{1}{p_i} \]
    is an integer if and only if the integer $1 + \sum_{i=1}^nw_ia_i$ is divisible by~$d$.
    By the Chinese Remainder Theorem, equivalently, for all $i \in \{1, \ldots, n\}$, the congruence
    \begin{equation} \label{eqn:u is coprime}
    a_i \equiv -w_i^{-1} \mod{p_i}
    \end{equation}
    holds. Since we require each coordinate $a_i$ to satisfy $0 \leq a_i \leq p_i - 1$, this proves the existence and uniqueness of $a_i$.

    From \eqref{eqn:u is coprime}, we find that $\gcd(a_i, p_i)=1$.

    Finally, let $i \in \{1, \ldots, n\}$. Since $n\geq 2$ and $0 < a_i < p_i$ hold, the inequalities
    \[0<\frac{a_i}{p_i} + \prod_{j=1}^n\frac{1}{p_j} < 1\]
    hold too. Therefore, the inequalities
    \[k_{\bm p}-1 < \sum_{j\in\{1,\ldots,n\}\setminus\{i\}}\frac{a_j}{p_j} < k_{\bm p}\]
    also hold, giving us the desired result.
\end{proof}

\begin{lemma} \label{thm:small modification}
    Let $n\geq 3$ be an integer and let $p_2, \ldots, p_n$ be pairwise coprime positive integers greater than 1. Let $a_2, \ldots, a_n$ be any integers with $0 <  a_i < p_i$ and $\gcd(a_i, p_i)=1$ for all $i\in\{2, \ldots, n\}$.
    Then, there exists a unique integer $k$ in $\{1,\ldots n-1\}$ such that the following inequalities hold:
    \[k-1 < \sum_{i=2}^{n}\frac{a_i}{p_i} < k.\]
    Moreover, there exists a unique integer $p_1' \in \{2, 3, \ldots, 1 + \prod_{j=2}^n p_j\}$ that is coprime to $\prod_{j=2}^n p_j$ such that for every positive integer $p_1 \geq 2$, the equality
    \[
         \sum_{i = 1}^{n} \frac{a_i}{p_i} + \prod_{i = 1}^n \frac{1}{p_i} = k
    \]
    holds for some $a_1\in\{1, \ldots, p_1-1\}$ if and only if $p_1 \equiv p_1' \mod{\prod_{j=2}^n p_j}$.
\end{lemma}

\begin{proof}
    The uniqueness of $k$ is clear. To prove the existence of~$k$, note that if $\sum_{i=2}^{n}\frac{a_i}{p_i}$ were an integer, then $p_i$ would divide~$a_i$, a contradiction since $p_i \geq 2$.

    For all integers $p_1 \geq 2$ and $a_i \in \{1, \ldots, p_1 - 1\}$, the rational number $\sum_{i=1}^n \frac{a_i}{p_i} + \prod_{i=1}^n \frac{1}{p_i}$ is an integer if and only if $a_1$ is the inverse of $q := -\prod_{i=2}^n p_i$ modulo $p_1$ and $p_1$ is a solution of the following system of congruence equations in the variable $x\in\mathbb{Z}$:
    \[
    \mleft\{\begin{aligned}
      x & \equiv (a_2 q/p_2)^{-1}\mod p_2,\\
      x & \equiv (a_3 q/p_3)^{-1}\mod p_3,\\
      & ~\, \vdots\\
      x & \equiv (a_n q/p_n)^{-1}\mod p_n.
    \end{aligned}\mright.
    \]
    Note that if $\sum_{i=1}^n \frac{a_i}{p_i} + \prod_{i=1}^n \frac{1}{p_i}$ is an integer, then it is necessarily equal to~$k$.
    The \lcnamecref{thm:small modification} now follows from the Chinese remainder theorem.
\end{proof}

\begin{corollary} \label{thm:infinitely many for each k}
    Let $n\geq 2$ be an integer and let $k\in\{1, \ldots, n-1\}$. Then, there are infinitely many tuples $\bm p:=(p_1, \ldots, p_n)$ of pairwise distinct prime numbers such that $k_{\bm p}$ (defined in \cref{lem:unique alpha}) is equal to $k$.
\end{corollary}

\begin{proof}
    Fix pairwise coprime positive integers $p_2, p_3, \ldots, p_{k+1}$ greater than~$k$.
    For all sufficiently large positive integers $p_{k+2}, p_{k+3}, \ldots, p_n$ such that $p_2,\ldots, p_n$ are pairwise coprime, defining $a_i = p_i - 1$ for $i \in \{2, 3, \ldots, k+1\}$ and $a_i = 1$ for $i \in \{k+2, k+3, \ldots, n\}$, we find
    \[
       k-1 < \sum_{i \in \{2, \ldots, n\}} \frac{a_i}{p_i} < k.
    \]
    Choosing $p_2, \ldots, p_n$ to be prime numbers, the \lcnamecref{thm:infinitely many for each k} now follows from \cref{thm:small modification} and Dirichlet's theorem on arithmetic progressions.
\end{proof}

For our purposes, the cases with $k_{\bm p} \in \{1, n-1\}$ are the most interesting.
These cases are the rarest in the following sense:
\begin{proposition} \label{prop:eulerian numbers}
Let $n$ and $k$ be integers satisfying $1 \leq k \leq n-1$.
Let $T$ denote one of the following:
\begin{enumerate}
\item the set of vectors $\bm p = (p_1, \ldots, p_n) \in \mathbb Z_{\geq2}^n$ of pairwise coprime integers,
\item the set of vectors $\bm p = (p_1, \ldots, p_n) \in \mathbb Z_{\geq2}^n$ of pairwise coprime integers such that $p_2, \ldots, p_n$ are prime powers,
\item the set of vectors $\bm p = (p_1, \ldots, p_n) \in \mathbb Z_{\geq2}^n$ of pairwise coprime integers such that $p_2, \ldots, p_n$ are prime numbers.
\end{enumerate}
For every integer $M$ and $l \in \{1, \ldots, n-1\}$, let $T_l(M)$ denote the set of all vectors $\bm p \in T$ such that $k_{\bm p} = l$, $p_1 \leq \operatorname{lcm}(p_2, \ldots, p_n) + 1$ and $p_i \leq M$ for all $i \in \{2, \ldots, n\}$.
Then,
\[
\lim_{M \to \infty} \frac{\abs{T_k(M)}}{\sum_{l=1}^{n-1}\abs{T_l(M)}} = A(n-1, k-1),
\]
where $A(n-1, k-1)$ denotes the \emph{Eulerian number}, meaning the number of permutations of $1, 2, \ldots, n-1$ with $k-1$ ascents.
\end{proposition}

\begin{proof}
For every $\bm q := (q_2, \ldots, q_n) \in \mathbb Z_{\geq1}^{n-1}$ and $l \in \{0, 1, \ldots, n-1\}$, let $U_l(\bm q)$ denote the set of all vectors $(a_2, \ldots, a_n) \in \mathbb Z_{\geq0}^{n-1}$ such that $\lfloor\frac{a_2}{q_2} + \ldots + \frac{a_n}{q_n}\rfloor = l$ and $a_i < q_i$ for all~$i$.
By \cref{thm:small modification}, instead of considering vectors in $T_l(M)$, we can consider pairs of vectors $(\bm q, \bm a) \in \mathbb Z_{\geq1}^{n-1} \times \mathbb Z_{\geq0}^{n-1}$ where $\bm a \in U_l(\bm q)$.
The number of vectors $\bm a \in U_l(\bm q)$ such that $\gcd(q_i, a_i) = 1$ for all $i$ is the sum of $(-1)^{\omega(\prod_i g_i)} \abs{U_l(\frac{q_2}{g_2}, \ldots, \frac{q_n}{g_n})}$ over the vectors $\bm g = (g_2, \ldots, g_n) \in \mathbb Z_{\geq1}^{n-1}$ such that for all~$j$, $g_j$ is square-free and divides~$q_j$, where $\omega$ denotes the prime omega function.
Summing only over the vectors $\bm g$ that satisfy $\frac{q_i}{g_i} \geq \sqrt{M}$ for all $i$ does not change the desired limit. Therefore, it suffices to prove that for all $\bm q$ with $q_i \geq \sqrt{M}$ for all~$i$, we have
\[
\abs{\frac{\abs{U_k(\bm q)}}{\sum_{l=0}^{n-1} \abs{U_l(\bm q)}} - A(n-1, k-1)} < \frac{n}{\sqrt{M}}.
\]
Let $\mathcal R_k$ denote the $k$th slice of the unit $(n-1)$-cube, that is,
\[
\mathcal R_k := \{(x_2, \ldots, x_n) \in [0, 1]^{n-1} \mid k \leq x_2 + \ldots + x_n \leq k+1\}.
\]
By \cite[Proposition~7.1]{Pet15}, we have $A(n-1, k-1) = \operatorname{vol}(\mathcal R_{k-1})$.
The \lcnamecref{prop:eulerian numbers} follows by approximating $\frac{\abs{U_k(\bm q)}}{\sum_{l=0}^{n-1} \abs{U_l(\bm q)}}$ by the volume of the convex hull of the points $(\frac{a_2}{q_2}, \ldots, \frac{a_n}{q_n})$ with $\bm a \in U_k(\bm q)$.
Details omitted.
\end{proof}

We can inductively construct examples of $\bm p$ with $k_{\bm p} = k$ for any $k \in \{1, \ldots, n-1\}$.

\begin{lemma} \label{lem:inductive}
Let $n \geq 2$ be an integer and let $\bm p \in \mathbb Z_{\geq2}^n$ be a vector of pairwise coprime integers. Let $\bm a \in \mathbb Z_{\geq0}^n$ be the unique vector with $a_i < p_i$ for all $i \in \{1, \ldots, n\}$ such that
\[
\sum_{i=1}^n \frac{a_i}{p_i} + \prod_{i=1}^n \frac{1}{p_i} = k_{\bm p}.
\]
Define $a_{n+1} := 1$. Then, both of the following hold:
\begin{enumerate}[label=\textup{(\alph*)}, ref=\alph*]
\item if $p_{n+1} = 1 + \prod_{i=1}^n p_i$, then
\[
\sum_{i=1}^{n+1} \frac{a_i}{p_i} + \prod_{i=1}^{n+1} \frac{1}{p_i} = k_{\bm p},
\]
\item if $p_{n+1} = -1 + \prod_{i=1}^n p_i$, then
\[
\sum_{i=1}^{n+1} \frac{p_i - a_i}{p_i} + \prod_{i=1}^{n+1} \frac{1}{p_i} = n + 1 - k_{\bm p}.
\]
\end{enumerate}
\end{lemma}

\begin{proof}
Direct computation.
\end{proof}

\begin{remark}\label{rem:after lemma inductive}
Let $2 \leq n < m$ be integers, let $\bm p, \bm q \in \mathbb Z_{\geq2}^n$ satisfy $k_{\bm p} = 1$ and $k_{\bm q} = n-1$ and let $k \in \{1, m-1\}$.
By repeatedly applying \cref{lem:inductive}, we can construct a vector $\bm r \in \mathbb Z_{\geq2}^m$ of pairwise coprime integers such that $k_{\bm r} = k$.
\end{remark}

Finally, we show that the vectors $\bm w_{\bm p}$ and $\bm v_{\bm q}$ (defined respectively in \cref{exa:sharp graded,exa:sharp Rees}) associated to respectively $\bm p$ and $\bm q$ with $k_{\bm p} = 1$ and $k_{\bm q} = n-1$ admit high weight bubbles.

\begin{theorem} \label{thm:infinitely many bubbles graded}
    Let $n \geq 4$ be an integer.
    Let $\bm p := (p_2, \ldots, p_n) \in \mathbb Z_{\geq2}^{n-1}$ be a vector of pairwise coprime integers with $k_{\bm p} = n-2$.
    Define $\bm w := \bm w_{\bm p}$ as in \cref{exa:sharp graded}.
    Then, all of the following hold:
        \begin{enumerate}[label=\textup{(\alph*)}, ref=\alph*]
            \item $\bm w$ is well-formed,
            \item there exists a \bubble{$\bm w$} $\bm u$ with $\bm w \cdot \bm u = (n-2)d_{\bm w}$ such that $(\bm w, \bm u, n-3)$ satisfies \cref{cond:combinatorial very ampleness}, and
            \item every \bubble{$\bm w$} $\bm u'$ satisfies $\bm w \cdot \bm u < (n-1)d_{\bm w}$.
        \end{enumerate}
\end{theorem}

\begin{proof}
    (a) Follows from the fact that $p_2,\ldots,p_n$ are pairwise coprime.

    (c) Follows from \cref{thm:multiple chains}.

    (b) First, we construct $\bm u \in \mathbb Z_{\geq1}^n$ such that $\bm w \cdot \bm u = (n-2)d_{\bm w}$. By \cref{lem:unique alpha}, there exists a unique vector $\bm{a}  := (a_2, \ldots, a_{n})\in\mathbb{Z}_{\geq 1}^{n-1}$ such that $a_i < p_i$ for all $i\in\{2,\ldots,n\}$ and the equality
    \[\sum_{i=2}^{n}\frac{a_i}{p_i} + \prod_{i=2}^{n}\frac{1}{p_i} = n-2\]
    holds. We define
    \[
    u_1 := 1\quad \text{and}\quad \forall i\in\{2,\ldots, n\}\colon \quad u_i := a_i.
    \]
    By construction, it follows that
    \[
        \bm w\cdot\bm u = (n-2)d_{\bm w}.
    \]

    Next we prove that $\bm u$ is a \bubble{$\bm w$}.
    Seeking a contradiction, we assume that there exists $\bm v\in\mathbb{Z}_{\geq 0}^{n}$ such that $\bm w \cdot \bm v = d_{\bm w}$ and $\bm v\preceq \bm u$. We consider three cases depending on the value of $v_1$ and derive a contradiction in each case.
    \begin{enumerate}[label=\textup{(\arabic*)}, ref=\arabic*]
        \item If $v_{1} = 0$, then for all $i\in\{2, \ldots, n\}$, $p_{i}$ divides $v_{i}$ and there exists $j \in \{2, \ldots, n\}$ such that $v_j > 0$. Hence $v_{j}\geq p_{j}$, contradicting that $\bm v\preceq \bm u$.
        \item If $v_{1} = 1$, we define $\bm v' := \pi_{1}(\bm v)\in \mathbb{Z}_{\geq 0}^{n-1}$ and observe that
        \[\pi_{1}(\bm w)\cdot \bm v' = d_{\bm w}-1.\]
        This contradicts the uniqueness of $k_{\bm p}$ from \cref{lem:unique alpha} and the assumption that $k_{\bm p} = n-2 > 1$.
        \item If $v_{1} > 1$, then $\bm v\not\preceq \bm u$, a contradiction.
    \end{enumerate}

    Finally, we prove that $(\bm w, \bm u, n-3)$ satisfies \cref{cond:combinatorial very ampleness}. Fix any $i\in\{2,\ldots, n\}$.
    We proceed by proof by contradiction. Assume that there exists $m\geq 2$ such that there exist $\bm v^1, \ldots, \bm v^m\in V^+_{\bm w}((n-3)d_{\bm w})$ satisfying the following equation:
    \begin{equation} \label{eqn:left-hand side well-formed}
    \bm u' := \bm u + \bigl((n-3)m-(n-2)\bigr)p_i \bm e_{i} = \sum_{j=1}^m\bm v^j.
    \end{equation}
    Note that since $\bm w\cdot \bm u' = (n-3)md_{\bm w}$ we have that $\bm v^j\in V_{\bm w}((n-3)d_{\bm w})$ for all $j\in\{1, \ldots, m\}$.
    Now, since $u'_{1} = u_{1} = 1$, there exists a unique $j_0\in \{1, \ldots, m\}$ such that $v_{1}^{j_0} = 1$ and $v_{1}^j = 0$ for all $j\in \{1, \ldots, m\}\setminus \{j_0\}$. Without loss of generality, we assume $j_0 = m$. Then, for all $l\in \{2,\ldots, n\}\setminus\{i\}$ and for all $j\in \{1, \ldots, m-1\}$, we have $v^j_{l}\leq u_{l} < p_l$ and $v^j_{l}$ is divisible by~$p_l$. This implies the following:
    \begin{equation} \label{eqn:right-hand side well-formed}
    \bm v^1 = \ldots = \bm v^{m-1} = (n-3)p_i\bm e_{i}.
    \end{equation}
    Using \eqref{eqn:left-hand side well-formed} and \eqref{eqn:right-hand side well-formed} together with $u_{i}' = \sum_{j=1}^m v_{i}^j$, we find
    \[u_{i} + \bigl((n-3)m-(n-2)\bigr)p_i = (m-1)(n-3)p_{i} + v^m_{i}.\]
    In particular, we obtain that
    \[p_i + v^m_{i} = u_{i} < p_i.\]
    This is absurd since $v_{i}^m$ must be nonnegative.
\end{proof}

\begin{theorem}\label{thm:infinitely many bubbles rees}
    Let $n \geq 3$ be an integer.
    Let $\bm p := (p_1, \ldots, p_n) \in \mathbb Z_{\geq2}^n$ be a vector of pairwise coprime integers with $k_{\bm p} = 1$.
    Define $\bm w := \bm v_{\bm p}$ as in \cref{exa:sharp Rees}.
    Then, both of the following hold:
    \begin{enumerate}[label=\textup{(\alph*)}, ref=\alph*]
        \item there exists a \bubble{$\bm w$} $\bm u$ with $\bm w \cdot \bm u = (n-1)d_{\bm w} + 1$ such that $(\bm w, \bm u, n-2)$ satisfies \cref{cond:combinatorial very ampleness}, and
        \item every \bubble{$\bm w$} $\bm u'$ satisfies $\bm w \cdot \bm u' < nd_{\bm w}$.
    \end{enumerate}
\end{theorem}

\begin{proof}
    (b) Follows from \cref{thm:multiple chains}.

    (a) First, we construct a \bubble{$\bm w$} $\bm u \in \mathbb Z_{\geq1}^n$ such that $\bm w \cdot \bm u = (n-1)d_{\bm w} + 1$.
    By \cref{lem:unique alpha}, there exists a unique vector $\bm a \in \mathbb Z_{\geq1}^n$ such that $a_i < p_i$ for all $i \in \{1, \ldots, n\}$ and $k_{\bm p} = 1$. This implies
    \[\bm w\cdot\bm a = d_{\bm w} - 1.\]
    If one defines, for all $i\in \{1,\ldots, n\}$,
    \[u_i := p_i-a_i,\]
    then we obtain
    \[\bm w\cdot \bm u = (n-1)d_{\bm w} + 1.\]
    If $\bm v \in \mathbb Z_{\geq0}^n$ is such that $\bm w \cdot \bm v = d_{\bm w}$, then there exists $i \in \{1, \ldots, n\}$ such that $p_i$ divides $v_i$ and $v_i > 0$, showing that $\bm v \npreceq \bm u$. This shows that $\bm u$ is a \bubble{$\bm w$}.

    Now, we prove that $(\bm w, \bm u, n-2)$ satisfies \cref{cond:combinatorial very ampleness}. Fix any $i\in\{1,\ldots, n\}$. We proceed once again by proof by contradiction. Assume that there exists $m\geq 2$ such that there exist $\bm v^1, \ldots, \bm v^m\in V^+_{\bm w}((n-2)d_{\bm w})$ satisfying the following equation:
    \begin{equation} \label{eqn:left-hand side arbitrary}
    \bm u' := \bm u + \bigl((n-2)m-(n-1)\bigr)p_i \bm e_{i} = \sum_{j=1}^m\bm v^j.
    \end{equation}
    Note that since $\bm w\cdot \bm u' = (n-2)md_{\bm w} + 1$, there exists $j_0\in \{1, \ldots, m\}$ such that $\bm w\cdot \bm v^{j_0} = (n-2)d_{\bm w}+1$ and $\bm v^j\in V_{\bm w}((n-2)d_{\bm w})$ for all $j\in\{1, \ldots, m\}\setminus\{j_0\}$. Without loss of generality, we may assume that $j_0 = m$.
    For all $l\in \{1,\ldots, n\}\setminus\{i\}$ and for all $j\in \{1, \ldots, m-1\}$, we have $v^j_{l}\leq u_{l} < p_l$ and $v^j_{l}$ is divisible by~$p_l$. This implies the following:
    \begin{equation} \label{eqn:right-hand side arbitrary}
    \bm v^1 = \ldots = \bm v^{m-1} = (n-2)p_i\bm e_{i}.
    \end{equation}
    Using \eqref{eqn:left-hand side arbitrary} and \eqref{eqn:right-hand side arbitrary} together with $u_{i}' = \sum_{j=1}^m v_{i}^j$, we find
    \[u_i + \bigl((n-2)m-(n-1)\bigr)p_i = (m-1)(n-2)p_i + v^m_i.\]
    In particular, we obtain that
    \[p_i + v^m_i = u_i < p_i.\]
    This is absurd since $v_i^m$ must be nonnegative.
\end{proof}

\appendix

\section{Algorithms} \label[app]{app:algorithm}

\begin{algorithm}[H]
\caption{Recursive}\label{alg: Recursive algorithm}
\DontPrintSemicolon
\Input{
\begin{itemize}
\item A vector $\bm w\in\mathbb{Z}_{\geq1}^n$,
\item a vector $\variable{mins}\in\mathbb{Z}^n$,
\item a vector $\variable{maxs}\in\mathbb{Z}^n$,
\item an integer $\variable{min\_product}$,
\item an arranged marked list of vectors $\variable{thresholds}\subseteq \mathbb{Z}_{\geq0}^n$ and
\item a boolean $\variable{return\_one}\in\{\variable{true},\,\variable{false}\}$.
\end{itemize}
}
\Output{If $\variable{return\_one} = \variable{false}$, then returns the list of all vectors $\bm u\in \mathbb{Z}^n$ satisfying the inequalities $\variable{mins} \preceq u \preceq \variable{maxs}$ and $\variable{min\_product} \leq \bm w\cdot \bm u$ such that there are no vectors $\bm v\in \variable{thresholds}$ with $\bm v\preceq \bm u$. If $\variable{return\_one} = \variable{true}$, then returns exactly one such vector if it exists, otherwise returns the empty list.}
    $\variable{a}\leftarrow \max\mleft(\variable{mins}_n,\, \mleft\lceil\dfrac{\variable{min\_product} - \pi_n(\bm w) \cdot \pi_n(\variable{maxs})}{w_n}\mright\rceil\mright)$.\label{line: lower bound a}\\
    $\variable{b}\leftarrow \min(\variable{maxs}_n, v_n-1)$ where $\bm v$ is the first element of $\variable{thresholds}$.\label{line: upper bound b}\\
    \If{$\variable{a} > \variable{b}$}
        {
            \Return [].\label{line: return empty list}
        }
    \If{$n = 1$\label{line: terminating case}}
        {
            \If{$\variable{return\_one} = \variable{true}$}
                {
                    \Return $[\variable{a}\bm e_1]$.
                }
            \Else
                {
                    \Return $[\variable{a}\bm e_1,\, (\variable{a}+1)\bm e_1,\,
                    \ldots,\, \variable{b}\bm e_1]$.
                }
        }

    $L\leftarrow []$.

    \For{$u_n\in \{\variable{a}, \variable{a}+1,\ldots, \variable{b}\}$\label{line: safe interval}}
        {
            $\variable{low} \leftarrow \variable{min\_product} - u_n w_n$.\label{line: new min product}\\
            $\variable{thrs}\leftarrow \{\pi_n(\bm v) \mid \bm v\in \variable{thresholds},\, v_n \leq u_n\}$.\label{line: new thresholds}\\
            Let $\widetilde{\variable{thrs}}$ be an arrangement of $\variable{thrs}$.\label{line: arrangement thrs}\\
            $V'\leftarrow \textsc{Recursive}(\pi_n(\bm w),\, \pi_n(\variable{mins}),\, \pi_n(\variable{maxs}),\,  \variable{low},\, \widetilde{\variable{thrs}},\, \variable{return\_one})$\\
            \For{$\bm u'\in  V'$}
                {
                    \textbf{append} $\kappa_n(\bm u') + u_n\bm e_n$ to $L$.\label{line: recursive buildup}\\
                    \If{$\variable{return\_one} = \variable{true}$}
                        {
                            \Return $L$.
                        }
                }
        }
    \Return $L$.
\end{algorithm}

\begin{algorithm}[H]
\caption{PositiveBubbles}\label{alg: Bubbles algorithm}
\DontPrintSemicolon
\Input{A nondecreasing vector $\bm w \in\mathbb{Z}_{\geq1}^n$.}
\Output{The complete list of all \bubbles{$\bm w$}~$\bm u$ with positive coordinates.}
    \If{$n < 3$\label{line: step 1}}
        {
            \Return $[]$.
        }
    $\bm w \leftarrow \frac{1}{\gcd(w_1,\ldots, w_n)}\cdot\bm w$.\label{line: step 2}\\

    For every $i\in \{1, \ldots, n\}$, define
    $
    \alpha_i := \mleft\{\begin{aligned}
        & w_{n} - 2 && \text{if $i<n$},\\
        & w_{n-1} - 2 && \text{otherwise}.
    \end{aligned}\mright.
    $

    \If{$\bm w \cdot \bm{\alpha} < 2d_{\bm w}$\label{line: step 3}}
        {
            \Return $[]$.\label{line: terminate step 3}
        }
    For every $i\in \{1, \ldots, n\}$, define:
    \begin{align*}
        \varepsilon_i & := \min\mleft(\alpha_i,\, \frac{d_{\bm w}}{w_i}-1\mright),\\
        \zeta_i & := \min \mleft(\{\varepsilon_i + 1\} \cup \mleft\{ \frac{w_j}{w_i} \;  \middle|\; j \in \{i+1, \ldots, n\},\, w_i \mid w_j \mright\}\mright),\\
        \beta_i & := \min\mleft(\varepsilon_i,\, \zeta_i - 1\mright).
    \end{align*}

    \If{$\bm w\cdot \bm{\beta} < 2d_{\bm w}$\label{line: step 4}}
        {
            \Return $[]$.
        }
    For every $i\in \{1, \ldots, n\}$, use  \cref{not:bound on u when all entries are positive} to define:
            \[
                \begin{aligned}
                    \delta_i & := \mleft\{\begin{aligned}
                                 & \min\left(\varepsilon_i,~ \mleft\lfloor \frac{\max(\pi_i(\bm w)) + c_{\bm w, i}}{w_i}\mright\rfloor - 2\right) && \text{if \(\sum_{j \in \{1, \ldots, n\} \setminus \{i\}} w_j < d_{\bm w}+w_i\)},\\
                                & \varepsilon_i && \text{otherwise},
                        \end{aligned}\mright.\\
                    \gamma_i & := \min\left(\delta_i, \zeta_i\right).
                \end{aligned}
            \]

            \If{$\exists i \in \{1, \ldots, n\}\colon \gamma_i \leq 0$}
                {
                    \Return $[]$.
                }

            \If{$\bm w\cdot \bm{\gamma} < 2d_{\bm w}$\label{line: step 5}}
                {
                    \Return $[]$.
                }

    $\variable{W}_{\bm \gamma} \leftarrow \{(\gamma_i + 1)\bm e_i\mid i\in\{1,\ldots, n\}\}\cup \{\bm v\in V_{\bm w}(d_{\bm w}) \mid \bm v \preceq \bm \gamma\}$.\\
    Let $\widetilde{\variable{W}_{\bm \gamma}}$ be an arrangement of $\variable{W}_{\bm \gamma}$.\\
    $\variable{mins} \leftarrow \sum_{i=1}^n\bm e_i$.\\
    \If{$\textnormal{\textsc{Recursive}}(\bm w,\, \variable{mins},\, \bm \gamma,\, 2d_{\bm w},\, \widetilde{\variable{W}_{\bm \gamma}},\,  \variable{true}) = \emptyset$\label{line: step 6}}
        {
            \Return $[]$.
        }
    $\variable{W}_{\bm \delta} \leftarrow \{(\delta_i + 1)\bm e_i\mid i\in\{1,\ldots, n\}\}\cup \{\bm v\in V_{\bm w}(d_{\bm w}) \mid \bm v \preceq \bm \delta\}$.\\
    Let $\widetilde{\variable{W}_{\bm \delta}}$ be an arrangement of $\variable{W}_{\bm \delta}$.\\
    \Return $\textnormal{\textsc{Recursive}}(\bm w,\, \variable{mins},\, \bm \delta,\, 2d_{\bm w},\, \widetilde{\variable{W}_{\bm \delta}},\,\variable{false})$.\label{line: step 7}
\end{algorithm}

\section{Tables} \label[app]{app:tables}

\Cref{tab:positive coordinates} contains all the \bubbles{$\bm w$} with positive coordinates where $\bm w \in \mathbb Z_{\geq1}^n$ is such that $\max(\bm w) < \upperbound{}$ and $w_1 < \ldots < w_n$ (see \cref{thm:bubbles with positive coordinates} for the proof).
For every such~$\bm w$, the \bubbles{$\bm w$} $\bm u$ in \cref{tab:positive coordinates} are sorted by the value of $\bm w \cdot \bm u$ with the first bubble $\bm v$ having the least $\bm w \cdot \bm v$.
\colorlet{ColorExactlyTwoD}{black}%
\colorlet{ColorTwoDPlusOne}{black}%
\colorlet{ColorTwoDPlusTwo}{black}%
\colorlet{ColorTwoDPlusThree}{black}%
\newcommand\ColorExactlyTwoD[1]{\textcolor{ColorExactlyTwoD}{\textbf{\smallerspaces{#1}}}}%
\newcommand\ColorTwoDPlusOne[1]{\textcolor{ColorTwoDPlusOne}{\textit{\smallerspaces{#1}}}}%
\newcommand\ColorTwoDPlusTwo[1]{\textcolor{ColorTwoDPlusTwo}{\textit{\smallerspaces{#1}}}}%
\newcommand\ColorTwoDPlusThree[1]{\textcolor{ColorTwoDPlusThree}{\textit{\smallerspaces{#1}}}}%
For every \bubble{$\bm w$} $\bm u$ in \cref{tab:positive coordinates},
\begin{enumerate}[label=\textup{(\alph*)}, ref=\alph*]
\item if $d_{\bm w} \mid \bm w \cdot \bm u$, then $\bm u$ is written in \ColorExactlyTwoD{bold},
\item if $\bm w \cdot \bm u \geq 2 d_{\bm w} + 1$ and $\bm u \notin 2 V_{\bm w}^+(d_{\bm w})$, then $\bm u$ is written in \ColorTwoDPlusTwo{italic}.
\end{enumerate}
If $\bm u$ is written in bold or italic, then $\bm w \cdot \bm u \in \{2d_{\bm w},\, 2d_{\bm w}+1,\, 2d_{\bm w}+2,\, 2d_{\bm w}+3\}$.

\begin{longtable}{cc}
\caption{Bubbles with positive coordinates for weights less than \upperbound{}\label{tab:positive coordinates}}\\
\toprule
$\bm w$ & Bubbles\\*
\midrule
\endfirsthead
\toprule
$\bm w$ & Bubbles\\*
\midrule
\endhead
\bottomrule
\endlastfoot
$(6, 14, 21)$ & $\{\ColorTwoDPlusOne{(6, 2, 1)}\}$\\
$(6, 22, 33)$ & $\{(10, 2, 1)\}$\\
$(6, 26, 39)$ & $\{\ColorTwoDPlusOne{(11, 2, 1)},\, (12, 2, 1)\}$\\
$(10, 14, 35)$ & $\{\ColorTwoDPlusOne{(5, 4, 1)},\, (6, 4, 1)\}$\\
$(12, 15, 20)$ & $\{\ColorTwoDPlusOne{(3, 3, 2)},\, (4, 3, 2)\}$\\
$(12, 21, 28)$ & $\{(6, 2, 2),\, (5, 3, 2),\, (6, 3, 2)\}$\\
$(12, 26, 39)$ & $\{\ColorTwoDPlusOne{(12, 2, 3)},\, \ColorTwoDPlusOne{(12, 5, 1)}\}$\\
$(14, 20, 35)$ & $\{\ColorTwoDPlusOne{(4, 6, 3)},\, \ColorTwoDPlusOne{(9, 6, 1)}\}$\\
$(15, 20, 24)$ & $\{\ColorTwoDPlusOne{(3, 5, 4)},\, \ColorTwoDPlusOne{(7, 2, 4)}\}$\\
$(15, 21, 35)$ & $\{(4, 4, 2),\, (6, 3, 2),\, (5, 4, 2),\, (6, 4, 2)\}$\\
$(15, 24, 40)$ & $\mleft\{\begin{aligned}
  & \ColorTwoDPlusOne{(7, 4, 1)},\, \ColorTwoDPlusTwo{(6, 3, 2)},\, (5, 4, 2),\\[-1\baselineskip]
  & (7, 3, 2),\, (6, 4, 2),\, (7, 4, 2)
\end{aligned}\mright\}$\\
$(20, 28, 35)$ & $\mleft\{\begin{aligned}
  & \ColorTwoDPlusOne{(6, 2, 3)},\, \ColorTwoDPlusTwo{(5, 4, 2)},\, (5, 3, 3),\, (4, 4, 3),\\[-1\baselineskip]
  & (6, 4, 2),\, (6, 3, 3),\, (5, 4, 3),\, (6, 4, 3)
\end{aligned}\mright\}$\\
$(21, 24, 28)$ & $\{(3, 6, 5),\, (7, 6, 2)\}$\\
$(21, 30, 35)$ & $\{(4, 6, 5),\, (9, 6, 2)\}$\\
$(24, 30, 40)$ & $\{\ColorTwoDPlusTwo{(3, 3, 2)},\, (4, 3, 2)\}$\\
$(28, 35, 40)$ & $\mleft\{\begin{aligned}
  & (4, 6, 6),\, \ColorTwoDPlusTwo{(9, 2, 6)},\, (3, 7, 6),\\[-1\baselineskip]
  & (8, 3, 6),\, (4, 7, 6),\, (9, 3, 6)
\end{aligned}\mright\}$\\
\midrule
$(1, 6, 10, 15)$ & $\{\ColorExactlyTwoD{(1, 4, 2, 1)}\}$\\
$(1, 6, 22, 33)$ & $\{\ColorExactlyTwoD{(1, 9, 2, 1)},\, (1, 10, 2, 1)\}$\\
$(1, 12, 20, 30)$ & $\{\ColorExactlyTwoD{(2, 4, 2, 1)},\, \ColorTwoDPlusOne{(3, 4, 2, 1)}\}$\\
$(1, 12, 21, 28)$ & $\{\ColorExactlyTwoD{(1, 4, 3, 2)},\, (1, 6, 2, 2),\, (1, 5, 3, 2),\, (1, 6, 3, 2)\}$\\
$(1, 12, 22, 33)$ & $\{\ColorExactlyTwoD{(1, 10, 2, 3)},\, \ColorExactlyTwoD{(1, 10, 5, 1)}\}$\\
$(1, 15, 21, 35)$ & $\mleft\{\begin{gathered}
  \ColorExactlyTwoD{(1, 6, 4, 1)},\, \ColorExactlyTwoD{(2, 5, 3, 2)},\, \ColorTwoDPlusOne{(2, 6, 4, 1)},\, (1, 4, 4, 2),\\[-1\baselineskip]
  (2, 4, 4, 2),\, (1, 6, 3, 2),\, (2, 6, 3, 2),\, (1, 5, 4, 2),\\[-1\baselineskip]
  (2, 5, 4, 2),\, (1, 6, 4, 2),\, (2, 6, 4, 2)
\end{gathered}\mright\}$\\
$(1, 15, 24, 40)$ & $\{\ColorTwoDPlusTwo{(1, 7, 4, 1)}\}$\\
$(1, 20, 28, 35)$ & $\{\ColorTwoDPlusTwo{(1, 6, 2, 3)}\}$\\
$(1, 21, 28, 36)$ & $\{\ColorExactlyTwoD{(1, 3, 8, 6)},\, \ColorExactlyTwoD{(1, 7, 5, 6)},\, \ColorExactlyTwoD{(1, 11, 2, 6)}\}$\\
$(1, 21, 30, 35)$ & $\{\ColorExactlyTwoD{(2, 3, 6, 5)}\}$\\
$(1, 24, 30, 40)$ & $\{\ColorTwoDPlusThree{(1, 3, 3, 2)},\, (1, 4, 3, 2)\}$\\
$(2, 12, 15, 20)$ & $\{\ColorExactlyTwoD{(1, 4, 2, 2)},\, (1, 3, 3, 2),\, (1, 4, 3, 2)\}$\\
$(2, 12, 20, 30)$ & $\{\ColorExactlyTwoD{(1, 4, 2, 1)}\}$\\
$(2, 21, 30, 35)$ & $\{\ColorExactlyTwoD{(1, 3, 6, 5)},\, \ColorExactlyTwoD{(1, 8, 6, 2)},\, (1, 4, 6, 5),\, (1, 9, 6, 2)\}$\\
$(3, 12, 15, 20)$ & $\{\ColorTwoDPlusOne{(1, 4, 2, 2)}\}$\\
$(3, 12, 20, 30)$ & $\{\ColorTwoDPlusOne{(1, 4, 2, 1)}\}$\\
$(3, 12, 21, 28)$ & $\{(1, 4, 3, 2)\}$\\
$(3, 24, 30, 40)$ & $\{(2, 4, 2, 2),\, (1, 3, 3, 2),\, (3, 4, 2, 2),\, (1, 4, 3, 2)\}$\\
$(4, 15, 24, 40)$ & $\mleft\{\begin{aligned}
  & \ColorExactlyTwoD{(1, 4, 4, 2)},\, \ColorTwoDPlusOne{(2, 7, 2, 2)},\, (1, 7, 4, 1),\, (1, 6, 3, 2),\\[-1\baselineskip]
  & (1, 5, 4, 2),\, (1, 7, 3, 2),\, (1, 6, 4, 2),\, (1, 7, 4, 2)
\end{aligned}\mright\}$\\
$(4, 20, 28, 35)$ & $\{\ColorTwoDPlusOne{(1, 3, 4, 3)}\}$\\
$(4, 24, 30, 40)$ & $\{\ColorExactlyTwoD{(1, 4, 2, 2)},\, (1, 3, 3, 2),\, (1, 4, 3, 2)\}$\\
$(5, 12, 20, 30)$ & $\{(1, 4, 2, 1)\}$\\
$(5, 15, 21, 35)$ & $\{(1, 6, 4, 1)\}$\\
$(5, 15, 24, 40)$ & $\{\ColorTwoDPlusOne{(1, 4, 4, 2)}\}$\\
$(5, 20, 28, 35)$ & $\{\ColorTwoDPlusTwo{(1, 3, 4, 3)},\, (1, 4, 4, 3)\}$\\
$(5, 24, 30, 40)$ & $\{\ColorTwoDPlusOne{(1, 4, 2, 2)},\, \ColorTwoDPlusOne{(3, 4, 3, 1)},\, (1, 3, 3, 2),\, (1, 4, 3, 2)\}$\\
$(5, 28, 35, 40)$ & $\{(1, 4, 7, 5),\, (1, 9, 3, 5)\}$\\
$(6, 15, 24, 40)$ & $\{\ColorTwoDPlusTwo{(1, 4, 4, 2)},\, (1, 5, 4, 2)\}$\\
$(6, 21, 30, 35)$ & $\{(1, 3, 6, 5)\}$\\
$(6, 24, 30, 40)$ & $\{\ColorTwoDPlusTwo{(1, 4, 2, 2)}\}$\\
$(7, 12, 21, 28)$ & $\{(1, 6, 3, 1)\}$\\
$(7, 20, 28, 35)$ & $\{\ColorTwoDPlusOne{(1, 6, 3, 2)},\, \ColorTwoDPlusOne{(2, 6, 4, 1)},\, (1, 5, 4, 2),\, (1, 6, 4, 2)\}$\\
$(7, 28, 35, 40)$ & $\{(1, 4, 6, 6),\, (1, 9, 2, 6)\}$\\
$(8, 15, 24, 40)$ & $\{\ColorTwoDPlusOne{(1, 7, 2, 2)}\}$\\
$(8, 28, 35, 40)$ & $\{(1, 3, 7, 6)\}$\\
$(10, 15, 24, 40)$ & $\{(1, 7, 4, 1)\}$\\
$(12, 15, 20, 30)$ & $\{\ColorTwoDPlusOne{(3, 1, 2, 1)},\, (4, 1, 2, 1)\}$\\
$(14, 20, 28, 35)$ & $\mleft\{\begin{gathered}
  \ColorTwoDPlusOne{(1, 6, 4, 1)},\, \ColorTwoDPlusOne{(2, 6, 1, 3)},\, \ColorTwoDPlusOne{(3, 6, 3, 1)},\\[-1\baselineskip]
  \ColorTwoDPlusOne{(5, 6, 2, 1)},\, \ColorTwoDPlusOne{(7, 6, 1, 1)}
\end{gathered}\mright\}$\\
$(15, 20, 24, 30)$ & $\{\ColorTwoDPlusOne{(1, 2, 4, 3)},\, \ColorTwoDPlusOne{(1, 5, 4, 1)},\, \ColorTwoDPlusOne{(3, 2, 4, 2)},\, \ColorTwoDPlusOne{(5, 2, 4, 1)}\}$\\
$(15, 20, 24, 40)$ & $\{\ColorTwoDPlusOne{(3, 1, 4, 2)},\, \ColorTwoDPlusOne{(3, 3, 4, 1)}\}$\\
$(15, 24, 30, 40)$ & $\mleft\{\begin{gathered}
  \ColorTwoDPlusOne{(1, 4, 3, 1)},\, \ColorTwoDPlusOne{(3, 4, 2, 1)},\, \ColorTwoDPlusOne{(5, 4, 1, 1)},\, \ColorTwoDPlusTwo{(2, 3, 2, 2)},\\[-1\baselineskip]
  \ColorTwoDPlusTwo{(4, 3, 1, 2)},\, (1, 4, 2, 2),\, (3, 4, 1, 2),\, (1, 3, 3, 2),\\[-1\baselineskip]
  (3, 3, 2, 2),\, (5, 3, 1, 2),\, (2, 4, 2, 2),\, (4, 4, 1, 2),\\[-1\baselineskip]
  (1, 4, 3, 2),\, (3, 4, 2, 2),\, (5, 4, 1, 2)
\end{gathered}\mright\}$\\
\midrule
$(1, 2, 12, 20, 30)$ & $\{\ColorTwoDPlusOne{(1, 1, 4, 2, 1)}\}$\\
$(1, 3, 24, 30, 40)$ & $\mleft\{\begin{aligned}
  & \ColorExactlyTwoD{(1, 1, 4, 2, 2)},\, (1, 2, 4, 2, 2),\\[-1\baselineskip]
  & (1, 1, 3, 3, 2),\, (1, 1, 4, 3, 2)
\end{aligned}\mright\}$\\
$(1, 4, 12, 21, 28)$ & $\{\ColorExactlyTwoD{(1, 1, 6, 3, 1)}\}$\\
$(1, 4, 24, 30, 40)$ & $\{\ColorTwoDPlusOne{(1, 1, 4, 2, 2)},\, (1, 1, 3, 3, 2),\, (1, 1, 4, 3, 2)\}$\\
$(1, 5, 24, 30, 40)$ & $\{\ColorTwoDPlusTwo{(1, 1, 4, 2, 2)}\}$\\
$(1, 6, 24, 30, 40)$ & $\{\ColorTwoDPlusThree{(1, 1, 4, 2, 2)}\}$\\
$(1, 7, 15, 21, 35)$ & $\{\ColorExactlyTwoD{(1, 1, 6, 2, 2)}\}$\\
$(1, 15, 24, 30, 40)$ & $\{\ColorTwoDPlusTwo{(1, 1, 4, 3, 1)},\, \ColorTwoDPlusTwo{(1, 3, 4, 2, 1)},\, \ColorTwoDPlusTwo{(1, 5, 4, 1, 1)}\}$\\
$(2, 3, 12, 20, 30)$ & $\{(1, 1, 4, 2, 1)\}$\\
$(2, 5, 12, 15, 20)$ & $\{\ColorExactlyTwoD{(1, 1, 4, 3, 1)}\}$\\
$(2, 5, 12, 20, 30)$ & $\{(1, 1, 4, 2, 1)\}$\\
$(2, 12, 15, 20, 30)$ & $\{(1, 3, 1, 2, 1),\, (1, 4, 1, 2, 1)\}$\\
$(2, 14, 21, 30, 35)$ & $\{\ColorExactlyTwoD{(1, 1, 4, 6, 4)},\, \ColorExactlyTwoD{(1, 1, 9, 6, 1)}\}$\\
$(3, 4, 15, 24, 40)$ & $\{\ColorExactlyTwoD{(1, 1, 7, 2, 2)}\}$\\
$(3, 4, 24, 30, 40)$ & $\{(1, 1, 4, 2, 2),\, (1, 1, 3, 3, 2),\, (1, 1, 4, 3, 2)\}$\\
$(3, 6, 24, 30, 40)$ & $\{(1, 1, 4, 2, 2)\}$\\
$(4, 5, 24, 30, 40)$ & $\mleft\{\begin{gathered}
  \ColorExactlyTwoD{(1, 2, 4, 3, 1)},\, (1, 1, 4, 2, 2),\, (1, 3, 4, 3, 1),\\[-1\baselineskip]
  (1, 1, 3, 3, 2),\, (1, 1, 4, 3, 2)
\end{gathered}\mright\}$\\
$(4, 10, 15, 24, 40)$ & $\{\ColorExactlyTwoD{(1, 1, 6, 4, 1)},\, (1, 1, 7, 4, 1)\}$\\
$(4, 10, 24, 30, 40)$ & $\{\ColorExactlyTwoD{(1, 1, 4, 3, 1)}\}$\\
$(4, 15, 24, 30, 40)$ & $\mleft\{\begin{gathered}
  \ColorExactlyTwoD{(1, 2, 4, 1, 2)},\, \ColorTwoDPlusOne{(2, 1, 2, 3, 2)},\, \ColorTwoDPlusOne{(2, 3, 2, 2, 2)},\\[-1\baselineskip]
  \ColorTwoDPlusOne{(2, 5, 2, 1, 2)},\, (1, 1, 4, 3, 1),\, (1, 3, 4, 2, 1),\\[-1\baselineskip]
  (1, 5, 4, 1, 1),\, (1, 2, 3, 2, 2),\, (1, 4, 3, 1, 2),\\[-1\baselineskip]
  (1, 1, 4, 2, 2),\, (1, 3, 4, 1, 2),\, (1, 1, 3, 3, 2),\\[-1\baselineskip]
  (1, 3, 3, 2, 2),\, (1, 5, 3, 1, 2),\, (1, 2, 4, 2, 2),\\[-1\baselineskip]
  (1, 4, 4, 1, 2),\, (1, 1, 4, 3, 2),\, (1, 3, 4, 2, 2),\\[-1\baselineskip]
  (1, 5, 4, 1, 2)
\end{gathered}\mright\}$\\
$(5, 6, 24, 30, 40)$ & $\{(1, 1, 4, 2, 2)\}$\\
$(5, 10, 24, 30, 40)$ & $\{\ColorTwoDPlusOne{(1, 1, 4, 3, 1)}\}$\\
$(5, 15, 24, 30, 40)$ & $\{\ColorTwoDPlusOne{(1, 2, 4, 1, 2)}\}$\\
$(6, 15, 24, 30, 40)$ & $\{\ColorTwoDPlusTwo{(1, 2, 4, 1, 2)},\, (1, 1, 4, 2, 2),\, (1, 3, 4, 1, 2)\}$\\
$(8, 15, 24, 30, 40)$ & $\{\ColorTwoDPlusOne{(1, 1, 2, 3, 2)},\, \ColorTwoDPlusOne{(1, 3, 2, 2, 2)},\, \ColorTwoDPlusOne{(1, 5, 2, 1, 2)}\}$\\
$(10, 15, 24, 30, 40)$ & $\{(1, 1, 4, 3, 1),\, (1, 3, 4, 2, 1),\, (1, 5, 4, 1, 1)\}$\\
$(15, 20, 24, 30, 40)$ & $\{\ColorTwoDPlusOne{(1, 1, 4, 1, 2)},\, \ColorTwoDPlusOne{(1, 3, 4, 1, 1)}\}$\\
\midrule
$(1, 3, 10, 24, 30, 40)$ & $\{\ColorExactlyTwoD{(1, 1, 1, 4, 3, 1)}\}$\\
$(1, 4, 10, 24, 30, 40)$ & $\{\ColorTwoDPlusOne{(1, 1, 1, 4, 3, 1)}\}$\\
$(2, 5, 12, 15, 20, 30)$ & $\{\ColorExactlyTwoD{(1, 1, 4, 1, 1, 1)}\}$\\
$(3, 4, 10, 24, 30, 40)$ & $\{(1, 1, 1, 4, 3, 1)\}$\\
$(3, 4, 15, 24, 30, 40)$ & $\mleft\{\begin{gathered}
  \ColorExactlyTwoD{(1, 1, 1, 2, 3, 2)},\, \ColorExactlyTwoD{(1, 1, 3, 2, 2, 2)},\\[-1\baselineskip]
  \ColorExactlyTwoD{(1, 1, 5, 2, 1, 2)}
\end{gathered}\mright\}$\\
$(4, 5, 10, 24, 30, 40)$ & $\{(1, 1, 1, 4, 3, 1)\}$\\
$(4, 10, 15, 24, 30, 40)$ & $\mleft\{\begin{gathered}
  \ColorExactlyTwoD{(1, 1, 2, 4, 2, 1)},\, \ColorExactlyTwoD{(1, 1, 4, 4, 1, 1)},\\[-1\baselineskip]
  (1, 1, 1, 4, 3, 1),\, (1, 1, 3, 4, 2, 1),\\[-1\baselineskip]
  (1, 1, 5, 4, 1, 1)
\end{gathered}\mright\}$\\
\end{longtable}

\subsection*{Acknowledgments}

We would like to thank Simon Brandhorst, Qing Liu and Roy Magen for useful comments.

\providecommand{\bysame}{\leavevmode\hbox to3em{\hrulefill}\thinspace}
\providecommand{\MR}{\relax\ifhmode\unskip\space\fi MR }
\providecommand{\MRhref}[2]{%
  \href{http://www.ams.org/mathscinet-getitem?mr=#1}{#2}
}
\providecommand{\href}[2]{#2}

\vspace{0.5\baselineskip}
\ShowAffiliations{\\[1\baselineskip]}%

\begin{thebibliography}{{Sta}25}

\bibitem[ABL25]{ABL25}
Praise Adeyemo, Dominic Bunnett, and Fabián Levicán, \emph{Embeddings of
  weighted projective spaces}, Preprint, 2025,
  \href{https://arxiv.org/abs/2510.05076}{\textsf{arXiv:2510.05076}}.

\bibitem[Arm68]{Arm68}
M.~A. Armstrong, \emph{The fundamental group of the orbit space of a
  discontinuous group}, Proc. Camb. Philos. Soc. \textbf{64} (1968), 299--301
  (English).

\bibitem[AT14]{AT14}
Marco Andreatta and Luca Tasin, \emph{Fano-mori contractions of high length on
  projective varieties with terminal singularities}, Bull. Lond. Math. Soc.
  \textbf{46} (2014), no.~1, 185--196, \textsc{doi}:\hskip
  .16667em\href{https://doi.org/10.1112/blms/bdt083}{\textsf{10.1112/blms/bdt083}}.

\bibitem[BG09]{BG09}
Winfried Bruns and Joseph Gubeladze, \emph{Polytopes, rings, and {K}-theory},
  Springer Monogr. Math., New York, NY: Springer, 2009 (English),
  \textsc{doi}:\hskip
  .16667em\href{https://doi.org/10.1007/b105283}{\textsf{10.1007/b105283}}.

\bibitem[BGT97]{bgt97}
Winfried Bruns, Joseph Gubeladze, and Ng\^o{}~Vi\^et Trung, \emph{Normal
  polytopes, triangulations, and {K}oszul algebras}, J. Reine Angew. Math.
  \textbf{485} (1997), 123--160, \url{https://eudml.org/doc/183538}.

\bibitem[BR86]{BR86}
Mauro Beltrametti and Lorenzo Robbiano, \emph{Introduction to the theory of
  weighted projective spaces}, Exposition. Math. \textbf{4} (1986), no.~2,
  111--162.

\bibitem[CCC11]{CCC11}
Jheng-Jie Chen, Jungkai~A. Chen, and Meng Chen, \emph{On quasismooth weighted
  complete intersections}, J. Algebraic Geom. \textbf{20} (2011), no.~2,
  239--262, \textsc{doi}:\hskip
  .16667em\href{https://doi.org/10.1090/S1056-3911-10-00542-4}{\textsf{10.1090/S1056-3911-10-00542-4}}.

\bibitem[Del75]{Del75}
Charles Delorme, \emph{Espaces projectifs anisotropes}, Bull. Soc. Math. Fr.
  \textbf{103} (1975), 203--223 (French), \textsc{doi}:\hskip
  .16667em\href{https://doi.org/10.24033/bsmf.1802}{\textsf{10.24033/bsmf.1802}}.

\bibitem[DLT26]{DLT26}
Kristin DeVleming, Jennifer Li, and Sebastián Torres, \emph{Weighted
  projective degenerations of {$\mathbb{P}^n$}}, Preprint, 2026,
  \href{https://arxiv.org/abs/2606.21660}{\textsf{arXiv:2606.21660}}.

\bibitem[EW91]{EW91}
G\"unter Ewald and Uwe Wessels, \emph{On the ampleness of invertible sheaves in
  complete projective toric varieties}, Results Math. \textbf{19} (1991),
  no.~3-4, 275--278, \textsc{doi}:\hskip
  .16667em\href{https://doi.org/10.1007/BF03323286}{\textsf{10.1007/BF03323286}}.

\bibitem[Gro61]{EGAII}
A.~Grothendieck, \emph{{\'E}l\'ements de g\'eom\'etrie alg\'ebrique. {II}.
  {Étude} globale \'el\'ementaire de quelques classes de morphismes}, Inst.
  Hautes \'Etudes Sci. Publ. Math. (1961), no.~8, 222, \textsc{doi}:\hskip
  .16667em\href{https://doi.org/10.1007/bf02699291}{\textsf{10.1007/bf02699291}}.

\bibitem[GW20]{GW20}
Ulrich Görtz and Torsten Wedhorn, \emph{Algebraic {G}eometry {I}: {S}chemes.
  {W}ith {E}xamples and {E}xercises}, second ed., Springer Studium
  Mathematik---Master, Springer Spektrum, Wiesbaden, 2020, \textsc{doi}:\hskip
  .16667em\href{https://doi.org/10.1007/978-3-658-30733-2}{\textsf{10.1007/978-3-658-30733-2}}.

\bibitem[Har77]{Har77}
Robin Hartshorne, \emph{Algebraic geometry}, Graduate Texts in Mathematics,
  vol. No. 52, Springer-Verlag, New York-Heidelberg, 1977, \textsc{doi}:\hskip
  .16667em\href{https://doi.org/10.1007/978-1-4757-3849-0}{\textsf{10.1007/978-1-4757-3849-0}}.

\bibitem[HL85]{HL85}
Helmut~A. Hamm and D{\~u}ng~Tr{\'a}ng L{\^e}, \emph{Lefschetz theorems on
  quasi-projective varieties}, Bull. Soc. Math. Fr. \textbf{113} (1985),
  123--142, \textsc{doi}:\hskip
  .16667em\href{https://doi.org/10.24033/bsmf.2023}{\textsf{10.24033/bsmf.2023}}.

\bibitem[HL20]{HL20}
\bysame, \emph{The {Lefschetz} theorem for hyperplane sections}, Handbook of
  geometry and topology of singularities I, Cham: Springer, 2020,
  \textsc{doi}:\hskip
  .16667em\href{https://doi.org/10.1007/978-3-030-53061-7_9}{\textsf{10.1007/978-3-030-53061-7\_9}},
  pp.~491--540.

\bibitem[IF00]{IF00}
A.~R. Iano-Fletcher, \emph{Working with weighted complete intersections},
  Explicit birational geometry of 3-folds, London Math. Soc. Lecture Note Ser.,
  vol. 281, Cambridge Univ. Press, Cambridge, 2000, \textsc{doi}:\hskip
  .16667em\href{https://doi.org/10.1017/CBO9780511758942.005}{\textsf{10.1017/CBO9780511758942.005}},
  pp.~101--173.

\bibitem[Kaw03]{Kaw03}
Masayuki Kawakita, \emph{General elephants of three-fold divisorial
  contractions}, J. Amer. Math. Soc. \textbf{16} (2003), no.~2, 331--362,
  \textsc{doi}:\hskip
  .16667em\href{https://doi.org/10.1090/S0894-0347-02-00416-2}{\textsf{10.1090/S0894-0347-02-00416-2}},
  see
  \url{https://www.kurims.kyoto-u.ac.jp/~masayuki/Website/Documents/erratum-03.pdf}
  for erratum.

\bibitem[KM92]{KM92}
J\'{a}nos Koll\'{a}r and Shigefumi Mori, \emph{Classification of
  three-dimensional flips}, J. Amer. Math. Soc. \textbf{5} (1992), no.~3,
  533--703, \textsc{doi}:\hskip
  .16667em\href{https://doi.org/10.2307/2152704}{\textsf{10.2307/2152704}}.

\bibitem[Oka23]{Oka23}
Takuzo Okada, \emph{Birationally solid {F}ano 3-fold hypersurfaces}, arXiv
  eprint, 2023,
  \href{https://arxiv.org/abs/2305.06183v3}{\textsf{arXiv:2305.06183v3}}.

\bibitem[Pet15]{Pet15}
T.~Kyle Petersen, \emph{Eulerian numbers}, Birkh{\"a}user Adv. Texts, Basler
  Lehrb{\"u}ch., New York, NY: Birkh{\"a}user/Springer, 2015,
  \textsc{doi}:\hskip
  .16667em\href{https://doi.org/10.1007/978-1-4939-3091-3}{\textsf{10.1007/978-1-4939-3091-3}}.

\bibitem[PS26]{PS26}
Victor Przyjalkowski and Constantin Shramov, \emph{Weighted complete
  intersections}, first ed., De Gruyter Expositions in Mathematics, Walter de
  Gruyter, 2026, \textsc{doi}:\hskip
  .16667em\href{https://doi.org/10.1515/9783110719048}{\textsf{10.1515/9783110719048}}.

\bibitem[Rei80]{Rei80}
Miles Reid, \emph{Canonical {$3$}-folds}, Journ\'ees de {G}\'eometrie
  {A}lg\'ebrique d'{A}ngers, {J}uillet 1979/{A}lgebraic {G}eometry, {A}ngers,
  1979, Sijthoff \& Noordhoff, Alphen aan den Rijn---Germantown, Md., 1980,
  pp.~273--310.

\bibitem[{Sta}25]{Stacks}
The {Stacks Project Authors}, \emph{\textit{Stacks Project}},
  \url{https://stacks.math.columbia.edu}, 2025.

\bibitem[T{\'o}t10]{Tót10}
L{\'a}szl{\'o} T{\'o}th, \emph{A survey of gcd-sum functions}, J. Integer Seq.
  \textbf{13} (2010), no.~8, 23 (English), Id/No 10.8.1.

\bibitem[Yam18]{Yam18}
Yuki Yamamoto, \emph{Divisorial contractions to {$cDV$} points with discrepancy
  greater than 1}, Kyoto J. Math. \textbf{58} (2018), no.~3, 529--567,
  \textsc{doi}:\hskip
  .16667em\href{https://doi.org/10.1215/21562261-2017-0028}{\textsf{10.1215/21562261-2017-0028}}.

\end{thebibliography}
\end{document}